\documentclass[intlim,righttag,10pt]{amsart}
\usepackage{stmaryrd}
\usepackage{amscd}
\usepackage{amssymb}
\usepackage[all]{xy}
\RequirePackage{mathrsfs}
\def\jcdot{{\scriptscriptstyle\bullet}}

\def\invlim{\mathop{\vtop{\ialign{##\crcr$\hfill{\lim}\hfil$\crcr
\noalign{\kern1pt\nointerlineskip}\leftarrowfill\crcr\noalign
{\kern -3pt}}}}\limits}
\def\dirlim{\mathop{\vtop{\ialign{##\crcr$\hfill{\lim}\hfil$\crcr
\noalign{\kern1pt\nointerlineskip}\rightarrowfill\crcr\noalign
{\kern -3pt}}}}\limits}
\def\lomapr#1{\smash{\mathop{\relbar\joinrel\longrightarrow}\limits^{#1}}}
 \def\verylomapr#1{\smash{\mathop{\relbar\joinrel\relbar\joinrel\relbar\joinrel\longrightarrow}\limits^{#1}}}

\def\phi{\varphi}
\def\epsilon{\varepsilon}
\let\mathcal\mathscr
\usepackage{fge}
\renewcommand{\setminus}{\mathbin{\fgebackslash}}

\newtheorem{theorem}[equation]{Theorem}
 \newtheorem{lemma}[equation]{Lemma}
 \newtheorem{proposition}[equation]{Proposition}
 \newtheorem{corollary}[equation]{Corollary}

\newtheorem{fact}[equation]{Fact}
\theoremstyle{definition}

\theoremstyle{remark}
\newtheorem{remark}[equation]{Remark}

\newtheorem{example}[equation]{Example}
\newtheorem*{acknowledgments}{Acknowledgments}

\def\cdot{{\scriptscriptstyle\bullet}}

\renewcommand{\phi}{\varphi}

\newcommand{\R}{\mathrm {R} }
\newcommand{\F}{\mathrm {F} }

\newcommand{\LL}{\mathrm {L} }

 \newcommand{\qp}{{\mathbf Q}_{p}}
\newcommand{\ovk}{\overline{K} }

\newcommand{\gp}{\operatorname{gp} }

 \newcommand{\coker}{\operatorname{coker} }
 \newcommand{\coim}{\operatorname{coim} }

 \newcommand{\holim}{\operatorname{holim} }
  \newcommand{\rig}{\operatorname{rig} }
 \newcommand{\hocolim}{\operatorname{hocolim} }

 \newcommand{\cont}{\operatorname{cont}  } 
 
  \newcommand{\proeet}{\operatorname{pro\acute{e}t}  }

 \newcommand{\eet}{\operatorname{\acute{e}t} }

 \newcommand{\dlog}{\operatorname{dlog} }
 
 \newcommand{\an}{\operatorname{an} }

  \newcommand{\Res}{\operatorname{Res} }
   \newcommand{\conv}{\operatorname{conv} }
 \newcommand{\Spec}{\operatorname{Spec} } 
  \newcommand{\Sp}{\operatorname{Sp} }  
   \newcommand{\Spwf}{\operatorname{Spwf} }

 \newcommand{\Spf}{\operatorname{Spf} }

 \newcommand{\Ext}{\operatorname{Ext} }

 \newcommand{\Gal}{\operatorname{Gal} }
 \newcommand{\tr}{ \operatorname{tr} }
 \newcommand{\can}{ \operatorname{can} }

 \newcommand{\id}{ \operatorname{Id} }
\newcommand{\synt}{ \operatorname{syn} }
 
  \newcommand{\ddim}{ \operatorname{Dim} }

\newcommand{\st}{\operatorname{st} }

\newcommand{\hk}{\operatorname{HK} } 
\newcommand{\dr}{\operatorname{dR} }

 \newcommand{\crr}{\operatorname{cr} }
  \newcommand{\Pro}{\operatorname{Pro} }
 \newcommand{\gr}{\operatorname{gr} }
 \newcommand{\im}{\operatorname{im} }

 \newcommand{\ve}{ \varepsilon  }
 \newcommand{\kr}{^{\scriptscriptstyle\bullet}}

 \newcommand{\sff}{{\mathcal{F}}}
 
 \newcommand{\sy}{{\mathcal{Y}}}
 \newcommand{\sh}{{\mathcal{H}}}
 \newcommand{\sg}{{\mathcal{G}}}
 
 \newcommand{\sv}{{\mathcal{V}}} 
 
  \newcommand{\su}{{\mathcal{U}}}
 \newcommand{\sbb}{{\mathcal{B}}}
 \newcommand{\scc}{{\mathcal{C}}}
 
 \newcommand{\sll}{{\mathcal{L}}}
 
 \newcommand{\so}{{\mathcal O}}
 \newcommand{\sj}{{\mathcal J}}
 
 \newcommand{\sa}{{\mathcal{A}}}
 
 \newcommand{\sx}{{\mathcal{X}}}
 \newcommand{\sss}{{\mathcal{S}}}
\newcommand{\sd}{{\mathcal{D}}}

\newcommand{\Ind}{\operatorname{Ind}}
 
 \newcommand{\wt}{\widetilde}
 \newcommand{\wh}{\widehat}

 \newcommand{\Z}{ {\mathbf Z} }
    
   \newcommand{\Q}{ {\mathbf Q}}
   \newcommand{\N}{{\mathbf N}}
         \newcommand{\rg}{\R\Gamma}
   \def\B{{\bf B}}
      \def\A{{\bf A}}

\numberwithin{equation}{section}

\begin{document}
 \title[Cohomology of $p$-adic Stein spaces]{Cohomology of $p$-adic Stein spaces}
 \date{\today}

\author{Pierre Colmez}
\address{CNRS, IMJ-PRG, Sorbonne Universit\'e, 4 place Jussieu,
75005 Paris, France}
\email{pierre.colmez@imj-prg.fr}
\author{Gabriel Dospinescu}
\address{CNRS, UMPA, \'Ecole Normale Sup\'erieure de Lyon, 46 all\'ee d'Italie, 69007 Lyon, France}
\email{gabriel.dospinescu@ens-lyon.fr}
\author{Wies{\l}awa Nizio{\l}}
\address{CNRS, UMPA, \'Ecole Normale Sup\'erieure de Lyon, 46 all\'ee d'Italie, 69007 Lyon, France}
\email{wieslawa.niziol@ens-lyon.fr}

 \thanks{This work  was partially supported by the  project  ANR-14-CE25 as well as  by the grant 346300 for IMPAN from the Simons Foundation and the matching 2015-2019 Polish MNiSW fund.}
\begin{abstract}
We compute $p$-adic \'etale and  pro-\'etale cohomologies of Drinfeld half-spaces. In the pro-\'etale case, 
the main input is a  comparison theorem for 
$p$-adic Stein spaces; the cohomology
groups involved here are much bigger than in the case of \'etale cohomology of
algebraic varieties or proper analytic spaces considered in all previous works.
 In the \'etale case, the classical $p$-adic comparison theorems allow us to pass 
to a computation of integral differential forms cohomologies which can be done because the standard 
formal models of Drinfeld half-spaces are pro-ordinary and their differential forms are acyclic. 
\end{abstract}
 \setcounter{tocdepth}{2}

 \maketitle
 \tableofcontents
 \section{Introduction}Let $p$ be a prime. 
 Let  $\so_K$ be a complete discrete valuation ring of mixed characteristic $(0,p)$ with perfect residue field $k$ and fraction field $K$. Let $F$ be the fraction field of the ring of Witt vectors
 $\so_F=W(k)$ of $k$. Let $\ovk$ be an algebraic closure of $K$, $C=\wh{\ovk}$  its $p$-adic completion and  $\sg_K=\Gal(\ovk/K)$.

\subsection{The $p$-adic \'etale cohomology of Drinfeld half-spaces}
 This paper reports on some results of  our research project that aims at understanding the $p$-adic (pro-)\'etale cohomology of $p$-adic symmetric spaces. The main question of interest being: does this cohomology  realize the hoped  for  $p$-adic local Langlands correspondence in analogy with the $\ell$-adic situation ?
 When we started this project we did not know what to expect 
and local computations were rather discouraging: geometric $p$-adic \'etale cohomology groups of
affinoids and their interiors are huge and not invariant by base change 
to a bigger complete algebraically closed field.
 However there was one computation done long ago by Drinfeld~\cite{Drinfeld} that stood out.
Let us recall it. 

   Assume that  $[K:\Q_p]<\infty$ and let ${\mathbb H}_K={\mathbb P}^1_K\setminus {\mathbb P}^1(K)$ be the Drinfeld half-plane, thought of as a rigid analytic space. It admits a natural action of $G:={\rm GL}_2(K)$. We set ${\mathbb H}_C:={\mathbb H}_{K,C}$.
  \begin{fact}{\rm (Drinfeld)}
If $\ell$ is a prime number {\rm (including $\ell=p$ !)},
  there exists a natural isomorphism of $G\times \sg_K$-representations
  $$
  H^1_{\eet}({\mathbb H}_C,\Q_\ell(1))\simeq {\rm Sp}^{\cont}({\Q_\ell})^*,
  $$
  where ${\rm Sp}^{\cont}({\Q_\ell}):=\scc({\mathbb P}^1(K),\Q_\ell)/\Q_\ell$ is the continuous Steinberg representation of $G$ with coefficients in $\Q_\ell$ equipped with a trivial action of $\sg_K$
  and $(-)^*$ denotes the weak topological dual.
  \end{fact}
 The proof is very simple: it uses Kummer theory and vanishing of the Picard groups (of the standard Stein covering of ${\mathbb H}_K$) \cite[]{FDP}, \cite[1.4]{CDN1}. 
This result was encouraging because it showed that the $p$-adic \'etale cohomology 
was maybe not as pathological as one could fear.

Drinfeld's result was generalized, for $\ell\neq p$, to higher dimensions by
Schneider-Stuhler~\cite{SS}.
   Let $d\geq 1$ and let ${\mathbb H}^d_K$  be the Drinfeld half-space~\cite{Drinfeld1} of dimension $d$, 
i.e., 
$${\mathbb H}^d_K:={\mathbb P}^d_K\setminus \bigcup_{H\in \sh}H,$$ 
where $\sh$ denotes the set of $K$-rational hyperplanes. We set $G:={\rm GL}_{d+1}(K)$. 
If $1\leq r\leq d$, and if $\ell$ is a prime number,
denote by ${\rm Sp}_r(\Q_\ell)$  and ${\rm Sp}^{\cont}_r(\Q_\ell)$ 
the generalized 
locally constant and continuous Steinberg $\Q_\ell$-representations of $G$ (see Section \ref{intro5}),
 respectively,  equipped with a trivial action of $\sg_K$.
\begin{theorem}{\rm (Schneider-Stuhler)}
\label{SSIn}Let $r\geq 0$
and let $\ell\neq p$. There are natural $G\times\sg_K$-equivariant isomorphisms 
 $$
 H^r_{\eet}({\mathbb H}^d_C,\Q_{\ell}(r))\simeq {\rm Sp}^{\cont}_r(\Q_{\ell})^*,\quad H^r_{\proeet}({\mathbb H}^d_C,\Q_{\ell}(r))\simeq {\rm Sp}_r(\Q_{\ell})^*.
 $$
\end{theorem}
The computations of Schneider-Stuhler work for any cohomology theory that satisfies certain axioms,
 the most important being the homotopy property with respect to the open unit ball,
which fails rather dramatically for the $p$-adic (pro-)\'etale cohomology since
the $p$-adic \'etale cohomology of the unit ball is huge.
Nevertheless, we prove the following result.
\begin{theorem}
 \label{intro2}Let $r\geq 0$.
  \begin{enumerate}
\item There is a natural isomorphism  of $G\times \sg_K$-locally convex topological vector spaces (over $\Q_p$).
$$H^r_{\eet}({\mathbb H}^d_C,\Q_p(r))\simeq  {\rm Sp}^{\cont}_r(\Q_p)^*.$$
These spaces  are weak duals of Banach spaces.
  \item 
 There is a strictly exact sequence of $G\times \sg_K$-Fr\'echet spaces
 $$
 \xymatrix{
 0\ar[r] & \Omega^{r-1}({\mathbb H}^d_C)/\ker d \ar[r] & H^r_{\proeet}({\mathbb H}^d_C,\Q_p(r))\ar[r]& {\rm Sp}_r(\Q_p)^*\ar[r] & 0.
} $$
\item The natural map $H^r_{\eet}({\mathbb H}^d_C,\Q_p(r))\to H^r_{\proeet}({\mathbb H}^d_C,\Q_p(r))$ identifies \'etale cohomology with the space of $G$-bounded vectors\footnote{ Recall
that a subset $X$ of a locally convex vector space over $\Q_p$ is called {\em bounded} if $p^nx_n\mapsto 0$ for all sequences $\{x_n\}, n\in\N,$ of elements of $X$. In the above, $x$ is called
a {\em $G$-bounded vector}
if its $G$-orbit is a bounded set.}  in the pro-\'etale cohomology. 
 \end{enumerate}
 \end{theorem}
Hence, the $p$-adic \'etale cohomology is given by the same dual of a Steinberg representation 
as its $\ell$-adic counterpart. 
However, the $p$-adic pro-\'etale cohomology 
is a nontrivial extension of the same dual of a Steinberg representation that describes  
its $\ell$-adic counterpart by a huge space.
\begin{remark}
In \cite{CDN1} we have generalized the above computation of Drinfeld in a different direction, namely, 
to the Drinfeld tower~\cite{Drinfeld1} in dimension $1$. We have shown that, if $K=\Q_p$,
 the $p$-adic local Langlands correspondence (see \cite{pLL}, \cite{CDP}) for de Rham Galois representations of dimension $2$ (of Hodge-Tate weights $0$ and $1$ and not trianguline)
 can be realized inside the $p$-adic \'etale cohomology of the Drinfeld tower (see \cite[Theorem 0.2]{CDN1} for a precise statement).  The two important cohomological inputs were:
 \begin{enumerate}
 \item  a $p$-adic comparison theorem that allows us to recover the $p$-adic pro-\'etale cohomology from the de Rham complex and the Hyodo-Kato cohomology; the latter being compared to the $\ell$-adic \'etale cohomology computed, in turn, by non-abelian Lubin-Tate theory,
 \item  the fact that the $p$-adic \'etale cohomology is equal to the space of $G$-bounded vectors in the $p$-adic pro-\'etale cohomology. 
\end{enumerate}
In contrast, here we obtain the third  part of Theorem \ref{intro2} only after proving the two previous parts. In fact, for a general rigid analytic variety, we do know that the natural map from $p$-adic \'etale cohomology to $p$-adic pro-\'etale cohomology does not have to be injective: this is the case, for example, for a unit open ball over a field that is not spherically complete.
\end{remark}

\begin{remark}
The proof of Theorem \ref{intro2} establishes a number of other isomorphisms (see Theorem~\ref{Steinberg})
refining results of~\cite{SS,IS,des}.
\end{remark}
\begin{remark}
{\rm (i)} For $r \geq d+1$, all spaces in Theorem \ref{intro2} are $0$.

{\rm (ii)} For $1\leq r\leq d$, 
the spaces on the left and on the right in the exact sequence in Theorem \ref{intro2} 
describing the pro-\'etale cohomology of ${\mathbb H}^d_C$, despite being huge spaces,
have some finiteness properties: they are both duals of admissible locally
analytic representations of $G$ (over $C$ on the left and $\Q_p$ on the right), of finite length
(on the left, this is due to Orlik and Strauch (\cite{Orl0} combined with~\cite{SO2})).
\end{remark}

\begin{remark}
For small Tate twists ($r\leq p-1)$, the Fontaine-Messing period map, which is an essential input in the proof of Theorem \ref{intro2},  is an isomorphism ``on the nose''. It is possible then that  our proof of Theorem \ref{intro2}, with a better control of the constants, 
could give the integral $p$-adic \'etale cohomology of the Drinfeld half-space for small Tate twists, that is, a 
topological isomorphism 
$$
H^r_{\eet}({\mathbb H}^d_C,{\mathbf F}_p(r))\simeq {\rm Sp}_r({\mathbf F}_p)^*.
$$
But, in fact, by combining the results of Chapter~6 
of this paper with the integral $p$-adic Hodge Theory of 
 Bhatt-Morrow-Scholze and \v{C}esnavi\v{c}ius-Koshikawa \cite{BMS1}, \cite{BMS2}, \cite{CK}
  one can prove such a result for all twists~\cite{CDN4}.
 \end{remark}
\subsection{A comparison theorem for $p$-adic pro-\'etale cohomology}
The proof of Theorem\,\ref{intro2} uses the result below, which is the
main theorem of this paper and generalizes the above mentioned comparison theorem 
to rigid analytic Stein spaces\footnote{ 
 Recall that  a rigid analytic space $Y$ is Stein if it has an admissible affinoid covering $Y=\cup_{i\in\N}U_i$ such that $U_i\Subset U_{i+1}$, i.e., the inclusion $U_i\subset U_{i+1}$ factors over the adic compactification of $U_i$. The key property we need is the acyclicity of cohomology of coherent sheaves.}  over $K$ with a semistable reduction. 
  Let the field $K$ be as stated at the beginning of the introduction. 
 \begin{theorem}
\label{intro1}Let $r\geq 0$. 
 Let $X$ be a semistable Stein weak formal scheme\footnote{See Section \ref{assumptions} for the definition.} over $\so_K$. There exists a commutative $\sg_K$-equivariant diagram of Fr\'echet spaces
 $$
 \xymatrix{
 0\ar[r] & \Omega^{r-1}({X_C})/\ker d \ar[r]\ar@{=}[d] & H^r_{\proeet}(X_C,\Q_p(r))\ar[r]\ar[d]^{\wt{\beta}} & (H^r_{\hk}(X_{{k}})\wh{\otimes}_{F}\B^+_{\st})^{N=0,\phi=p^r}\ar[r]\ar[d]^{\iota_{\hk}\otimes\theta} & 0\\
 0\ar[r] & \Omega^{r-1}({X_C})/\ker d \ar[r]^-{d} & \Omega^r({X_C})^{d=0}\ar[r]^-{\can} & H^r_{\dr}(X_C)\ar[r] & 0
 }
 $$
 The rows are strictly exact and the maps $\wt{\beta}$ and $\iota_{\hk}\otimes\theta$ are strict (and have closed images). Moreover, $$\ker(\wt{\beta})\simeq\ker(\iota_{\hk}\otimes\theta)\simeq (H^r_{\hk}(X_{{k}})\wh{\otimes}_{F}\B^+_{\st})^{N=0,\phi=p^{r-1}}.$$
 \end{theorem}
 Here $H^r_{\hk}(X_{{k}})$ is the overconvergent Hyodo-Kato cohomology of Grosse-Kl\"onne \cite{GK2},  
$$\iota_{\hk}:H^r_{\hk}(X_{{k}})\otimes_FK\stackrel{\sim}{\to} H^r_{\dr}(X_C)$$
is the Hyodo-Kato isomorphism, $\B^+_{\st}$ is the semistable ring of periods defined by Fontaine, and $\theta: \B^+_{\st}\to C$ is  Fontaine's projection. 
 \begin{example}
 In the case the Hyodo-Kato cohomology  vanishes we obtain a particularly simple formula. Take, for example, 
 the rigid affine space ${\mathbb A}^d_K$. For $r\geq 1$, we have  $H^r_{\dr}({\mathbb A}^d_K)=0$ and, by the Hyodo-Kato isomorphism, also  $H^r_{\hk}({\mathbb A}^d_K)=0$. Hence the above theorem   yields  
an isomorphism $$H^r_{\proeet}({\mathbb A}^d_C,\Q_p(r)) \stackrel{\sim}{\leftarrow}
\Omega^{r-1}({{\mathbb A}^d_C})/\ker d.$$ 

 This was our first proof of this fact but there is a more direct argument in \cite{CN2}. 
Another approach, using relative fundamental exact sequences in pro-\'etale topology and their pushforwards to \'etale topology, can be found in \cite{AC}. 
 \end{example}
\begin{remark}
{\rm (i)} We think of the above theorem as a one-way comparison theorem, i.e., the pro-\'etale cohomology $H^r_{\proeet}(X_{C},\Q_p(r))$ is the pullback of the diagram $$
(H^r_{\hk}(X_{k})\wh{\otimes}_F\B^+_{\st})^{N=0,\phi=p^r}\verylomapr{\iota_{\hk}\otimes\theta}H^r_{\dr}(X_K)\wh{\otimes}_K C \stackrel{\can}{\longleftarrow}\Omega^r({X_C})^{d=0}$$ built from the Hyodo-Kato cohomology and a piece of the de Rham complex. 

{\rm (ii)} 
When we started doing computations of pro-\'etale cohomology groups (for the affine line), 
we could not understand why
the $p$-adic pro-\'etale cohomology seemed to be so big
while the Hyodo-Kato cohomology was so small (actually $0$ in that case): this was against
what the proper case was teaching us.  
If $X$ is proper, $\Omega^{r-1}({X_K})/\ker d=0$ (since the Hodge-de Rham spectral sequence degenerates) and the upper line of the above diagram becomes
$$0\to H^r_{\proeet}(X_C,\Q_p(r))
\to (H^r_{\hk}(X_k)\wh{\otimes}_{F}\B^+_{\st})^{N=0,\phi=p^r}\to 
(H^r_{\dr}(X_K)\wh{\otimes}_K \B_{\dr}^+)/{\rm Fil}^r\to 0.$$
Hence the huge term on the left disappears, and an extra term
on the right shows up. This seemed to indicate that there was no real hope of computing
$p$-adic \'etale and pro-\'etale cohomologies of big spaces.  It was learning about Drinfeld's result
that convinced us to look further\footnote{Actually, as was pointed out to us by Grosse-Kl\"onne and Berkovich, the proof of Drinfeld, in the case $\ell=p$, is flawed but one can find a correct proof in \cite{FDP}.}.
\end{remark}

  \subsection{Proof of Theorem \ref{intro1}}
 The starting point of  computations of pro-\'etale and \'etale cohomologies in these theorems is the same: the classical comparison theorem between $p$-adic nearby cycles and syntomic sheaves \cite{Ts}, \cite{CN}. When applied to the Stein spaces we consider here it yields:
 \begin{proposition}
 Let $X$ be a semistable Stein formal scheme\footnote{See Section \ref{assumptions} for the definition.} over $\so_K$.  Then the Fontaine-Messing period morphisms
 \begin{align*}
 & \alpha^{\rm FM}: \R\Gamma_{\synt}({X}_{\so_C},\Q_p(r))\to \R\Gamma_{\proeet}(X_C,\Q_p(r)),\\
  & \alpha^{\rm FM}: \R\Gamma_{\synt}({X}_{\so_C},\Z_p(r))_{\Q_p}\to \R\Gamma_{\eet}(X_C,\Q_p(r))
\end{align*}
are strict quasi-isomorphisms after truncation  $\tau_{\leq r}$. 
 \end{proposition}
Here the crystalline geometric syntomic cohomology is that defined by  Fontaine-Messing (see Section \ref{definitionFM} for the details)
$$
\R\Gamma_{\synt}({X}_{\so_C},\Z_p(r)):=[\R\Gamma_{\crr}({X}_{\so_C})^{\phi=p^r}\to \R\Gamma_{\crr}({X}_{\so_C})/F^r], \quad F^r\R\Gamma_{\crr}({X}_{\so_C}):=\R\Gamma_{\crr}({X}_{\so_C},\sj^{[r]}),
$$
where the crystalline cohomology is  absolute, i.e., over $W(k)$, and we use  $[A\to B]$ to denote the mapping fiber.
The syntomic cohomology $\R\Gamma_{\synt}({X}_{\so_C},\Q_p(r))$ is defined by taking $\R\Gamma_{\synt}(-,\Z_p(r))_{\Q_p}$ on quasi-compact pieces and then gluing.

The next  step is to transform the Fontaine-Messing type syntomic cohomology (that works very well for defining period maps from syntomic cohomology to \'etale cohomology
but is not terribly useful for computations) into Bloch-Kato type syntomic cohomology (whose definition is
motivated by the Bloch-Kato's definition of Selmer groups; it involves much more concrete objects).
This can be done in the case of   the pro-\'etale topology\footnote{At least when $X$ is associated to a weak formal scheme.} but  only partially in the case of the \'etale topology. 
\subsubsection{Pro-\'etale cohomology.} For the pro-\'etale topology, we define a Bloch-Kato syntomic cohomology (denoted by $\R\Gamma^{\rm BK}_{\synt}({X}_{C},\Q_p(r))$) using overconvergent differential forms which, due to imposed overconvergence, behaves reasonably well locally.  Then we construct a map from Fontaine-Messing to Bloch-Kato syntomic cohomology 
as shown in the top part of the following commutative diagram, where the rows are distinguished triangles
$$
\xymatrix@C=1.2cm{
\R\Gamma_{\synt}({{X}}_{\so_C},\Q_p(r))\ar[r]\ar[d]^{\wr} & \R\Gamma_{\crr}({{X}}_{\so_C},F)^{\phi=p^r}\ar[r]^-{\can}\ar[d]^{\wr} &  \R\Gamma_{\crr}({{X}}_{\so_C},F)/F^r\ar[d]^{\wr}\\
\R\Gamma^{\rm BK}_{\synt}({X}_{C},\Q_p(r))\ar[r]\ar[d]^{\theta} & (\R\Gamma_{\hk}(X_k)\wh{\otimes}_F\B^+_{\st})^{N=0,\phi=p^r}\ar[d]^{\theta\iota_{\hk}}\ar[r]^{\iota_{\hk}\otimes\iota} & 
(\R\Gamma_{\dr}(X_K)\wh{\otimes}_K\B^+_{\dr})/F^r\ar[d]^{\theta}\\
 \Omega^{\geq r}({X_K})\wh{\otimes}C\ar[r] & \Omega\kr({X_K})\wh{\otimes}_KC\ar[r] & \Omega^{\leq r-1}({X_K})\wh{\otimes}_KC.
 }
 $$
 Here $\R\Gamma_{\crr}({{X}}_{\so_C},F)$  and its filtrations are defined by the same procedure as $\R\Gamma_{\synt}({{X}}_{\so_C},\Q_p(r))$ (starting from rational absolute crystalline cohomology). 
 The horizontal triangles are distinguished (the top two by definition). 
 The construction of the top vertical  maps and the proof that they are isomorphisms is  nontrivial and constitutes the technical heart of this paper. 
 These maps are basically  K\"unneth maps, that use  the interpretation of period rings as crystalline  cohomology of certain ``base'' rings (for example, $\A_{\crr}\simeq \R\Gamma_{\crr}(\so_{C})$),  coupled with a rigidity of  $\phi$-eigenspaces of crystalline chomology, and  followed by a change of topology (from crystalline to overconvergent) that can be done because $X_K$ is Stein (hence $X_k$ has proper and smooth irreducible components). To control the topology we work in the derived category of locally convex topological vector spaces over $\Q_p$ which, since $\Q_p$ is spherically complete, is reasonably well-behaved.

 The bottom vertical maps in the diagram are induced by the projection $\theta: \B^+_{\dr}\to C$ and use  the fact that, since  $X_K$ is  Stein, we have   $\R\Gamma_{\dr}(X_K)\simeq \Omega\kr({X_K})$.
 The diagram in Theorem \ref{intro1} follows by applying $H^r$ to the above diagram.
\subsection{Proof of Theorem \ref{intro2}} 
 To prove the pro-\'etale  part of Theorem \ref{intro2}, by Theorem \ref{intro1}, it suffices  to show that 
 \begin{equation}
 \label{deszcz}
 (H^r_{\hk}(X_k)\wh{\otimes}_F\B^+_{\st})^{N=0,\phi=p^r}\simeq {\rm Sp}_r(\Q_p)^*.
 \end{equation}
   But we know from Schneider-Stuhler \cite{SS}
   that there is a natural isomorphism $H^r_{\dr}(X_K)\simeq {\rm Sp}_r(K)^*$ of $G$-representations.  Moreover, we know that 
   both sides are generated by standard symbols, i.e.,  cup products of  symbols of $K$-rational hyperplanes thought of as invertible functions on $X_K$ (this is because $ {\rm Sp}_r(K)^*$, by definition,  is generated by standard symbols 
   and Iovita-Spiess  prove that so is  $H^r_{\dr}(X_K)$) 
   and that this isomorphism is compatible with symbols \cite[Theorem 4.5]{IS}. 
   Coupled with the Hyodo-Kato isomorphism and the  irreducibility of the representation ${\rm Sp}_r(K)^*$ this yields a natural isomorphism
 $H^r_{\hk}(X_k)\simeq {\rm Sp}_r(F)^*$. 
 This isomorphism 
 is unique once we impose that it should be compatible with the standard symbols. It follows that  
we have a natural isomorphism $H^r_{\hk}(X_k)^{\phi=p^r}\simeq {\rm Sp}_r(\Q_p)^*$,
which implies $H^r_{\hk}(X_k)\cong F\otimes_{\Q_p}H^r_{\hk}(X_k)^{\phi=p^r}$ 
and (\ref{deszcz}). 
  
\subsubsection{\'Etale cohomology.}  The situation is more complicated for \'etale cohomology. Let $X$ be a semistable Stein formal scheme over~$\so_K$. 
An analogous  computation to the one  above yields the following strict quasi-isomorphism of distinguished triangles (see Section \ref{ZURICH})
  $$
\xymatrix@C=1cm@R=.8cm{
\R\Gamma_{\synt}({X}_{\so_C},\Z_p(r))_{\Q_p}\ar[r]\ar@{=}[d] & \R\Gamma_{\crr}({X}_{\so_C})^{\phi=p^r}_{\Q_p}\ar[r]\ar[d]^{\wr} &  \R\Gamma_{\crr}({X}_{\so_C})_{\Q_p}/F^r\ar[d]^{\wr}\\
\R\Gamma_{\synt}({X}_{\so_C},\Z_p(r))_{\Q_p}\ar[r]& (\R\Gamma_{\crr}(X_k/\so_F^0)\wh{\otimes}_{\so_F}\A_{\st})_{\Q_p}^{N=0,\phi=p^r}\ar[r]^-{\gamma_{\hk}\otimes\iota} & 
(\R\Gamma_{\dr}(X)\wh{\otimes}_{\so_K}\A_{\crr,K})/F^r,
 }
 $$
 where $\so_F^0$ denotes $\so_F$ equipped with the log-structure induced by $1\mapsto 0$ and $\A_{\st}$, $\A_{\crr,K}$ are certain period rings. But, in general,  the map $\gamma_{\hk}$ is difficult to identify. 
 In the case of the  Drinfeld half-space though its domain and target simplify significantly by the acyclicity of the sheaves of differentials proved by Grosse-Kl\"onne \cite{GKI, GKB}.
This  makes it possible to describe it and, as a result, to compute the \'etale syntomic cohomology.
 
    Let  $X$ be the standard formal model of ${\mathbb H}^d_K$ \cite[Section 6.1]{GK2}. Set $${\rm HK}_r:=(\R\Gamma_{\crr}(X_k/\so_F^0)\wh{\otimes}_{\so_F}\A_{\st})_{\Q_p}^{N=0,\phi=p^r}.$$
 We show that there are natural $G\times\sg_K$-equivariant strict (quasi-)isomorphisms (see Lemma \ref{formulas1}, Proposition \ref{lyon11})
 \begin{align}
 \label{kwakus1}
 H^r{\rm HK}_r  \simeq H^0_{\eet}(X_{\overline{k}},W\Omega^r_{\log})_{\Q_p}, & \quad 
  H^{r-1}{\rm HK}_r  \simeq (H^0_{\eet}(X_{\overline{k}},W\Omega^{r-1}_{\log})\wh{\otimes}_{\so_F}\A^{\phi=p}_{\crr})_{\Q_p},\\
    (\R\Gamma_{\dr}(X)\wh{\otimes}_{\so_K}\A_{\crr,K})/F^r& \simeq \oplus_{r-1\geq i\geq 0}(H^0(X,\Omega^i)\wh{\otimes}_{\so_K}\A_{\crr,K})/F^{r-i}[-i],\notag
\end{align}
where $W\Omega^r_{\log}$ is the sheaf of logarithmic de Rham-Witt differentials.
They follow from the isomorphisms (see Proposition \ref{lyon11})
\begin{align} 
\label{deszcz1}
H^0_{\eet}(X_{\overline{k}},W\Omega^r_{\log})\wh{\otimes}_{\Z_p}W(\overline{k}) & \stackrel{\sim}{\to} H^r_{\crr}(X_{\overline{k}}/W(\overline{k})^0),\\
\iota_{\hk}:  H^r_{\crr}(X_k/\so_F^0)\otimes_{\so_F}K & \simeq H^r_{\dr}(X)\otimes_{\so_K}K.\notag
\end{align}
The second one is just the  original Hyodo-Kato isomorphism from \cite{HK}. The first one follows from the acyclicity of the sheaves $\Omega^j_X$  (see Lemma \ref{degenerate}) and the fact that $X_k$ is pro-ordinary (see Corollary \ref{gen-ordinary}), which, in turn and morally speaking, 
follows from the fact that $X_k$ is a normal crossing scheme whose  closed strata are classically ordinary (being products of blow-ups of projective spaces).
Now, 
  the acyclicity of the sheaves $\Omega^j_X$ again and  the fact that the differential is trivial on their global sections (both facts proved by Grosse-Kl\"onne   \cite{GKI}, \cite{GKB}) imply (\ref{kwakus1}).
  
  Hence, we obtain the long exact sequence 
$$
(H^0_{\eet}(X_{\overline{k}},W\Omega^{r-1}_{\log})
  \wh{\otimes}_{\Z_p}{\A}^{\phi=p}_{\crr})_{\Q_p}\lomapr{\gamma^{\prime}_{\hk}}
(H^0(X,\Omega^{r-1})\wh{\otimes}_{\so_{K}}\so_C)_{\Q_p}\to H^r_{\synt}({X}_{\so_C},\Z_p(r))_{\Q_p}\to      H^0_{\eet}(X_{\overline{k}},W\Omega^r_{\log})_{\Q_p}  \to 0
$$
We check that  the map $\gamma^{\prime}_{\hk}$ is surjective (see Corollary \ref{synt11}):  (a bit surprisingly) the Hyodo-Kato isomorphism $\iota_{\hk}$ above holds already integrally 
and $\gamma_{\hk}^{\prime}=\iota_{\hk}\otimes\theta$, where  $\theta: \A_{\crr}^{\phi=p}\to\so_C$ is the canonical projection.
This yields the isomorphism
 $$
 H^r_{\synt}({X}_{\so_C},\Z_p(r))_{\Q_p}\stackrel{\sim}{\to}      H^0_{\eet}(X_{\overline{k}},W\Omega^r_{\log})_{\Q_p}. $$
 A careful study of the topology allows to conclude  that this isomorphism is topological.

  Hence it remains to show that there exists a natural topological isomorphism
 \begin{equation}
 \label{deszcz2}
 H^0_{\eet}(X_{\overline{k}},W\Omega^r_{\log})_{\Q_p}\simeq {\rm Sp}^{\cont}_r(\Q_p)^*.
 \end{equation}
 We do that (see Theorem \ref{comp19}) by showing that we can replace $\overline{k}$ by $k$ and using the maps 
 $$
 H^0_{\eet}(X_k,W\Omega^r_{\log})\otimes_{\Z_p}K \stackrel{f}{\to}
 H^r_{\dr}(X_K)^{\text{ $G$-{\rm bd}}} \simeq  ({\rm Sp}_r(K)^*)^{\text{$G$-{\rm bd}}} 
 \stackrel{\sim}{\leftarrow}
    {\rm Sp}_r^{\cont}(\Q_p)^*\otimes_{\Q_p} K.
$$
   Here, the  second  isomorphism is that of  Schneider-Stuhler \cite{SS}. The map $f$ (a composition of natural maps with the Hyodo-Kato isomorphism) is injective by pro-ordinarity of $Y$. It is surjective because $H^0_{\eet}(X_k,W\Omega^r_{\log})$ is compact and nontrivial and 
   ${\rm Sp}^{\rm cont}_r(\Z_p)/p\simeq {\rm Sp}_r({\mathbf F}_p)$ is irreducible -- a  nontrivial fact proved by Grosse-Kl\"onne \cite[Cor. 4.3]{GKD}. This yields an isomorphism (\ref{deszcz2}). 
  
   Part (3) of Theorem \ref{intro2} follows now easily from the two previous (compatible) parts and the fact that $({\rm Sp}_r(K)^*)^{\text{$G$-{\rm bd}}} \simeq
    {\rm Sp}_r^{\cont}(\Q_p)^*\otimes_{\Q_p} K$ and $ H^r_{\dr}(X_K)^{\text{ $G$-{\rm bd}}} \simeq  H^r_{\dr}(X)\otimes_{\so_K}K$ (the latter isomorphism uses the fact that $X$ can be covered by  $G$-translates of an open subscheme $U$ such that $U_K$ is an affinoid), 

    \begin{acknowledgments}This paper owes great deal to the work of Elmar Grosse-Kl\"onne. We are very grateful to him for his patient and detailed explanations of the computations and constructions in his papers. 
 We would like to thank Fabrizio
 Andreatta, Bruno Chiarellotto, Fr\'ed\'eric D\'eglise, Ehud de Shalit, Veronika Ertl, Laurent Fargues, Florian Herzig, Luc Illusie, Arthur-C\'esar Le Bras,  Sophie Morel, Arthur Ogus, and Lue Pan for helpful conversations related to the subject of this paper. 
 This paper was partly written during our visits to the IAS at Princeton,  the Tata Institute in Mumbai, Banach Center in Warsaw (P.C, W.N), BICMR in Beijing (P.C.), Fudan University in Shanghai (W.N.), Princeton University (W.N.), and the Mittag-Leffler Institute (W.N.). We thank these institutions for their hospitality. 
 \end{acknowledgments}
 \subsubsection{Notation and conventions.}\label{Notation}
 Let $\so_K$ be a complete discrete valuation ring with fraction field
$K$  of characteristic 0 and with perfect
residue field $k$ of characteristic $p$. Let $\varpi$ be a uniformizer of $\so_K$ that we fix in this paper. Let $\ovk$ be an algebraic closure of $K$ and let $\so_{\ovk}$ denote the integral closure of $\so_K$ in $\ovk$. Let
$W(k)$ be the ring of Witt vectors of $k$ with 
 fraction field $F$ (i.e, $W(k)=\so_F$); let $e$ be the ramification index of $K$ over $F$.   Set $\sg_K=\Gal(\overline {K}/K)$, and 
let $\sigma$ be the absolute
Frobenius on $W(\overline {k})$. 

We will denote by $\so_K$,
$\so_K^{\times}$, and $\so_K^0$, depending on the context,  the scheme $\Spec ({\so_K})$ or the formal scheme $\Spf (\so_K)$ with the trivial, the canonical (i.e., associated to the closed point), and the induced by $\N\to \so_K, 1\mapsto 0$,
log-structure, respectively.  Unless otherwise stated all   formal schemes are $p$-adic, locally of finite type, and equidimensional. For a ($p$-adic formal) scheme $X$ over $\so_K$, let $X_0$ denote
the special fiber of $X$; let $X_n$ denote its reduction modulo $p^n$.

We will denote by $\A_{\crr}, \B^+_{\crr}, \B^+_{\st},\B^+_{\dr}$ the crystalline, semistable, and  de Rham period rings of Fontaine. We have
$ \B^+_{\st}=\B^+_{\crr}[u]$ and $\phi(u)=pu, N(u)=-1$. The embedding  $\iota=\iota_{\varpi}: \B^+_{\st}\to \B^+_{\dr}$ is defined by $u\mapsto u_{\varpi}=\log([\varpi^{\flat}]/\varpi)$ and the Galois action on $\B^+_{\st}$ is induced from the one on $\B^+_{\dr}$ via this embedding.

   Unless otherwise stated, we work in the derived (stable) $\infty$-category $\sd(A)$ of left-bounded complexes of a quasi-abelian category $A$ (the latter will be clear from the context).
  Many of our constructions will involve (pre)sheaves of objects from $\sd(A)$. The reader may consult the notes of Illusie \cite{IL} and Zheng \cite{Zhe} for a brief introduction to how to work with such (pre)sheaves  and \cite{Lu1}, \cite{Lu2} for a thorough treatment.
  
  We will use a shorthand for certain homotopy limits:
 if $f:C\to C'$ is a map  in the derived $\infty$-category of a quasi-abelian  category, we set
$$[\xymatrix{C\ar[r]^f&C'}]:=\holim(C\to C^{\prime}\leftarrow 0).$$

\section{Review  of $p$-adic functional analysis}\label{review}
We gather here some basic facts from $p$-adic functional analysis that we use in the paper. 
\emph{The reader is advised, on the first reading, to ignore this chapter 
and topological issues in ensuing computations.}
 \subsection{The category of locally convex $K$-vector spaces} We start with the rational case, where we work in the category of locally convex $K$-vector spaces.  Our main references are \cite{Sch}, \cite{Pr},  \cite{Em}.
\subsubsection{Derived category of locally convex $K$-vector spaces} 
  A topological $K$-vector space\footnote{For us, a {\em $K$-topological vector space} is a $K$-vector space with a linear topology.} is called {\em locally convex} ({\em convex} for short) if there exists a 
neighbourhood basis of the origin consisting of $\so_K$-modules. Since $K$ is spherically complete, the theory of such spaces resembles the theory of locally convex  topological vector spaces over ${\mathbf R}$ or
 ${\mathbf C}$ (with some simplifications). 
 
 We denote by $C_K$ the category of convex $K$-vector spaces. 
  It  is a quasi-abelian category\footnote{An additive category with kernels and cokernels is called {\em quasi-abelian} if every pullback of a strict epimorphism is a strict epimorphism and every pushout of a strict monomorphism  is  a strict monomorphism. Equivalently, an additive category with kernels and cokernels is called {\em quasi-abelian} if  $\Ext(-,-)$ is bifunctorial.}  \cite[2.1.11]{Pr}.
Kernels, cokernels, images, and coimages  are taken in the  category of vector spaces and equipped with the induced topology \cite[2.1.8]{Pr}. A morphism $f:E\to F$ is {\em strict} if and only if it is relatively open, i.e., for any neighbourhood $V$ of $0$ in $E$ there is a neighbourhood $V^{\prime}$ 
of $0$ in $F$ such that $f(V)\supset V^{\prime}\cap f(E)$ \cite[2.1.9]{Pr}. 

Our convex $K$ vector spaces are not assumed to be separated.
We often use the following simple observation: {\it if $F$ is separated and we have an injective 
morphism $f:E\to F$ then $E$ is separated as well; if, moreover, 
$F$ is finite dimensional and  $f$ is bijective then $f$ is an isomorphism in $C_K$}. 
 
 The category $C_K$ has a  natural
 exact category structure: the admissible monomorphisms are  embeddings, the
admissible epimorphisms are open surjections. 
A complex $E\in C(C_K)$ is called   {\em strict} if its differentials are strict.
   There are  truncation functors on $C(C_K)$: 
\begin{align*}
\tau_{\leq n}E & :=\cdots \to E^{n-2}\to E^{n-1}\to 
\ker(d_n)\to 0\to\cdots\\
 \tau_{\geq n} E & :=\cdots \to 0\cdots \to
\coim(d_{n-1}) \to E^n\to E^{n+1}\to\cdots
\end{align*}
with cohomology objects $$\wt{H}^n(E):=
\tau_{\leq n}\tau_{\geq n}(E)=(\coim(d_{n-1})\to \ker(d_n)).
$$
We note that here $\coim(d_{n-1})$ and $\ker(d_n)$ are equipped naturally with the quotient and subspace topology, respectively. The  cohomology $H^*(E)$ taken in the category of $K$-vector spaces we will call {\em algebraic} and, if necessary, we will always equip it with the sub-quotient topology. 

  We will denote the left-bounded derived $\infty$-category of $C_K$ by $\sd(C_K)$. 
       A morphism of complexes that is a quasi-isomorphism in $\sd(C_K)$, i.e., its cone is strictly exact,  will be called a {\em strict quasi-isomorphism}. We will denote by $D(C_K)$ the homotopy category of $\sd(C_K)$ \cite[1.1.5]{Pr}.

 For $n\in \Z$, let $D_{\leq n}(C_K)$ (resp. $D_{\geq n}(C_K)$) denote the full subcategory of $D(C_K)$  of complexes that are strictly exact in degrees $k >n$ (resp. $k<n$)\footnote{Recall \cite[1.1.4]{Sn} that a sequence $A\stackrel{e}{\to} B\stackrel{f}{\to} C$ such that $fe=0$ is called {\em strictly exact} if the morphism $e$ is strict and the natural map $\im e\to \ker f $ is an isomorphism. }. The above truncation maps extend to truncations functors
  $\tau_{\leq n}: D(C_K)\to D_{\leq n}(C_K)$ and $\tau_{\geq n}: D(C_K)\to D_{\geq n}(C_K)$. The pair $(D_{\leq n}(C_K),D_{\geq n}(C_K)$) defines a $t$-structure on $D(C_K)$ by \cite{Sn}. The (left) heart $D(C_K)^{\heartsuit}$  is an abelian category $LH(C_K)$: every object of $LH(C_K)$ is represented (up to  equivalence)  by a monomorphism $f:E\to F$, where $F$ is in degree $0$, i.e., it is isomorphic to a complex $0\to E\stackrel{f}{\to} F\to 0$; 
{\it if $f$ is strict} this object is also represented by the cokernel of $f$
(the whole point of this construction is to keep track of the two possibly different topologies
on $E$: the given one and the one inherited by the inclusion into $F$).
 
  We have an embedding
$I: C_K\hookrightarrow LH(C_K)$, $E\mapsto (0\to E)$,
that induces an equivalence $D(C_K)\stackrel{\sim}{\to} D(LH(C_K))$ that is compatible with $t$-structures. These t-structures pull back to  $t$-structures on the derived dg categories $\sd(C_K), \sd(LH(C_K))$ and so does the above equivalence. There is a functor (the {\em classical part}) $C: LH(C_K)\to C_K$ that sends the monomorphism $f: E\to F$ to $\coker f$. We have $CI\simeq \id_{C_K}$ and a natural epimorphism $e: \id_{LH(C_K)}\to IC$.

  We will denote by $\wt{H}^n: \sd(C_K)\to \sd(LH(C_K))$ the associated cohomological functors.  Note that $C\wt{H}^n=H^n$ and we have a natural epimorphism $\wt{H}^n\to IH^n$.  If, evaluated on $E$,  this epimorphism is an isomorphism  we will say that the  cohomology $\wt{H}^n(E)$ is {\em classical}.

  We will often use the following simple facts (\cite[Prop. 1.2.28, Cor. 1.2.27]{Sn}):
  \begin{enumerate}
  \item If, in the  following short exact sequence in $LH(C_K)$,  both $A_1$ and $A_2$ are in the essential image of $I$ then so is $A$:
  $$
  0\to A_1\to A\to A_2\to 0.
  $$
\item A complex $E\in C(C_K)$ is strictly exact in a specific degree if and only if $\sd(I)(E)$ is exact in the same degree.
  \end{enumerate}
  
  All the above has a dual version: we have a notion of a costrict morphism, a right abelian envelope $I: C_K\to RH(C_K)$, and the cohomology objects $\wt{H}^n: \sd(C_K)\to \sd(RH(C_K)), $ $ \wt{H}^n:=
  (\coker(d_{n-1})\to\im(d_n))$ in degrees $0$ and $1$ (which we will, if necessary, write as $RH^n$). We have $C\wt{H}^n=H^n$, where  $C: RH(C_K)\to C_K$ sends the epimorphism $f: E\to F$ to $\ker f$. 
  There is a natural monomorphism $IH^n\to \wt{H}^n$; if, evaluated on $E$,  this monomorphism is an isomorphism  we will say that the  cohomology $\wt{H}^n(E)$ is {\em classical}. 
\subsubsection{Open Mapping Theorem}\label{OMT}
Let $f:X\to Y$ be a continuous surjective map of locally convex $K$-vector spaces. We will need a well-known version of the Open Mapping Theorem that says that $f$ is open if both $X$ and $Y$ are $LF$-spaces, i.e., countable inductive limits of Fr\'echet spaces\footnote{ If the spaces involved are actually Banach, we will sometimes use the notation $LB$ instead of $LF$.}. 

If $E,F$ are Fr\'echet, $f:E\to F$ is strict if and only if $f(E)$ is closed in $F$
(the ``if'' part follows from the Open Mapping Theorem, the ``only if'' part from the fact
that a Fr\'echet space is a metric space and a complete subspace of a metric space is closed). The following lemma is a simple consequence of this observation but we spell it out because we will use all the time. 
\begin{lemma}
\label{kolobrzeg-winter}
\begin{enumerate}
\item Let $E$ be a complex of Banach or Fr\'echet spaces. If $E$  has  trivial algebraic cohomology then it is strictly acyclic. 
\item A morphism $E_1\to E_2$ of complexes of Banach or Fr\'echet spaces that is an (algebraic)  quasi-isomorphism is a strict quasi-isomorphism. \end{enumerate}
\end{lemma}
\begin{proof}
The second point follows from the first one by applying the latter to the cone of the morphism. For the first point, note that the kernel of a differential is a closed subspace of a Fr\'echet space hence a Fr\'echet space and we can evoke the observation above the lemma. 
\end{proof}
\subsubsection{ Tensor products} \label{tensor} Let $V$, $W$ be two convex $K$-vector spaces. The abstract tensor product $V\otimes_K W$ can be equipped with several natural topologies among them the  projective and injective tensor product topologies: $V\otimes_{K,\pi}W$ and $ V\otimes_{K,\ve}W$. 
Recall that the projective  tensor product topology is universal 
 for jointly continuous bilinear maps $V \times W \to U$; the injective tensor product topology, on the other hand, is defined by  cross seminorms that satisfy a product formula and is the ``weakest'' 
topology with such property. There is a natural map $V\otimes_{K,\pi}W\to V\otimes_{K,\ve}W$. 
We denote by  $V\wh{\otimes}_{K,\alpha}W$, $\alpha=\pi,\ve$,  the Hausdorff  completion of $V\otimes_{K,\alpha} W$ with respect to the topology $\alpha$. 
  
  Recall the following facts.
\begin{enumerate}
\item The projective tensor product functor $(-)\otimes_{K,\pi}W$ preserves strict epimorphisms; the injective tensor product functor $(-)\otimes_{K,\ve}W$ preserves strict monomorphisms. 
 \item The  natural map
$V\otimes_{K,\pi}W\to V\otimes_{K,\ve}W$ is an isomorphism\footnote{Here we used the fact that our  field $K$ is spherically complete.}  \cite[Theorem 10.2.7]{PG}. In what follows we will often just write $V\otimes_KW$ for both products.
\item From (1), (2), and the exactness properties of Hausdorff completion \cite[Cor. 1.4]{Var},  it follows that the tensor product functor $(-)\wh{\otimes}_KW: C_K\to C_K$ is left exact, i.e., it carries strictly exact sequences
$$
0\to V_1\to V_2\to V_3\to0
$$
to strictly exact sequences
$$
0\to V_1\wh{\otimes}_KW\to V_2\wh{\otimes}_K W\to V_3\wh{\otimes}_K W.
$$
Moreover, the image of the last map  above is dense \cite[p.45]{Var}. It follows that this map is surjective if its image  is complete as happens, for example,  in the case when the spaces $V_*, W$ are Fr\'echet \cite[Cor. 1.7]{Var}. 
\item For $V=\varprojlim_nV_n$, where each $V_n$ is a Fr\'echet space, and for a Fr\'echet space $W$, 
there is a natural isomorphism 
$$V\wh{\otimes}_{K,\pi} W=(\varprojlim_nV_n)\wh{\otimes}_{K,\pi}W
\stackrel{\sim}{\to} \varprojlim_n(V_n\wh{\otimes}_{K,\pi}W).
$$ 
For products this is proved in  \cite[Prop. 9, p.192]{LNM} and the general case follows from the fact that tensor product is exact on sequences of  Fr\'echet spaces.
\item  Let $\{V_n\}$, $n\in \N$,  be a regular\footnote{Inductive system $\{V_n\}$, $n\geq 0$, with injective transition maps is called {\em regular} if for each bounded set $B$ in $V=\varinjlim_n V_n$ there exists an $n$ such that $B\subset V_n$ and $B$ is bounded in $V_n$.} inductive system of Fr\'echet spaces with injective nuclear\footnote{A map $f:V\to W$ between two convex $K$-vector spaces is called {\em nuclear} if it can be factored $f:V\to V_1\stackrel{f_1}{\to} W_1\to W$, where the map $f_1$ is a compact map between Banach spaces.} transition maps. Then, for any Banach space $W$,  we have an isomorphism \cite[Theorem 1.3]{Mo}
$$
(\varinjlim_nV_n)\wh{\otimes}_K W\stackrel{\sim}{\leftarrow} \varinjlim_n (V_n\wh{\otimes}_K W).
$$
\end{enumerate}

 \subsubsection{Acyclic inductive systems}\label{acyclic-inductive}
If  $\{V_n\}_n, {n\in\N}$ is an inductive system of locally convex $K$-vector spaces   then it is called {\em  acyclic}  if 
   $\LL^1\hocolim_nV_n=0$.
  We will find useful  the following criterium of acyclicity \cite[Theorem 1.1]{W1}.
\begin{proposition}
An inductive system $\{V_n\}_n, n\in\N$, of Fr\'echet spaces with injective transition maps is acyclic if and only if in every space $V_n$ there is a convex neighbourhood $U_n$ of $0$
such that
\begin{enumerate}
\item $U_n\subset U_{n+1}, n\in\N$, and
\item For every $n\in\N$ there is $m >n$ such that all topologies of the spaces $V_k$, $k>m$, coincide on $U_n$.
\end{enumerate}
\end{proposition}
\subsubsection{Derived tensor products.}
The category $C_K$ has enough injectives hence we have the right derived functor $V\wh{\otimes}_K^R W$ of the tensor functor $V\wh{\otimes}_K W$. We will need to know the conditions under which  it is strictly quasi-isomorphic to the tensor functor.
\begin{lemma} \label{bios2}The natural map
$$
V\wh{\otimes}_KW\to V\wh{\otimes}^R_KW
$$
is a quasi-isomorphism when
\begin{enumerate}
\item both $V$ and $W$ are Fr\'echet spaces,
\item $V=\varinjlim_n V_n$, for an  acyclic inductive system $\{V_n\}, n\in\N$, of Banach spaces, and $W$ is a Banach space. 
\end{enumerate}
\end{lemma}
\begin{proof}
In the first case, take an injective resolution $W\to I\kr$  of $W$ by Fr\'echet spaces $I^i$, $i\geq 0$. This means that the map $W\to I\kr$ is a strict quasi-isomorphism. Such a resolution exists by \cite[Prop. 2.1.12]{Pr}. Tensoring this resolution  with $V$ we get a sequence
\begin{equation}
\label{bios1}
0\to V\wh{\otimes}W\to V\wh{\otimes}I^0\to V\wh{\otimes}I^1\to \cdots
\end{equation}
By Section \ref{tensor}, this sequence is strictly exact, as wanted. 

 In the second case, we take an injective  resolution $W\to I\kr$  of $W$ by Banach  spaces $I^i$, $i\geq 0$. Such a resolution exists by loc. cit. 
 Tensoring this resolution  with $V$ we get a sequence (\ref{bios1}) as above.
Since $V=\varinjlim_nV_n$, by Section \ref{tensor}, this sequence is an inductive limit of sequences
$$
0\to V_n\wh{\otimes}W\to V_n\wh{\otimes}I^0\to V_n\wh{\otimes}I^1\to \cdots
$$
which are strictly exact by Section \ref{tensor}. Hence, by Section \ref{acyclic-inductive}, the sequence (\ref{bios1}) is strictly exact, as wanted.
\end{proof}
\subsection{The category of pro-discrete $\so_K$-modules}\label{pro-discrete}
  For integral topological questions we have found it convenient to work in the category $PD_K$ of pro-discrete $\so_K$-modules.  
For details see  \cite[Section 2]{BC}, \cite[Section 1]{Wit},  \cite{KS}, \cite{KS1}. 
\subsubsection{The category of pro-discrete $\so_K$-modules and its ind-completion.}
Objects in the category $PD_K$ are topological $\so_K$-modules that are countable inverse limits, as topological $\so_K$-modules, of discrete $\so_K$-modules $M^i$, $i\in \N$. 
It is a quasi-abelian category.  It has countable filtered projective limits. Countable product is exact. 

Objects in $PD_K$ are not necessarily separated for the $p$-adic topology: for example 
Banach spaces are objects of $PD_K$ (if $B$ is a Banach with unit ball $B_0$, then $B=\varprojlim B/p^nB_0$),
hence also countable products or projective limits of Banach spaces (i.e.~Fr\'echet spaces) are objects of $PD_K$.
    
   Inside  $PD_K$ we distinguish the category $PC_K$ of pseudocompact $\so_K$-modules, i.e., pro-discrete modules $M\simeq\invlim_iM_i$ such that each $M_i$ is of finite length (we note that if $K$ is a finite extension of $\Q_p$ this is equivalent to $M$ being profinite). It is an abelian category. It has countable exact products as well as exact countable filtered projective limits. 

Let $\Ind(PD_K)$ be the  ind-completion of $PD_K$. That is, 
$PD_K$ is a full subcategory of $\Ind(PD_K)$ and  $\Ind(PD_K)$ has the universal property that any functor $F: PD_K\to C$ into a category with filtered inductive limits extends uniquely to a functor $\wt{F}: \Ind(PD_K)\to C$ which preserves 
filtered inductive limits. In particular, any functor $F: PD_K\to C$ induces a compatible functor $\wt{F}: \Ind(PD_K)\to \Ind(C)$ and if $C$ has filtered  inductive limits then the inclusion $C\to \Ind(C)$ 
has a canonical left inverse $\Ind(C)\to C$.

The category $\Ind(PD_K)$ is also quasi-abelian \cite[Theorem 1.3.1]{KS}. The natural functor $PD_K\to \Ind(PD_K)$ is exact. The category $\Ind(PD_K)$ admits filtered inductive limits and projective limits. The $\varprojlim$ functor is left exact. For any functor $F:PD_K\to C$ to a quasi-abelian category, the functor $F$ is left, resp. right, exact if and only if so is the functor $\wt{F}$. 
 \subsubsection{Tensor product}
 For $M,N\in PD$ we define the completed tensor product
 $$ M\wh{\otimes}_{\so_K}N:=\varprojlim_{U\in\su_M, V\in \su_N}M/U\otimes_{\so_K}N/V, $$
 where $\su_M, \su_N$ denote the inductive system  of open submodules of $M,N$ and $M/U\otimes_{\so_K}N/V$ is given the discrete topology. It is a pro-discrete $\so_K$-module. 
 It satisfies the usual  universal property with respect to pro-discrete $\so_K$-modules \cite[Prop. 6.1]{Wit}.  In general, the completed tensor product  is neither right nor left exact. It is however right exact  when restricted to $PC_K$ \cite[Prop. 1.10]{Wit}. It commutes also with filtered limits $\{N_i\}_i$ such that $N=\varprojlim_iN_i$ surjects onto $N_i$, $i\in I$ \cite[Prop. 1.7]{Wit}; in particular, it commutes with products of pro-discrete $\so_K$-modules and with filtered limits of pseudocompact $\so_K$-modules. 
 
 \subsubsection{The functor to convex spaces.} 
 Since $K \simeq \varinjlim(\so_K\lomapr{\varpi} \so_K\lomapr{\varpi}\so_K\lomapr{\varpi} \cdots)$, the 
algebraic tensor product $M\otimes_{\so_K}K$ is an inductive limit:
$$
M\otimes_{\so_K}K\simeq  \varinjlim(M\lomapr{\varpi} M\lomapr{\varpi}M\lomapr{\varpi} \cdots).
$$
We equip it with the induced inductive limit topology. 
This defines a tensor product functor 
  $$(-){\otimes}K: PD_K\to C_{K}, \quad M\mapsto M\otimes _{\so_K}K.
  $$
 Since $C_K$ admits filtered inductive limits, the functor $(-){\otimes}K$ extends 
to a functor $(-){\otimes}K: \Ind(PD_K)\to C_{K}$. 
   \begin{remark}
  Recall that $K$ belongs to $PD_K$: $K\simeq \varprojlim(K/\so_K\stackrel{\varpi}{\leftarrow}K/\so_K\stackrel{\varpi}{\leftarrow} \cdots)$.  The above  tensor functor should be distinguished from the functor $(-)\wh{\otimes}K: PD_K\to PD_K$ discussed above which is always zero and which we will never use.
  \end{remark}
  
    The  functor $(-){\otimes}K$ is right exact but not, in general, left exact\footnote{We will call a functor $F$ right exact if  it transfers strict exact sequences $0\to A\to B\to C\to0$ to  costrict exact sequences 
    $F(A)\to F(B)\to F(C)\to 0$; functor RR in the language of Schneiders \cite[Section 1.1]{Sch}.}. For example,
  the short strict exact sequence
  $$
  0\to \prod_{i\geq 0}p^i\Z_p\lomapr{\can}\prod_{i\geq 0}\Z_p\to \prod_{i\geq 0}\Z_p/p^i\to 0
  $$
  after tensoring with $\Q_p$ is not costrict exact on the left (note that $(\prod_{i\geq 0}\Z_p/p^i){\otimes}\Q_p$ is not Hausdorff). We will consider its (compatible) left derived functors
  $$
  (-){\otimes}^LK: \sd^{-}(PD_K)\to \Pro(\sd^{-}(C_K)),\quad  (-){\otimes}^LK: \sd^{-}(\Ind(PD_K))\to \Pro(\sd^{-}(C_K)).
  $$
  The following fact is probably well-known but we could not find a reference (see however \cite[Prop. 3.9, Cor. 3.13]{BC} for the case of profinite modules).
  \begin{proposition}\label{acyclic-integral}
  If $E$ is a complex of torsion free and $p$-adically complete (i.e., $E\simeq \varprojlim_n E/p^n$) modules from $PD_K$ then the natural map
  $$
   E{\otimes}^LK\to  E{\otimes}K
  $$ is a strict quasi-isomorhism. 
  \end{proposition}
  \begin{proof}By \cite[Lemma 14.1]{Kel} our proposition is implied by the following lemma that shows that the terms of the complex $E$ are $F$-acyclic for the functor $F=(-){\otimes}K$. 
  \begin{lemma}\label{genial}
If  $0\to A\to B\overset{\pi}{\to} C\to 0$ is a strict exact sequence of  pro-discrete $\so_K$-modules  and $C$ is torsion free and $p$-adically complete 
then  $\pi:B\to C$ admits a $\so_K$-linear continuous section and  $B\simeq A\oplus C$ as a  topological $\so_K$-module.
\end{lemma}
\begin{proof}
The strict exact sequence  $0\to A\to B\to C\to 0$ is a limit of exact sequences 
$0\to A^i\to B^i\to C^i\to 0$, where all the terms are discrete  and $A^{i+1}\to A^i$ is surjective
(idem for $B$ and $C$) \cite[Remark 2.9]{BC}.
Let  $A_i$ be the kernel of  $A\to A^i$ (idem for  $B$ and $C$).  Now, if $X=A,B,C$, the $X_i$
form a decreasing filtration and a series  $\sum x_n$, $x_n\in X$,  converges in $X$ if and only if, for all  $i\in\N$, there exists
$n(i)$ such that  $x_n\in X_i$ for all $n\geq n(i)$. Moreover, 
the sequence $0\to A_i\to B_i\to C_i\to 0$ is exact for all  $i$
(since  $A^i\to B^i$ is  injective).

Let  $\overline C=C/\varpi$ and let  $\overline C_i$ be the image of  $C_i$ in $\overline C$.
The  $\overline C_i$, $i\in\N$,  form a decreasing filtration of  $\overline C$ and  $\overline C_i/\overline C_{i+1}$
is a discrete $k$-module (it is a quotient of  $C_i/C_{i+1}\subset C^{i+1}$).
Choose a basis $(\overline e_{i,j})_{j\in J_i}$ of  $\overline C_i/\overline C_{i+1}$ over $k$,
a lifting $\tilde e_{i,j}$ of  $\overline e_{i,j}$ in $\overline C_i$,
a lifting $e_{i,j}$ of  $\tilde e_{i,j}$ in $C_i$,
and a lifting $\hat e_{i,j}$ of $e_{i,j}$
in $B_i$ (it  exists because  $B_i\to C_i$ is surjective).

Let $Y=\prod_{i\in\N}\ell_\infty^0(J_i,\so_K)$, where $\ell_\infty^0(J_i,\so_K)$ is  the space of sequences 
$(y_{i,j})_{j\in J_i}$ with values in  $\so_K$, going to  $0$ at infinity.
If $y=((y_{i,j})_{j\in J_i})_{i\in I}\in Y$, the double series
$\sum_{i\in\N}(\sum_{j\in J_i}y_{i,j}e_{i,j})$ and
$\sum_{i\in\N}(\sum_{j\in J_i}y_{i,j}\hat e_{i,j})$ converge in $C$ and $B$ respectively:
the series corresponding to a fixed $i$  converges because $y_{i,j}\to 0$ when  $j\to\infty$ and its sum belongs to 
 $C_i$ or $B_i$ and hence the series $\sum_i$ converges as well.
We denote by $\iota_C(y)$ and  $\iota_B(y)$ the sums of these series.
The map   $\iota_X:Y\to X$, $X=B,C$,  is a continuous injection (to check  injectivity,
it suffices to argue  modulo~$\varpi$, where it is clear).  
Moreover, we have $\pi\circ\iota_B=\iota_C$.

To finish the proof of the lemma it suffices to show that  $\iota_C$ is a topological isomorphism
(because then $s=\iota_B\circ\iota_C^{-1}$ gives the desired continuous section and we have a topological isomorphism
$A\oplus C\stackrel{\sim}{\to} B$ sending $(a,c)$ to $a+s(c)$ with  the inverse given by 
$b\mapsto(b-s(\pi(b)),\pi(b))$).

Let us start with proving it modulo~$\varpi$.
For $c\in\overline C$, one constructs an element of $\overline Y$ using the following algorithm.
Set $c_0=c$ and $\overline C_0=\overline C$.  The image of  $c_0$ in  $\overline C_0/\overline C_1$
can be written, in a unique way, in the form
$\sum_{j\in J_0}y_{0,j}e_{0,j}$ (and there is only a finite number of $y_{0,j}$ that are nonzero).
Hence $c_1=c-\sum_{j\in J_0}y_{0,j}\tilde e_{0,j}\in\overline C_1$.
The image of  $c_1$ in  $\overline C_1/\overline C_2$
can be written, in a unique way, in the form
$\sum_{j\in J_1}y_{1,j}e_{1,j}$ (and there is only a finite number of $y_{1,j}$ that are nonzero).
Hence $c_2=c_1-\sum_{j\in J_1}y_{1,j}\tilde e_{1,j}\in\overline C_2$.
We continue in this way  and get in the end an element 
$y_c=((y_{i,j})_{j\in J_i})_{i\in I}$ of  $\overline Y$ whose image
by $\iota_C$ is $c$ and  $c\mapsto y_c$ is a continuous inverse of  $\iota_C$ (modulo~$\varpi$): it is an inverse since we have uniqueness at every stage.

Let  $\overline\alpha:\overline C\to \overline Y$ be the inverse of  $\iota_C$ constructed above.
Write  $[\ ]:\overline Y\to Y$ for the map sending 
$((y_{i,j})_{j\in J_i})_{i\in I}$ to  $(([y_{i,j}])_{j\in J_i})_{i\in I}$, where  $[\ ]:k\to\so_K$
is the  Teichm\"uller lift; it is a continuous map.
The inverse  $\alpha:C\to Y$ of  $\iota_C$ is given by the following algorithm:
for $c\in C$, set $c_0=c$, and $c_1=\frac{1}{\varpi}(c_0-\iota_C([\overline\alpha(\overline c_0)]))$
(we write  $\overline c_0$ for the image of $c_0$ modulo~$\varpi$ and we can divide by $\varpi$
since  $C$ has no torsion).
Then set  $c_2=\frac{1}{\varpi}(c_1-\iota_C([\overline\alpha(\overline c_1)]))$, etc.
Finally, set 
$\alpha(c)=\sum_{n\geq 0}\varpi^n[\overline\alpha(\overline c_n)]$. We have $\iota_C(\alpha(c))=\sum_{n\geq 0}\varpi^n\iota_C([\overline\alpha(\overline c_n)])=
\sum_{n\geq 0}(\varpi^nc_n-\varpi^{n+1}c_{n+1})=c_0=c$. Hence  $\alpha=\iota_C^{-1}$, as wanted.
\end{proof}

  \end{proof}
\section{Syntomic cohomologies}
The period map of Fontaine-Messing (see Section ~\ref{PR3}) 
gives a description of (pro-)\'etale cohomology in
terms of the syntomic cohomology of Fontaine-Messing. This syntomic cohomology can be made
more concrete (see~Section~\ref{PR4})
by mimicking the construction of Selmer groups by Bloch and Kato \cite{BKS}:  syntomic cohomology is described as derived  filtered eigenspaces of Frobenius acting on
the geometric Hyodo-Kato cohomology~\cite{HK}.  Now, the geometric Hyodo-Kato cohomology behaves very badly
locally (as does the de Rham cohomology) and the standard way to fix this problem is to do everything
in an overconvergent way.  So we  
define (see~Section~\ref{system11}) overconvergent syntomic cohomology a la Bloch-Kato, 
replacing the usual Hyodo-Kato cohomology
by the overconvergent Hyodo-Kato cohomology of Grosse-Kl\"onne~\cite{GK2} which we review
in Section~\ref{PR1}. In the next chapter, we
will show (Theorem~\ref{fini3}) that these two syntomic cohomologies coincide for Stein spaces.
This definition a la Bloch-Kato makes it easy to show (Proposition~\ref{fd11})
that the overconvergent syntomic cohomology
sits in a ``fundamental diagram'' involving the de Rham complex and overconvergent Hydo-Kato cohomology. It follows that so does pro-\'etale cohomology and this "fundamental diagram" will become our main tool for computations of the latter later on in the paper.
  \subsection{Overconvergent  Hyodo-Kato  cohomology}\label{PR1}
We will review in this section the definition of the overconvergent Hyodo-Kato cohomology and the overconvergent Hyodo-Kato isomorphism due to Grosse-Kl\"onne \cite{GK2}.  We will pay particular attention to topological issues. 
\subsubsection{Dagger spaces and weak formal schemes}
\label{assumptions}
We will review, very briefly, basic facts concerning dagger spaces and weak formal schemes. Our main references are \cite{Mer, GK0, Vez},  where the interested reader can find a detailed exposition. 

   We start with dagger spaces. For $\delta\in \mathbb{R}^+$, set
$$
T_n(\delta)=K\{\delta^{-1}X_1,\ldots,\delta^{-1}X_n\}:=\{\sum_v a_vX^v\in K[[X_1,\ldots,X_n]]\mid\lim_{ |v|\to\infty} |a_v|\delta^{|v|}=0\}.
$$
Here $| v|=\sum_{i=1}^{n} v_i$, $v=(v_1,\ldots,v_n)\in\N^n$. We have $T_n:=K\{X_1,\ldots,X_n\}=T_n(1)$. If $\delta\in p^\Q$, this is an affinoid $K$-algebra; the associated Banach norm    $|\cdot|_{\delta}: T_n(\delta)\to \mathbb{R}$, $|\sum a_vX^v|_{\delta}= \max_v |a_v|\delta^{|v|}$. We set
$$
K[X_1,\ldots,X_n]^{\dagger}:=\bigcup_{\delta >1,\delta\in p^\Q}T_n(\delta)=\bigcup_{\delta >1}T_n(\delta)
$$
It is a Hausdorff $LF$-space.

  A {\em dagger algebra} $A$  is a topological $K$-algebra isomorphic to a quotient of the overconvergent Tate algebra $K[X_1,\ldots,X_d]^{\dagger}$. It is canonically a Hausdorff $LF$-algebra \cite[Cor. 3.2.4]{Bam0}. It defines a sheaf of topological $K$-algebras $\so^{\dagger}$ on $\Sp \wh{A}$, $\wh{A}$ being the $p$-adic completion of $A$,  which is called a {\em dagger structure} on $\Sp \wh{A}$. The pair $\Sp(A):=(|\Sp \wh{A}|,\so^{\dagger})$ is called a {\em dagger affinoid}.

A {\em dagger space}\footnote{Sometimes called {\em rigid analytic space with overconvergent structure sheaf}.} $X$ is a pair $(\wh{X},\so^{\dagger})$ where $\wh{X}$ is a rigid analytic space over $K$ and $\so^{\dagger}$ is a sheaf of topological $K$-algebras on $\wh{X}$ such that, for some affinoid open covering $\{\wh{U}_i\to\wh{X}\}$, there are dagger structures $U_i$ on $\wh{U}_i$ such that $\so^{\dagger}|\wh{U}_i\simeq \so^{\dagger}_{U_i}$. The set of global sections $\Gamma(X,\so^{\dagger})$ has a structure of a convex $K$-vector space given by the projective limit $\varprojlim_Y \Gamma(Y,\so^{\dagger}|Y)$, where $Y$ runs over all affinoid  subsets of $X$. In the case of dagger affinoids this agrees with the previous definition.

 Let $X=\Sp(A)\to Y=\Sp(B)$ be a morphism of affinoid dagger spaces and let $U\subset X$ be an affinoid subdomain. We write $U\Subset_Y X$ if there exists a surjection $\tau:B[X_1,\ldots,X_r ]^{\dagger} \to A$ and $\delta\in p^\Q,\delta >1,$ such that $U\subset \Sp(A[\delta^{-1}\tau(X_1),\ldots,\delta^{-1}\tau(X_r)]^{\dagger})$.
A morphism $f:X\to Y$ of dagger (or rigid) spaces is called 
{\em partially proper} if $f$ is separated and if there exist admissible coverings $Y=\bigcup Y_i$ and $f^{-1}(Y_i)=\bigcup X_{ij}$, all $i$, such that for every $X_{ij}$ there exists an  affinoid subset $\wt{X}_{ij}\subset f^{-1}(Y_i)$ with $X_{ij}\Subset_Y\wt{X}_{ij}$. A partially proper dagger space that is quasi-compact is called {\em proper}. This notion is compatible with the one for rigid spaces. In fact, the category of partially proper dagger spaces is equivalent to the category of partially proper rigid spaces \cite[Theorem 2.27]{GK0}. In particular, the   rigid analytification of a finite type scheme over $K$ is partially proper.
   
  A dagger (or rigid) space $X$ is called {\em Stein} if it admits an admissible affinoid covering $X=\bigcup_{i\in \N}U_i $ such that $U_i\subset^{\dagger}U_{i+1}$ for all $i$; we call the covering $U_i,i\in\N$, a Stein covering. Here the notation $U_i\subset^{\dagger} U_{i+1}$ means that the map $\wh{U}_i=\Sp(C)\subset \wh{U}_{i+1}=\Sp(D)$ is an open immersion of affinoid rigid spaces induced  by a map  $D\simeq T_n(\delta)/I\to C\simeq T_n/IT_n $ for some $I$ and $\delta >1$. Stein spaces are partially proper.

  We pass now to weakly formal schemes; the relation between dagger  spaces and weak formal schemes parallels \cite{LZ} the one between rigid spaces and formal schemes due to Raynaud. A {\em weakly complete} $\so_K$-algebra $A^{\dagger}$ (with respect to $(\varpi)$) is an $\so_K$-algebra which  is $\varpi$-adically separated and which satisfies the following condition:  for any power series $f\in\so_K\{X_1,\ldots,X_n\}$,  
$
f=\sum a_vX^v,
$ such that 
there exists a constant $c$ for which  $c(v_p(a_v)+1)\geq |v|$, all $v$,  and for any $n$-tuple $x_1,\ldots,x_n\in A^{\dagger}$, the series $f(x_1,\ldots,x_n)$ converges to an element of $A^{\dagger}$. The {\em weak completion} of an $\so_K$-algebra $A$ is the smallest weakly complete subalgebra $A^{\dagger}$ of $\wh{A}$ containing the image of $A$. 
  
   A {\em weak formal scheme} is a locally ringed space $(X,\so)$ that is locally isomorphic to an affine weak formal scheme. An {\em affine weak formal scheme} is a locally ringed space $(X,\so)$ such that $X=\Spec(A^{\dagger}/\varpi)$  for some weakly complete finitely generated $\so_K$-algebra $A^{\dagger}$ and the sheaf $\so$ is given on the standard basis of open sets by 
$\Gamma(X_{\overline{f}},\so)=(A^{\dagger}_f)^{\dagger}$, $f\in A^{\dagger}$. We say that $X=\Spwf (A^{\dagger})$, the {\em weak formal spectrum} of $A^{\dagger}$. 
For a weak formal scheme $X$, flat over $\so_K$, the associated dagger space $X_K$ is partially proper if and only if  all irreducible closed subsets $Z$ of $X$ are proper over $\so_K$ \cite[Remark 1.3.18]{Hub}.

   A  weak formal scheme  over $\so_K$ is called {\em semistable} if,  locally for the Zariski topology,  it admits \'etale maps to the weak formal spectrum 
    $\Spwf(\so_K[X_1,\ldots,X_n]^{\dagger}/(X_1\cdots X_r-\varpi))$, $1\leq r\leq n$. We equip it with the log-structure coming from the special fiber. We have a similar definition for formal schemes.
    A (weak) formal  scheme $X$ is called {\em Stein} if its generic fiber $X_K$ is Stein. 
It is called {\em Stein with a semistable reduction} if  it has a  semistable reduction over $\so_K$ (and then
the irreducible components of $Y:=X_0$ are proper and smooth) and 
 there exist closed (resp. open) subschemes $Y_s, s\in\N,$ (resp. $U_s, s\in \N$) of $Y$ such that 
\begin{enumerate}
\item each $Y_s$ is a finite union of irreducible components,
\item $Y_s\subset U_s\subset Y_{s+1}$ and their union is $Y$,
\item the tubes  $\{]U_{s}[_{X}\},s\in\N,$ form  a Stein covering of $X_K$.
\end{enumerate}
We will call the covering $\{U_s\},s\in\N,$ a Stein covering of $Y$.
The schemes $U_s, Y_s$ inherit their log-structure from $Y$ (which is canonically a log-scheme log-smooth over $k^{0}$). The log-schemes $Y_s$ are not log-smooth (over $k^0$) but they are ideally log-smooth, i.e., they have a canonical idealized log-scheme structure and are ideally log-smooth for this structure\footnote{Recall \cite{Og1} that an idealized log-scheme is a log-scheme together with an ideal in its log-structure that maps to zero in the structure sheaf. There is a notion of  log-smooth morphism of idealized log-schemes. Log-smooth idealized log-schemes behave like classical log-smooth log-schemes. One can extend the definitions of  log-crystalline,  log-convergent, and log-rigid cohomology, as well as that of  de Rham-Witt complexes to  idealized log-schemes. In what follows we will often skip the word ``idealized'' if understood.}.
\subsubsection{Overconvergent Hyodo-Kato cohomology}
\label{definition}
 Let $X$ be a semistable weak formal scheme over $\so_K$.
We would like to define  the overconvergent Hyodo-Kato cohomology as  the rational overconvergent rigid cohomology of $X_0$ over $\so_F^0$:
  $$\R\Gamma_{\hk}(X_0):=\R\Gamma_{\rig}(X_0/\so_F^0).$$The foundations of log-rigid cohomology  missing \footnote{See however \cite{Sh2}.} this has to be done by hand \cite[1]{GK2}.  

   Let $Y$ be a fine $k^0$-log-scheme. Choose  an open covering $Y=\cup_{i\in I}Y_i$ and, for every $i\in I$,  an exact closed immersion $Y_i\hookrightarrow  Z_i$ into a log-smooth weak formal $\so_F^0$-log-scheme $Z_i$. For each nonempty finite subset $J\subset I$ choose (perhaps after refining the covering) an exactification\footnote{Recall that {\em an exactification} is an operation that turns closed immersions of log-schemes into exact closed immersions.} \cite[Prop. 4.10]{KT}
$$
Y_J=\cap_{i\in J}Y_i\stackrel{\iota}{\to} Z_J\stackrel{f}{\to} \prod_{\so_F^0}(Z_i)_{i\in J}
$$
of the diagonal embedding $Y_J\to \prod_{\so_F^0}(Z_i)_{i\in J}$. Let $\Omega\kr_{Z_J/\so^0_F}$ be the   de Rham complex of the weak formal log-scheme $Z_J$ over $\so_F^0$. This is a complex of sheaves on $Z_J$; tensoring it with $F$ we obtain a  complex of sheaves $\Omega\kr_{Z_{J,F}}$ on the $F$-dagger space $Z_{J,F}$. By \cite[Lemma 1.2]{GK2},  the tube $]Y_J[_{Z_J}$ and the restriction $\Omega\kr_{]Y_J[_{Z_J}}:=\Omega\kr_{Z_{J,F}}|_{]Y_J[_{Z_J}}$ of $\Omega\kr_{Z_{J,F}}$ to $]Y_J[_{Z_J}$ depend only on the embedding system  $\{Y_i\hookrightarrow Z_i\}_i$ not on the chosen exactification $(\iota,f)$.
 Equip the de Rham complex
$
 \Gamma (]Y_J[_{Z_J},\Omega\kr)  
$
 with the topology induced from the structure sheaf of the dagger space $]Y_J[_{Z_J}$.

   For $J_1\subset J_2$, one has  natural  restriction maps $\delta_{J_1,J_2}:]Y_{J_2}[_{Z_{J_2}}\to ]Y_{J_1}[_{Z_{J_1}}$ and  $\delta^{-1}_{J_1,J_2}\Omega\kr_{]Y_{J_1}[_{Z_{J_1}}}\to\Omega\kr_{]Y_{J_2}[_{Z_{J_2}}}$. Well-ordering  $I$, we get  a simplicial dagger space $]Y_{\jcdot}[_{Z_{\jcdot}}$  and a sheaf $\Omega\kr_{]Y_{\jcdot}[_{Z_{\jcdot}}}$ on $]Y_{\jcdot}[_{Z_{\jcdot}}$.
 Consider the complex $\R\Gamma(]Y_{\jcdot}[_{Z_{\jcdot}},\Omega\kr)$. We equip it with the topology induced from the product topology on every cosimplicial level.
  In the  classical derived category of $F$-vector spaces this complex is independent of choices made but we will make everything independent of choices by simply taking limit over all the possible choices. 
  We define a complex in $\sd(C_F)$
\begin{equation}
\label{system}
\R\Gamma_{\rig}(Y/\so_F^0)
:=\hocolim\Gamma(]Y_{\jcdot}[_{Z_{\jcdot}},\Omega\kr),
\end{equation}
where the limit is over the   category of hypercovers built 
from the data that we have described above\footnote{Note that the category of hypercovers, up to a simplicial homotopy, is filtered.  Indeed, since we have fiber products, the issue here is just with equalizers but those exists, up to a simplicial homotopy, by the very general fact \cite[Tag 01GS]{Sta}. Moreover, they induce a homotopy on the corresponding complexes. }. Note that the data corresponding  to affine coverings form a cofinal system. 
We set
$$
\wt{H}^i_{\rig}(Y/\so_F^0):=\wt{H}^i \R\Gamma_{\rig}(Y/\so_F^0),\quad {H}^i_{\rig}(Y/\so_F^0):={H}^i \R\Gamma_{\rig}(Y/\so_F^0).
$$

   The complex $\R\Gamma_{\rig}(Y/\so_F^0)$ is equipped with a  Frobenius endomorphism $\phi$ defined by lifting Frobenius to the schemes $Z_i$ in the above construction. In the case $Y$ is log-smooth over $k^0$ we also  have a monodromy endomorphism\footnote{The formula that follows, while entirely informal, should give the reader an idea about the definition of the monodromy. The formal definition can be found in \cite[formula (37)]{NN}.} $N=\Res(\bigtriangledown(\dlog 0))$ defined by the logarithmic connection satisfying $p\phi N=N\phi$. 
\begin{proposition} 
\label{rain1}
Let $Y$ be a semistable scheme over $k$ with the induced log-structure \cite[2.1]{GK2}. 
\begin{enumerate}
\item If $Y$ is quasi-compact then $H^*_{\rig}(Y/\so_F^0)$ is a finite dimensional $F$-vector space with its unique locally convex Hausdorff topology.
\item The endomorphism $\phi$ on $H^*_{\rig}(Y/\so_F^0)$ is  a  homeomorphism.
\item If $k$ is finite then $H^*_{\rig}(Y/\so_F^0)$ is a mixed $F$-isocrystal, i.e., the eigenvalues of $\phi$  are Weil numbers. 
\end{enumerate}
\end{proposition}
\begin{proof}
All algebraic statements concerning the  cohomology are proved in \cite[Theorem 5.3]{GK2}. They follow immediately from the following weight spectral sequence  \cite[5.2, 5.3]{GK2} that reduces the statements to the analogous ones for (classically) smooth schemes over $k$
\begin{align}
\label{weights}
E^{-k,i+k}_1 & =\bigoplus_{j\geq 0, j\geq -k}\prod _{S\in \Theta_{2j+k+1}}H^{i-2j-k}_{\rig}(S/\so_F)\Rightarrow H^i_{\rig}(X_0/\so_F^0).
\end{align}
Here  $\Theta_j$ denotes the set of all intersections $S$ of $j$ different irreducible components of $X$ that are equipped with trivial log-structure. By assumptions, the intersections $S$ are smooth  over $k$.

 Let us pass to topology.  Recall  the following fact (that we will repeatedly use in the paper)
  \begin{lemma}{\rm (\cite[Lemma 4.7]{GK0}, \cite[Cor. 3.2]{GK1})}\label{czesto}
  Let $Y$ be a smooth Stein space or a smooth affinoid dagger space. All de Rham differentials $d_i:\Omega^i(Y)\to \Omega^{i+1}(Y)$ are strict and have closed images.
  \end{lemma}
  \begin{remark}
  The above lemma holds also for log-smooth Stein spaces with the log-structure given by a normal crossing divisor. The proof in \cite[Cor. 3.2]{GK1} goes through using the fact that for such quasi-compact log-smooth spaces the rigid de Rham cohomology is isomorphic to the rigid de Rham cohomology of the open locus where the log-structure is trivial (hence it is finite dimensional and equipped with the canonical Hausdorff topology).
  \end{remark}
 
   We claim that, in the notation used above, if $Y_J$ is affine, then the complex
 $$
 \Gamma (]Y_J[_{Z_J},\Omega\kr) =\R\Gamma (]Y_J[_{Z_J},\Omega\kr)
$$
has finite dimensional algebraic cohomology $H^*$ whose topology is Hausdorff. Moreover, its cohomology $\wt{H}^*$ is classical. Indeed, note that, using the contracting homotopy of the Poincar\'e Lemma for an open ball, we may assume that the tube $]Y_{J}[_{Z_J}$ is the generic fiber of a weak formal scheme lifting $Y_{J}$ to $\so^0_F$. 
Now,  write $H^i=\ker d_i/\im d_{i-1}$ with the induced quotient topology.  By the above lemma,  the natural map 
 $\coim d_{i-1}\to \im d_{i-1}$ is an isomorphism and $\im d_{i-1}$ is closed in $\ker d_i$.  Hence 
 $\wt{H}^i$ is classical and  $\wt{H}^i\stackrel{\sim}{\to} H^i$ is Hausdorff, as wanted.

   Note that, by the above, a  map between two de Rham complexes associated to two (different) embeddings of $Y_I$ is a strict quasi-isomorphism. This implies that, if $Y$ is affine,  all the arrows in the system (\ref{system}) are strict quasi-isomorphisms and the cohomology of $\R\Gamma_{\rig}(Y/\so_F^0)$ is isomorphic to the cohomology of $\Gamma(]Y_{\jcdot}[_{Z_{\jcdot}},\Omega\kr)$ for any embedding data.

 This proves claim (1) of our proposition for affine schemes; the  case of a general quasi-compact scheme can be treated in the same way (choose a covering by a finite number of affine schemes). Claim (2) follows easily from claim (1). 
 \end{proof}
\begin{remark}
In an analogous way to $\R\Gamma_{\rig}(Y/\so_F^0)$ we define complexes
 $ \R\Gamma_{\rig}(Y/\so_K^{\times})\in \sd(C_K)$.
 For a quasi-compact $Y$, their cohomology groups are classical; they  are finite $K$-vector spaces with their canonical  Hausdorff topology.
\end{remark}

\subsubsection{Overconvergent Hyodo-Kato isomorphism}
Set $r^+:=k[T], r^{\dagger}:=\so_F[T]^{\dagger}$  with the log-structure associated to $T$.
Let $X$ be a log-scheme over $r^+:= k[T]$ (in particular, we allow  log-schemes over $k^0$). Assume that there exists an open covering $X=\cup_{i\in I}X_i$ and, for every $i$, an exact closed immersion $X_i\hookrightarrow  \wt{X}_i$ into a  log-scheme log-smooth over $\tilde{r}:= \so_F[T]$. For each nonempty finite subset $J\subset I$, choose an exactification (product is taken over~$\tilde{r}$)
$$
X_J:=\cap_{i\in J}X_i\stackrel{\iota}{\hookrightarrow}\wt{X}_J\stackrel{f}{\to}\prod_{i\in J}\wt{X}_i
$$
of the diagonal embedding as in Section \ref{definition}.

   Let $\sx_J$ be the weak completion of $\wt{X}_J$. Define the de Rham complex $\Omega\kr_{\sx_J/r^{\dagger}}$ as the weak completion of the de Rham complex $\Omega\kr_{\wt{X}_J/\tilde{r}}$. The tube $]X_J[_{\sx_J}$ with the  complex ($\Omega\kr_{\sx_J/r^{\dagger}}\otimes\Q)|_{]X_J[_{\sx_J}}$ is independent of the chosen factorization $(\iota, f)$. For varying $J$ one has natural transition maps, hence a simplicial dagger space $]X_{\jcdot}[_{\sx_{\jcdot}}$ and a  complex 
\begin{equation}
\label{complx}
(\Omega\kr_{\sx_{\jcdot}/r^{\dagger}}\otimes\Q)|_{]X_{\jcdot}[_{\sx_{\jcdot}}}
\end{equation}
One shows that, in the derived category of vector spaces over $\Q_p$,
$$
\R\Gamma(]X_{\jcdot}[_{\sx_{\jcdot}},\Omega\kr_{\sx_{\jcdot}/r^{\dagger}}\otimes\Q|_{]X_{\jcdot}[_
{\sv_{\jcdot}}})
$$
is independent of choices. We make it though functorial as a complex  by going to limit over all the choices and define a complex in $\sd(C_F)$
$$
\R\Gamma_{\rig}(X/r^{\dagger}):=\hocolim\Gamma(]X_{\jcdot}[_{\sx_{\jcdot}},(\Omega\kr_{\sx_{\jcdot}/r^{\dagger}}\otimes\Q)|_{]X_{\jcdot}[_
{\sx_{\jcdot}}}),
$$
 where the index set runs over the data described above.

  Cohomology $\R\Gamma_{\rig}(X/r^{\dagger})$ is equipped with a  Frobenius endomorphism $\phi$ defined by lifting mod $p$ Frobenius to the schemes $\wt{X}_i$ in the above construction in a manner compatible with the Frobenius on $r^{\dagger}$ induced by $T\mapsto T^p$. 
  If $X$ is log-smooth over $k^0$, we also have a  monodromy endomorphism $N=\Res(\bigtriangledown(\dlog T))$ defined by the logarithmic connection satisfying $p\phi N=N\phi$. The map $p_0:  \R\Gamma_{\rig}(X/r^{\dagger})\to \R\Gamma_{\rig}(X/\so_F^0)$ induced by $T\mapsto 0$ is compatible with Frobenius and monodromy. 
  
     For a general (simplicial) log-scheme with  boundary $(\overline{X},X)$ over $r^+$ that satisfies certain mild condition\footnote{The interested reader can find a description of this condition in \cite[1.10]{GK2}. It  will be always satisfied by the log-schemes we work with in this paper. } the definition of the rigid cohomology $\R\Gamma_{\rig}((\overline{X},X)/r^{\dagger})$ is analogous.    For details of the construction we refer the reader to \cite[1.10]{GK2} and for the definition of log-schemes with boundary to \cite{GKC}.

   Let $X_0$ be a semistable scheme over $k$ with the induced log-structure \cite[2.1]{GK2}. Let $\{X_i\}_{i\in I}$ be the irreducible components of $X_0$ with induced log-structure.  Denote by $M_{\jcdot}$ the nerve of the covering $\coprod_{i\in I}X_i\to X_0$. We define the complex $\R\Gamma_{\rig}(M_{\jcdot}/\so^0_F)\in \sd(C_K)$ in an analogous way to $ \R\Gamma_{\rig}(X_0/\so^0_F)$ using the embedding data described in \cite[1.5]{GK2}.
 
 \begin{lemma}
 \label{Cech}Let $\so$ denote $\so_F^0$ or $\so_K^{\times}$.
 The natural map
 $$
 \R\Gamma_{\rig}(X_0/\so)\to\R\Gamma_{\rig}(M_{\jcdot}/\so)
 $$
 is a strict quasi-isomorphism. 
 \end{lemma}
 \begin{proof}
 It suffices to argue locally, so we may assume that there exists an exact embedding of $X_0$ into a  weak formal scheme $X$ that is log-smooth over $\so$.  
 
  First we prove that above map is a quasi-isomorphism. The complex $\R\Gamma_{\rig}(M_{\jcdot}/\so)$ can be computed by de Rham complexes on the tubes $]M_J[_X$, where, for a nonempty subset $J\subset I$, we set  $M_J=\cap_{j\in J}X_j$ with the induced log-structure.
To compute $\R\Gamma_{\rig}(X_0/\so)$, recall that, for a weak formal scheme $X$ and a closed subscheme $Z$ of its special fiber, if $Z=\cup_{i\in I}Z_i$ is a finite covering by closed subschemes of $Z$, then the dagger space covering $]Z[_X=\cup_{i\in I}]Z_i[_X$ is admissible open \cite[3.3]{GK2}. Hence $\R\Gamma_{\rig}(X_0/\so)$ can be computed as the de Rham cohomology of the nerve of the covering $X_K=\cup_{i\in I}]M_i[_X$. Since the two above mentioned simplicial de Rham complexes are equal, we are done.

  Now,  strictness of the above quasi-isomorphism follows from the fact that  the cohomology groups of the left complex are finite dimensional vector spaces (over $F$ or $K$) with their canonical Hausdorff  topology and so are the cohomology groups of the right complex (basically by the same argument using the quasi-isomorphism $\R\Gamma_{\rig}(M_J/\so)\stackrel{\sim}{\to}\R\Gamma_{\rig}(M_J^{\tr}/\so)$ \cite[Lemma 4.4]{GK2}, where $M^{\tr}_J$ denotes the open set of $M_J$, where the horizontal log-structure is trivial.)
 \end{proof}
  Let $J\subset I$ and $M=M_J=\cap_{j\in J}X_j$. Grosse-Kl\"onne \cite[2.2]{GK2} attaches to $M$ finitely many log-schemes with boundary $(P^{J^{\prime}}_M,V^{J^{\prime}}_M)$, $\emptyset \subsetneq J^{\prime}\subset J$. We think of   $(P^{J^{\prime}}_M,V^{J^{\prime}}_M)$ as the vector bundle $V^{J^{\prime}}_M$ on $M$ (built from the log-structure corresponding to $J^{\prime}$)
  that is compactified by the projective space bundle~$P^{J^{\prime}}_M$. 
   It is a log-scheme with boundary over $r^+$ which, in particular, means that $V^{J^{\prime}}_M$ is a genuine log-scheme over $r^+$ (however this is not the case for $P^{J^{\prime}}_M$). We note, that in application to the Hyodo-Kato isomorphism for $M$  all we need are index sets $J^{\prime}$ with just one element. This construction of Grosse-Kl\"onne corresponds to defining the Hyodo-Kato isomorphism using not the deformation space $r^{\rm PD}_{\varpi}$ as in the classical constructions but its compactification (a projective space). The key  advantage being that the cohomology of the structure sheaf of the new deformation space is now trivial.
   
   The following proposition is the main result of \cite{GK2}.
  \begin{proposition}
  \label{GK1}
  Let $\emptyset\neq J^{\prime}\subset J\subset I$ and let $\so_F(0)=\so_F^0,\so_F(\varpi)=\so_K^{\times}$. The  map
  $$
  \R\Gamma_{\rig}((P^{J^{\prime}}_M,V^{J^{\prime}}_M)/r^{\dagger})\otimes_FF(a)\to\R\Gamma_{\rig}(M/\so_F(a)),\quad a=0,\varpi,
  $$
  defined by restricting to the zero section $M=M_J\to P^{J^{\prime}}_M$ and sending $T\mapsto a$,  is a strict quasi-isomorphism.
  \end{proposition}
  \begin{proof}
  The algebraic quasi-isomorphism was  proved in  \cite[Theorem 3.1]{GK2}. To show that this quasi-isomorphism is strict we can argue locally, for $X_0$ affine. Then the cohomology of the complex on the right is a finite rank vector space over $F(a)$ with its natural locally convex and Hausdorff topology. Algebraic quasi-isomorphism and continuity of the restriction map imply that the cohomology of the complex on the left is Hausdorff as well. Since it is  a locally convex space the map has to be an isomorphism in $C_{F(a)}$, as wanted.
  \end{proof}
  \begin{example}We have found that the best way to understand the above proposition is through an example supplied by Grosse-Kl\"onne himself in \cite{GK2}. Let $X_0$ be of dimension $1$ and let $M$ be the intersection of two irreducible components.  Hence  the underlying scheme  of $M$ is equal to $\Spec k$. Let $U$ be the $2$-dimensional open unit disk over $K$ with coordinates $x_1,x_2$, viewed as a dagger  space. Consider its two closed subspaces:  $U^0$ defined by  $x_1x_2=0$ and $U^{\varpi}$ defined by $x_1x_2=\varpi$. 

Let $\wt{\Omega}_U\kr$ be the de Rham complex of $U$ with log-poles along the divisor $U^0$; let $\Omega_U\kr$ be its quotient by its sub-$\so_U$-algebra generated by $\dlog (x_1x_2)$. Denote by  $\Omega\kr_{U^0}$  and $\Omega\kr_{U^{\varpi}}$ its restriction to $U^0$ and $U^{\varpi}$, respectively. We note that $U^{\varpi}$ is (classically) smooth and that $\Omega\kr_{U^{\varpi}}$ is its (classical) de Rham complex. We view  the $k^0$-log-scheme $M$ as an exact closed log-subscheme of the  formal log-scheme $\Spf(\so_F[[x_1,x_2]]/(x_1x_2))$ that is log-smooth over $\so_F^0$ or of  the formal log-scheme $\Spf(\so_K[[x_1,x_2]]/(x_1x_2-\varpi))$ that is  log-smooth over  $\so_K^{\times}$. The corresponding tubes are  $U^0$ and $U^{\varpi}$. We have
\begin{equation}
\label{GK2}
\R\Gamma_{\rig}(M/\so_F^0)\otimes_FK=\R\Gamma(U^0,\Omega\kr),\quad \R\Gamma_{\rig}(M/\so_K^{\times})=\R\Gamma(U^{\varpi},\Omega\kr).
\end{equation}
We easily  see that  $H^*(U^0,\Omega\kr)\simeq H^*(U^{\varpi},\Omega\kr)$; in particular, $H^1(U^0,\Omega\kr_{U^0})=H^1(U^{\varpi},\Omega\kr)$ is a one dimensional $K$-vector space generated by  $\dlog x_1$. 
  
  The quasi-isomorphism between the cohomologies in (\ref{GK2}) is constructed via the following deformation space $(P,V)$:
  $$P=({\mathbb P}^1_K\times {\mathbb P}^1_K)^{\an}=((\Spec(K[x_1])\cup\{\infty\})\times (\Spec(K[x_2])\cup\{\infty\}))^{\an},\quad V:=(\Spec(K[x_1,x_2]))^{\an}
    $$
 Let $\wt{\Omega}\kr_P$ be the de Rham complex of $P$ with log-poles along the divisor
$$
(\{0\}\times {\mathbb P}^1_K)\cup ({\mathbb P}^1_K\times\{0\})\cup (\{\infty\}\times {\mathbb P}^1_K)\cup ({\mathbb P}^1_K\times\{\infty\}).
$$
The section $\dlog(x_1x_2)\in \wt{\Omega}^1_U(U)=\wt{\Omega}^1_P(U)$ extends canonically to a section $\dlog(x_1x_2)\in \wt{\Omega}^1_P(P)$.  Let $\Omega\kr_P$ be the quotient of $\wt{\Omega}\kr_P$ by its sub-$\so_P$-algebra generated by $\dlog (x_1x_2)$. The natural restriction maps 
$$
\R\Gamma(U^0,\Omega\kr)\leftarrow \R\Gamma(P,\Omega\kr)\to \R\Gamma(U^{\varpi},\Omega\kr), \quad  0\mapsfrom T, T\mapsto \varpi,
$$
are  quasi-isomorphisms.  This is because we have killed one differential of $\wt{\Omega}_P$ and the logarithmic differentials of ${\mathbb P}^1_
{\so_K}$ are isomorphic to the structure sheaf hence have cohomology which is $1$-dimensional in degree $0$ and trivial otherwise.
 \end{example}

   Varying the index set $J^{\prime}$ in a coherent way one glues the log-schemes $(P^{J^{\prime}}_M,V^{J^{\prime}}_M)$ into a simplicial $r^+$-log-scheme $ (P_{\jcdot},V_{\jcdot})$ with boundary. 
     Set $\R\Gamma_{\rig}(\overline{X}_0/r^{\dagger}):=\R\Gamma_{\rig}((P_{\jcdot},V_{\jcdot})/r^{\dagger})$.
       We have  the corresponding simplicial log-scheme  $
 M^{\prime}_{\jcdot}
 $ over $k^0$.
There is a natural map $M_{\jcdot}\to M^{\prime}_{\jcdot}$ (that induces a strict quasi-isomorphism 
 $\R\Gamma_{\rig}(M^{\prime}_{\jcdot}/\so_F^0)
 \stackrel{\sim}{\to}\R\Gamma_{\rig}(M_{\jcdot}/\so_F^0)$) and a natural map
  $ M^{\prime}_{\jcdot} \to (P_{\jcdot},V_{\jcdot})$.
  The following proposition is an immediate corollary of Proposition \ref{GK1} and Lemma \ref{Cech}. 
  \begin{proposition}\label{Lyonnn}{\rm (\cite[Theorem 3.4]{GK2})} Let $a=0,\varpi$. 
  The natural maps
  $$
        \R\Gamma_{\rig}(X_0/\so_F(a))\to \R\Gamma_{\rig}(M^{\prime}_{\jcdot}/\so_F(a))
        \leftarrow\R\Gamma_{\rig}(\overline{X}_0/r^{\dagger})\otimes_FF(a)
  $$ 
  are strict quasi-isomorphisms. 
  \end{proposition}
  
  Let $X$ be a semistable weak formal scheme over $\so_K$. We define the overconvergent Hyodo-Kato cohomology of $X_0$ as $\R\Gamma_{\hk}(X_0):=\R\Gamma_{\rig}(X_0/\so_F^0)$.
  Recall that the Hyodo-Kato map $$\iota_{\hk}: \R\Gamma_{\hk}(X_0)\to\R\Gamma_{\dr}(X_K)$$
  is defined as  the zigzag (using the maps from the above proposition)
  $$\R\Gamma_{\hk}(X_0)=\R\Gamma_{\rig}(X_0/\so_F^0)\stackrel{\sim}{\leftarrow} \R\Gamma_{\rig}(\overline{X}_0/r^{\dagger}){\to} \R\Gamma_{\rig}(\overline{X}_0/r^{\dagger})\otimes_FK\stackrel{\sim}{\to}
\R\Gamma_{\rig}(X_0/\so_K^{\times})\simeq \R\Gamma_{\dr}(X_K).
$$
It yields  the (overconvergent) Hyodo-Kato strict quasi-isomorphism $$ \iota_{\hk}: \R\Gamma_{\hk}(X_0)\otimes_F K\stackrel{\sim}{\to}\R\Gamma_{\dr}(X_K).
$$
\begin{remark}
The overconvergent Hyodo-Kato map, as its classical counterpart, depends on the choice of the uniformizer $\varpi$. This dependence  takes the usual form \cite[Prop. 4.4.17]{Ts}.
\end{remark}
  
\subsection{Overconvergent  syntomic cohomology}
In this section we will define syntomic cohomology (a la Bloch-Kato) using overconvergent Hyodo-Kato and de Rham cohomologies of Grosse-Kl\"onne and discuss the fundamental diagram that it fits into.  We call this definition {"a la Bloch-Kato"} because it is inspired by Bloch-Kato's definition of local Selmer groups \cite{BKS}. 
\subsubsection{Period rings $\wh{\B}^+_{\st}$, $\A_{\crr,K}$}   We will recall the definition of  the  rings of periods $\wh{\B}^+_{\st}$ and $\A_{\crr,K}$ that we will need.
We denote by $r^+_\varpi$ the algebra $\so_F[[T]]$
with the log-structure associated to $T$.  Sending $T$ to $\varpi$ induces
a surjective morphism $r_\varpi^+\to \so_K^{\times}$.    We denote by $r_\varpi^{\rm PD}$ the $p$-adic divided power envelope of $r_\varpi^+$ with respect to the kernel of this morphism. Frobenius is defined by $T\mapsto T^p$, monodromy by $T\mapsto T$.

   We start with the definition of  the  ring of periods $\wh{\B}^+_{\st}$ \cite[p.253]{Ts}.
   Let
$$
\wh{\A}_{\st,n}:=H^0_{\crr}(\so_{C,n}^{\times}/r^{\rm PD}_{\varpi,n}), \quad \wh{\A}_{\st}:=\varprojlim_n\wh{\A}_{\st,n},\quad   \wh{\B}^+_{\st}:=\wh{\A}_{\st}[1/p]. $$
We note that  $\wh{\B}^+_{\st}$ is a Banach space over $F$ (which makes it easier to handle topologically than $\B^+_{\st}$).
The ring $\wh{\A}_{\st,n}$ has a natural action of $\sg_K$, Frobenius $\phi $,
and a monodromy
operator $N$.
We have a  morphism $\A_{\crr,n}\to \wh{\A}_{\st,n}$ induced by the map
$H^0_{\crr}(\so_{C,n}/\so_{F,n})\to H^0_{\crr}(\so_{C,n}^{\times}/r^{\rm PD}_{\varpi,n})$. Both it and 
the natural map $r^{\rm PD}_{\varpi,n}\to \wh{\A}_{\st,n}$ are compatible with all the structures (Frobenius, monodromy, and Galois action). Moreover, we have the exact sequence 
\begin{equation}
\label{Breuil}
0\to \A_{\crr,n}\to \wh{\A}_{\st,n}\stackrel{N}{\to} \wh{\A}_{\st,n}\to  0.
\end{equation}

  We can view $\wh{\A}_{\st,n}$
   as the ring of the PD-envelope of the closed immersion
   $$
   \Spec \so_{C,n}^{\times}\hookrightarrow \Spec (\A^{\times}_{\crr,n}\otimes_{\so_{F,n}}r^+_{\varpi,n})
   $$
   defined by the maps $\theta: \A_{\crr,n}\to \so_{C,n}$ and  $r^+_{\varpi,n} \to \so_{K,n}$, $T\mapsto \varpi$. Here $\A^{\times}_{\crr,n}$ is $\A_{\crr,n}$ equipped with the unique log-structure extending the one on $\so^{\times}_{C,n}$. 
   This makes $\Spec \so_{C,1}^{\times}\hookrightarrow \Spec \wh{\A}_{\st,n}$ into a PD-thickening in the crystalline site of
   $\so_{\overline{K},1}^{\times}$. It follows \cite[Sec. 3.9]{Ka3} that
    \begin{equation}
\label{isom1}
\widehat{\A}_{\st,n}\simeq \R\Gamma_{\crr}(\so^{\times}_{C,n}/r^{\rm PD}_{\varpi,n}).
\end{equation}
There is a canonical $\B^+_{\crr}$-linear isomorphism 
$
\B^+_{\st}\stackrel{\sim}{\to}\wh{\B}_{\st}^{+,N-{\rm nilp}}
$
compatible \cite[Theorem 3.7]{Ka3} with the action of $\sg_K$, $\phi$, and $N$.

  We will now pass to the definition of  the  ring of periods $\A_{\crr,K}$ \cite[4.6]{Ts}.
   Let
$$
{\A}_{\crr,K,n}:=H^0_{\crr}(\so_{C,n}^{\times}/\so_{K,n}^{\times}), \quad {\A}_{\crr,K}:=\varprojlim_n{\A}_{\crr,K,n}.$$
The ring ${\A}_{\crr,K,n}$ is a flat $\Z/p^n$-module and ${\A}_{\crr,K,n+1}\otimes\Z/p^n\simeq {\A}_{\crr,K,n}$; moreover, it has a natural action of $\sg_K$. These properties generalize to $H^0_{\crr}(\so_{C,n}^{\times}/\so_{K,n}^{\times},\sj^{[r]})$, for $r\in\Z$, and we have $H^i_{\crr}(\so_{C,n}^{\times}/\so_{K,n}^{\times},\sj^{[r]})=0$, $i\geq 1$, $r\in\Z$. Set $$
F^r{\A}_{\crr,K,n}:=H^0_{\crr}(\so_{C,n}^{\times}/\so_{K,n}^{\times},\sj^{[r]}),\quad F^r{\A}_{\crr,K}:=\varprojlim_nF^r{\A}_{\crr,K,n}.
$$
We have
$$
F^r{\A}_{\crr,K,n}\simeq \rg_{\crr}(\so^{\times}_{C,n}/\so^{\times}_{K,n},\sj^{[r]}), \quad F^r{\A}_{\crr,K,n}/F^s\simeq \rg_{\crr}(\so^{\times}_{C,n}/\so^{\times}_{K,n},\sj^{[r]}/\sj^{[s]}), \quad r\leq s.
$$
The natural map $\gr^r_F\A_{\crr,n}\to\gr^r_F\A_{\crr,K,n}$ is a $p^a$-quasi-isomorphism for a constant $a$ depending on $K$, $a\sim v_p(d_{K/F})$,  \cite[Lemma 4.6.2]{Ts}.
We set $\B^+_{\crr,K}:=\A_{\crr,K}[1/p]$. There is  a natural $\sg_K$-equivariant map $\iota:\wh{\B}^+_{\st}\to
  \B_{\dr}^+$ induced by the maps 
$$
p_{\varpi}:\widehat{\B}^+_{\st}
 \to     \B^+_{\crr,K},\quad 
\B^+_{\crr}/F^r\stackrel{\sim}{\to}\B^+_{\crr,K}/F^r\stackrel{\sim}{\leftarrow}\B^+_{\crr}/F^r,
$$  
where  $p_{\varpi}$ denotes the map induced by sending $T\mapsto \varpi$. The composition ${\B}^+_{\st}{\to}
\widehat{\B}^+_{\st}\stackrel{\iota}{\to}\widehat{\B}^+_{\dr}$ is the map $\iota=\iota_{\varpi}$ from Section \ref{Notation}. 

\subsubsection{Overconvergent  geometric syntomic cohomology}
\label{system11}
Let $X$ be a semistable weak formal scheme over $\so_K$. 
 Take $r\geq 0$. 
We define the overconvergent geometric syntomic cohomology of $X_K$ by the following mapping fiber (taken in $\sd(C_{\Q_p})$)
$$
\R\Gamma_{\synt}(X_{C},\Q_p(r)):=[[\R\Gamma_{\hk}(X_0)\wh{\otimes}_{F}^R\wh{\B}^+_{\st}]^{N=0,\phi=p^r}
\verylomapr{\iota_{\hk}\otimes\iota}(\R\Gamma_{\dr}(X_K)\wh{\otimes}_K^R\B^+_{\dr})/F^r].
$$
This is an overconvergent analog of the algebraic geometric syntomic cohomology studied in \cite{NN}. 
Here, we wrote $[\R\Gamma_{\hk}(X_0)\wh{\otimes}_{F}^R\wh{\B}^+_{\st}]^{N=0,\phi=p^r}$ for the homotopy limit of the commutative diagram\footnote{In general, in what follows we will use the brackets $[\ ]$ to denote derived eigenspaces and the brackets $(\ )$ or nothing to denote the non-derived ones.}
$$
\xymatrix{
\R\Gamma_{\hk}(X_0)\wh{\otimes}_{F}^R\wh{\B}^+_{\st} \ar[r]^{\phi-p^r} \ar[d]^N& \R\Gamma_{\hk}(X_0)\wh{\otimes}_{F}^R\wh{\B}^+_{\st}\ar[d]^N\\
\R\Gamma_{\hk}(X_0)\wh{\otimes}_{F}^R\wh{\B}^+_{\st}  \ar[r]^{p\phi-p^r} & \R\Gamma_{\hk}(X_0)\wh{\otimes}_{F}^R\wh{\B}^+_{\st}.
}
$$
The filtration on $\R\Gamma_{\dr}(X_K)\wh{\otimes}_K^R\B^+_{\dr}$ is defined by the formula
$$
F^r(\R\Gamma_{\dr}(X_K)\wh{\otimes}_K^R\B^+_{\dr}):=\hocolim_{i+j\geq r}F^i\R\Gamma_{\dr}(X_K)\wh{\otimes}_K^RF^j\B^+_{\dr}.
$$
Set
  $$ 
  {\rm HK}(X_{C},r):=[\R\Gamma_{\hk}(X_0)\wh{\otimes}_F^R\wh{\B}^+_{\st}]^{N=0,\phi=p^r},\quad {\rm DR}(X_{C},r):=(\R\Gamma_{\dr}(X_K)\wh{\otimes}_K^R\B^+_{\dr})/F^r.
  $$
  Hence
  $$
  \R\Gamma_{\synt}(X_{C},\Q_p(r))=[{\rm HK}(X_{C},r)\verylomapr{\iota_{\hk}\otimes\iota}{\rm DR}(X_{C},r)].
  $$
  
 \begin{example}Assume that $X$ is quasi-compact. 
\label{LF-tensor}We claim that then  the complex
 $$
 \R\Gamma _{\hk}(X_{0})\wh{\otimes}_F^R\wh{\B}^+_{\st}  
$$
has classical cohomology isomorphic  to $H^*_{\hk}(X_{0})\wh{\otimes}_F\wh{\B}^+_{\st}$, a finite rank free module  over $\wh{\B}^+_{\st}$. To show this, consider the distinguished triangle
$$
 H^0_{\hk}(X_0)\to \R\Gamma _{\hk}(X_{0})\to \tau_{\geq 1}\R\Gamma _{\hk}(X_{0}).
$$
Tensoring it with $\wh{\B}^+_{\st}$ we obtain the distinguished triangle
\begin{equation}
\label{warszawa}
 H^0_{\hk}(X_0)\wh{\otimes}_F^R\wh{\B}^+_{\st}\to \R\Gamma _{\hk}(X_{0})\wh{\otimes}_F^R\wh{\B}^+_{\st}\to (\tau_{\geq 1}\R\Gamma _{\hk}(X_{0}))\wh{\otimes}_F^R\wh{\B}^+_{\st}.
\end{equation}
Note that, since $H^0_{\hk}(X_0)$ is a finite rank vector space over $F$,  we have the natural strict quasi-isomorphism
\begin{equation}
\label{herbata1}
H^0_{\hk}(X_0)\wh{\otimes}_F\wh{\B}^+_{\st}\stackrel{\sim}{\to} H^0_{\hk}(X_0)\wh{\otimes}_F^R\wh{\B}^+_{\st}.
\end{equation}
Since $\wt{H}^0((\tau_{\geq 1}\R\Gamma _{\hk}(X_{0}))\wh{\otimes}_F^R\wh{\B}^+_{\st})=0$, this implies that 
\begin{align*}
 & H^0_{\hk}(X_0)\wh{\otimes}_F\wh{\B}^+_{\st}\stackrel{\sim}{\to}\wt{H}^0(\R\Gamma _{\hk}(X_{0})\wh{\otimes}_F^R\wh{\B}^+_{\st}),\\
& \wt{H}^i(\R\Gamma _{\hk}(X_{0})\wh{\otimes}_F^R\wh{\B}^+_{\st})\stackrel{\sim}{\to} \wt{H}^i((\tau_{\geq 1}\R\Gamma _{\hk}(X_{0}))\wh{\otimes}_F^R\wh{\B}^+_{\st}),\quad i\geq 1.
\end{align*}
Repeating now the above computation for $(\tau_{\geq 1}\R\Gamma _{\hk}(X_{0}))\wh{\otimes}_F^R\wh{\B}^+_{\st}$, $(\tau_{\geq 2}\R\Gamma _{\hk}(X_{0}))\wh{\otimes}_F^R\wh{\B}^+_{\st}$, etc, we get that
\begin{equation}
\label{gabriel}
 \wt{H}^i(\R\Gamma _{\hk}(X_{0})\wh{\otimes}_F^R\wh{\B}^+_{\st}  )\simeq H^i_{\hk}(X_{0})\wh{\otimes}_F\wh{\B}^+_{\st},\quad i\geq 0,
\end{equation}
as wanted. 

\begin{lemma}
\label{paris15}
Let $X$ be quasi-compact.
 The above isomorphism induces a natural  isomorphism
$$
\wt{H}^i({\rm HK}(X_C,r))\simeq({H}^i_{\hk}(X_0)\wh{\otimes}_F\wh{\B}^+_{\st})^{N=0,\phi=p^r}, \quad i\geq 0,
$$
of Banach spaces (so $\wt{H}^i({\rm HK}(X_C,r))$ is classical).
\end{lemma}
\begin{proof}
The argument here is  similar to the one given in \cite[Cor. 3.26]{NN} for the Beilinson-Hyodo-Kato cohomology but requires a little bit more care. 
We note that  ${H}^i_{\hk}(X_0)$ is a finite dimensional $(\phi,N)$-module (by Proposition \ref{rain1}). 
For a finite $(\phi, N)$-module $M$,   we  have the following short exact sequences 
 \begin{align}
 \label{seqq}
 0\to & M\otimes_{F}\B^+_{\crr}\stackrel{\beta}{\to} M\otimes_{F}\wh{\B}^+_{\st}\stackrel{N}{\to} M\otimes_{F}\wh{\B}^+_{\st}\to 0,\\
0\to & (M\otimes_{F}\B^+_{\crr})^{\phi=p^r}\to M\otimes_{F}\B^+_{\crr}\lomapr{p^r-\phi} M\otimes_{F}\B^+_{\crr}\to 0.\notag
\end{align}
The first one follows, by induction on $m$ such that $N^m=0$ on $M$, from the fundamental exact sequence, i.e.,   the same sequence for $M=F$. The map 
$\beta$ is the (Frobenius equivariant) trivialization map 
 defined as follows
\begin{align}\label{trivialization}
\beta: M\otimes\B^+_{\crr}  \stackrel{\sim}{\to} (M\otimes\wh{\B}^+_{\st})^{N=0},\quad
m\otimes b & \mapsto \exp(Nu)m\otimes b.
\end{align}
We note here that it is not Galois equivariant;  however this  fact will not be a problem for us in this proof.  The second exact sequence follows from
\cite[Remark 2.30]{CN}.

  We will first show that
   \begin{equation}
   \label{gabriel1}
\wt{H}^i([\R\Gamma_{\hk}(X_0)\wh{\otimes}_F^R\wh{\B}^+_{\st}]^{N=0})\simeq({H}^i_{\hk}(X_0)\wh{\otimes}_F\wh{\B}^+_{\st})^{N=0}.
\end{equation}
Set ${\hk}:=\R\Gamma_{\hk}(X_0)\wh{\otimes}_F^R\wh{\B}^+_{\st}$. We have the long exact sequence
$$
\stackrel{N}{\to}\wt{H}^{i-1}({\hk})\to \wt{H}^i([{\hk}]^{N=0})\to \wt{H}^i({\hk}^{})\stackrel{N}{\to} \wt{H}^i({\hk})\to \wt{H}^{i+1}([{\hk}]^{N=0})\to 
$$
By the isomorphism (\ref{gabriel}) and the exact sequence (\ref{seqq}), it splits into the short exact sequences
$$
0\to \wt{H}^i([{\hk}]^{N=0})\to H^i_{\hk}(X_0)\otimes_F\wh{\B}^+_{\st}\stackrel{N}{\to}H^i_{\hk}(X_0)\otimes_F\wh{\B}^+_{\st}\to 0
$$
The isomorphism (\ref{gabriel1}) follows. By (\ref{seqq}), we also have 
 $\wt{H}^i([{\hk}]^{N=0})\simeq H^i_{\hk}(X_0)\otimes_F\B^+_{\crr}$. 
 
 Now, set $D:=[{\hk}]^{N=0}$. We have the long exact sequence
 $$
\lomapr{\phi-p^r} \wt{H}^{i-1}(D)\to \wt{H}^{i}([D]^{\phi=p^r})\to \wt{H}^{i}(D)\lomapr{\phi-p^r} \wt{H}^{i}(D)\to \wt{H}^{i+1}([D]^{\phi=p^r})\to
 $$ Since 
 $\wt{H}^i(D)\simeq H^i_{\hk}(X_0)\otimes_F\B^+_{\crr}$, 
 the sequence   (\ref{seqq}) implies that the above long exact sequence  splits into the short exact sequences
 $$
 0\to  \wt{H}^{i}([D]^{\phi=p^r})\to H^i_{\hk}(X_0)\otimes_F\B^+_{\crr}\lomapr{\phi-p^r}H^i_{\hk}(X_0)\otimes_F\B^+_{\crr}\to 0
 $$
Our lemma follows from the sequence in (\ref{seqq}).
\end{proof}

  Assume now that $X$ is Stein and let $\{U_n\}, n\in \N,$ be  a Stein covering. We claim that we have a natural strict quasi-isomorphism
\begin{equation}
\label{warszawa0}
\R\Gamma_{\hk}(X_0)\wh{\otimes}_F^R\wh{\B}^+_{\st}\stackrel{\sim}{\to} \holim_n(\R\Gamma_{\hk}(U_{n,0})\wh{\otimes}_F^R\wh{\B}^+_{\st}).
\end{equation}
To show this we will compute the cohomology of both sides. Since, by Lemma \ref{bios2}, the natural map
$$
H^i_{\hk}(X_0)\wh{\otimes}_F\wh{\B}^+_{\st}\to H^i_{\hk}(X_0)\wh{\otimes}_F^R\wh{\B}^+_{\st},\quad i\geq 0,
$$
is a strict quasi-isomorphism,  the argument in Example \ref{LF-tensor} goes through and we get that
$$
\wt{H}^i(\R\Gamma_{\hk}(X_0)\wh{\otimes}_F^R\wh{\B}^+_{\st})\simeq H^i_{\hk}(X_0)\wh{\otimes}_F\wh{\B}^+_{\st},\quad i\geq 0.
$$

Similarly, applying $\holim$ to the analogs of the distinguished triangle (\ref{warszawa}), we get the distinguished triangles
$$
 \holim_n(H^0_{\hk}(U_{n,0})\wh{\otimes}_F^R\wh{\B}^+_{\st})\to \holim_n(\R\Gamma _{\hk}(U_{n,0})\wh{\otimes}_F^R\wh{\B}^+_{\st})\to \holim_n((\tau_{\geq 1}\R\Gamma _{\hk}(U_{n,0}))\wh{\otimes}_F^R\wh{\B}^+_{\st}).
$$
We have 
\begin{align}
\label{warszawa1}
& \wt{H}^0\holim_n(H^0_{\hk}(U_{n,0})\wh{\otimes}_F^R\wh{\B}^+_{\st})\simeq  \invlim_n(H^0_{\hk}(U_{n,0})\wh{\otimes}_F\wh{\B}^+_{\st}),\\
 & \wt{H}^i\holim_n(H^0_{\hk}(U_{n,0})\wh{\otimes}_F^R\wh{\B}^+_{\st})\simeq  \wt{H}^i\holim_n(H^0_{\hk}(U_{n,0})\wh{\otimes}_F\wh{\B}^+_{\st})=0,\quad i\geq 1.\notag
\end{align}
The vanishing in the second line can be seen by evoking  the Mittag-Leffler criterium in the category of convex vector spaces \cite[Cor. 2.2.12]{Pr}:
the projective system  $\{H^0_{\hk}(U_{n,0})\wh{\otimes}_F\wh{\B}^+_{\st}\}$ is Mittag-Leffler because so is the projective system $\{H^0_{\hk}(U_{n,0})\}$. 
Now, the argument in Example \ref{LF-tensor} can be repeated and it will yield that
$$
\wt{H}^i(\holim_n(\R\Gamma_{\hk}(U_{n})\wh{\otimes}_F^R\B^+_{\st}))\simeq \invlim_n(H^i_{\hk}(U_{n,0})\wh{\otimes}_F\wh{\B}^+_{\st}),\quad i\geq 0.
$$

 The computations of cohomology of both sides of (\ref{warszawa0}) being compatible, to prove that they are strictly quasi-isomorphic, it  remains  to show that the natural map
$$
(\invlim_nH^i_{\hk}(U_{n,0}))\wh{\otimes}_F\wh{\B}^+_{\st}\to \invlim_n(H^i_{\hk}(U_{n,0})\wh{\otimes}_F\wh{\B}^+_{\st})
$$
is an isomorphism. But this follows from the fact that each $H^i_{\hk}(U_{n,0})$ is a finite rank vector space (see Section \ref{tensor}).

    To sum up the above discussion: 
    \begin{lemma}
The cohomology of $\R\Gamma_{\hk}(X_0)\wh{\otimes}_F^R\wh{\B}^+_{\st}$ is classical and we have
\begin{equation}
\label{india1}
\wt{H}^i(\R\Gamma_{\hk}(X_0)\wh{\otimes}_F^R\wh{\B}^+_{\st})\simeq {H}^i_{\hk}(X_0)\wh{\otimes}_F\wh{\B}^+_{\st}\simeq \varprojlim_n({H}^i_{\hk}(U_{n,0})\wh{\otimes}_F\wh{\B}^+_{\st}).
\end{equation}
\end{lemma}

\begin{lemma}
\label{Lyon-again}
The cohomology $\wt{H}^i([\R\Gamma_{\hk}(X_0)\wh{\otimes}_F^R\wh{\B}^+_{\st}]^{N=0,\phi=p^r})$ is classical and we have  natural isomorphisms
$$H^i([\R\Gamma_{\hk}(X_0)\wh{\otimes}_F^R\wh{\B}^+_{\st}]^{N=0,\phi=p^r})\simeq (H^i_{\hk}(X_0)\wh{\otimes}_F\wh{\B}^+_{\st})^{N=0,\phi=p^r}.
$$
In particular, the space $H^i([\R\Gamma_{\hk}(X_0)\wh{\otimes}_F^R\wh{\B}^+_{\st}]^{N=0,\phi=p^r})$ is Fr\'echet.  Moreover,    
  $$\wt{H}^i([\rg_{\hk}(X_0)\wh{\otimes}_F^R\wh{\B}^+_{\st}]^{N=0})\simeq (H^i_{\hk}(X_0)\wh{\otimes}_F\wh{\B}^+_{\st})^{N=0}
\simeq H^i_{\hk}(X_0)\wh{\otimes}_F\B^+_{\crr},
$$
where the last isomorphism is not, in general, Galois equivariant.
\end{lemma}
\begin{proof}For the first claim, 
we argue  as in the proof of Lemma \ref{paris15} using analogs of the exact sequences (\ref{seqq}) (for $M=H^i_{\hk}(X_0)$):
 \begin{align}
 \label{seqq1}
 0\to & H^i_{\hk}(X_0)\wh{\otimes}_{F}\B^+_{\crr}\stackrel{\beta}{\to} H^i_{\hk}(X_0)\wh{\otimes}_{F}\wh{\B}^+_{\st}\stackrel{N}{\to} H^i_{\hk}(X_0)\wh{\otimes}_{F}\wh{\B}^+_{\st}\to 0,\\
0\to & (H^i_{\hk}(X_0)\wh{\otimes}_{F}\B^+_{\crr})^{\phi=p^r}\to H^i_{\hk}(X_0)\wh{\otimes}_{F}\B^+_{\crr}\lomapr{p^r-\phi} H^i_{\hk}(X_0)\wh{\otimes}_{F}\B^+_{\crr}\to 0.\notag
\end{align}
These sequences are limits of sequences  (\ref{seqq}) applied to $H^i_{\hk}(U_{n,0})$, $n\in\N$. We wrote $\beta:=\varprojlim_n\beta_n$ and used the isomorphism (\ref{india1}) (and its analog for $\B^+_{\crr}$) as well as the vanishing of 
$\wt{H}^j\holim_n(H^i_{\hk}(U_{n,0})\wh{\otimes}_F\B^+_{\crr})$ and $\wt{H}^j\holim_n((H^i_{\hk}(U_{n,0})\wh{\otimes}_F\B^+_{\crr})^{\phi=p^r})$ for $j\geq 1$. The vanishing of the first cohomology follows from the fact that the projective system $\{H^i_{\hk}(U_{n,0})\}$ is Mittag-Leffler. The vanishing of the second cohomology is a little subtler. Note that the system of  Banach spaces $\{(H^i(U_{n,0})\wh{\otimes}_F\B^+_{\crr})^{\phi=p^r}\}$ can be lifted to a system of finite dimensional BC spaces
with Dimensions $(d_i,h_i)$, $d_i,h_i\geq 0$ \cite[Prop. 10.6]{CB}. 
The images of the terms in the system in a fixed BC space form a chain with decreasing dimensions $D$ 
(in lexicographical order). Since the height $h$ of any BC subspace of these spaces 
is also $\geq 0$ \cite[Lemma 2.6]{fBC}, they stabilize. Hence the original system satisfies the Mittag-Leffler criterium from \cite[Cor. 2.2.12]{Pr} and, hence, it is acyclic.

The last claim of the lemma was proved while proving the first claim.

\end{proof}
\end{example}
 \begin{example}  
 \label{derham}Assume that $X$ is affine or Stein.  In that case $\Omega^i(X_K)$ is an $LB$-space or a Fr\'echet space, respectively. The de Rham cohomology  $\wt{H}^i_{\dr}(X_K)$ is classical; it is a finite dimensional $K$-vector space with its natural Hausdorff topology or a Fr\'echet space, respectively.
 
\smallskip
  $\bullet$ {\em Assume first that $X$ is Stein}. We claim that (in $\sd(C_K)$)
    \begin{align}
    \label{derham11}
     F^r(\R\Gamma_{\dr}(X_K)  \wh{\otimes}^R_K\B^+_{\dr}) &  \simeq F^r(\Omega\kr({X_K})\wh{\otimes}_{K}\B^+_{\dr})
\\\notag & =(\so({X_K})\wh{\otimes}_K F^r\B^+_{\dr}\to\Omega^1({X_K})\wh{\otimes}_K F^{r-1}\B^+_{\dr}\to \cdots)\\\notag
      {\rm DR}(X_{C},r)    = (\R\Gamma_{\dr}(X_K) & \wh{\otimes}^R_K  \B^+_{\dr})/F^r  \simeq (\Omega\kr({X_K})\wh{\otimes}_{K}\B^+_{\dr})/F^r\\  
   = (\so({X_K}) & \wh{\otimes}_K (\B^+_{\dr}/F^r)\to
      \Omega^1({X_K})\wh{\otimes}_K(\B^+_{\dr}/F^{r-1})\to\cdots 
        \to  \Omega^{r-1}({X_K})\wh{\otimes}_K(\B^+_{\dr}/F^{1})).\notag
  \end{align}

  In low degrees we have
\begin{align*}
{\rm DR}(X_{C},0) &=0,\quad 
      {\rm DR}(X_{C},1)\simeq \so({X_K})\wh{\otimes}_K C,\\
      {\rm DR}(X_{C},2) & \simeq (\so({X_K})\wh{\otimes}_K(\B^+_{\dr}/F^2)\to
      \Omega^1({X_K})\wh{\otimes}_K C).
      \end{align*}

  To prove the first strict quasi-isomorphism in (\ref{derham11}), it suffices to show that the natural map
  $$
 \Omega\kr(X_K)\wh{\otimes}_K\B^+_{\dr} \to \Omega\kr(X_K)\wh{\otimes}^R_K\B^+_{\dr} 
  $$
  is a strict quasi-isomorphism. Or that so is the map
  \begin{equation}
  \label{bios3}
   \Omega^i(X_K)\wh{\otimes}_K\B^+_{\dr} \to \Omega^i(X_K)\wh{\otimes}^R_K\B^+_{\dr}, \quad i\geq 0.
\end{equation}
  But this follows from  Lemma \ref{bios2} since both $ \Omega^i(X_K)$  and $\B^+_{\dr}$
  are  Fr\'echet spaces.

  To prove the second strict quasi-isomorphism above, i.e., the natural strict quasi-isomorphism
  $$
F^r(\Omega\kr(X_K)\wh{\otimes}_K\B^+_{\dr} )\stackrel{\sim}{\to} \hocolim_{i+j\geq r}(F^i\Omega\kr(X_K)\wh{\otimes}^R_KF^j\B^+_{\dr}),
  $$
  since the inductive limit is, in fact, finite, it suffices to show that the natural map
  $$
 F^i\Omega\kr(X_K)\wh{\otimes}_KF^j\B^+_{\dr}\to F^i\Omega\kr(X_K)\wh{\otimes}^R_KF^j\B^+_{\dr}
$$
  is a strict quasi-isomorphism and this follows from the strict quasi-isomorphism (\ref{bios3}).
  
  Recall that the de Rham complex is built from Fr\'echet spaces and it has strict differentials.    The complex   ${\rm DR}(X_{C},r)$ is a complex of Fr\'echet spaces as well. Its differentials are also strict:
  write the $i$'th differential as a composition
  \begin{equation}
  \label{fixed}
  \Omega^i(X_K)\wh{\otimes}_K\B^+_{\dr}/F^{r-i}\verylomapr{\id\otimes \can}  \Omega^i(X_K)\wh{\otimes}_K\B^+_{\dr}/F^{r-i-1} \verylomapr{d_i\otimes \id}\Omega^{i+1}(X_K)\wh{\otimes}_K\B^+_{\dr}/F^{r-i-1}.
  \end{equation}
  Since $\Omega^i(X_K)$ is a Fr\'echet space and $\B^+_{\dr}/F^{i}$ is a Banach space the first map is  surjective and strict (we use here point (3) from Section \ref{tensor})). The second map is induced from the differential $d_i: \Omega^i(X_K) \to \Omega^{i+1}(X_K)$, which is strict, hence it is strict since everything in sight is Fr\'echet. 
It follows that 
 the cohomology  $ \wt{H}^i{\rm DR}(X_{C},r)$ is classical and Fr\'echet as well.  
  
\smallskip
 $\bullet$ {\em  Assume now that $X$ is affine}. 
Then the computation is a bit more complicated because the spaces $\Omega^i(X_K)$ and $\B^+_{\dr}$ (an $LB$-space and a Fr\'echet space, respectively) do not work together well with tensor products. We claim that (in $\sd(C_K)$)
    \begin{align*}
     F^r(\R\Gamma_{\dr}(X_K)\wh{\otimes}^R_K(\B^+_{\dr}/F^r)) & \simeq F^r(\Omega\kr({X_K})\wh{\otimes}_{K}(\B^+_{\dr}/F^r)) \\
    & =(\so({X_K})\wh{\otimes}_K F^r(\B^+_{\dr}/F^r)\to\Omega^1({X_K})\wh{\otimes}_K F^{r-1}(\B^+_{\dr}/F^r)\to \cdots)\\
     {\rm DR}(X_{C},r)  = (\R\Gamma_{\dr}(X_K) & \wh{\otimes}^R_K\B^+_{\dr})/F^r  
     \simeq (\Omega\kr({X_K})\wh{\otimes}_{K}(\B^+_{\dr}/F^r))/F^r\\  
= (\so({X_K}) & \wh{\otimes}_K(\B^+_{\dr}/F^r)\to
      \Omega^1({X_K})\wh{\otimes}_K(\B^+_{\dr}/F^{r-1})\to\cdots 
       \to  \Omega^{r-1}({X_K})\wh{\otimes}_K(\B^+_{\dr}/F^{1}))
  \end{align*}
  The first and the second strict quasi-isomorphisms we can prove just as in the Stein case. We can again invoke Lemma \ref{bios2} here because 
 $F^i\B^+_{\dr}/F^j\B^+_{\dr}$ is a Banach space and $\Omega^i(X_K)$ can be represented by an inductive limit of an acyclic inductive system of Banach spaces. It remains to prove the third strict quasi-isomorphism, i.e., that the natural map
  $$
  (\Omega\kr(X_K)\wh{\otimes}^R_K\B^+_{\dr})/F^r  {\to} (\Omega\kr(X_K)\wh{\otimes}^R_K(\B^+_{\dr}/F^r))/F^r 
  $$
  is a strict quasi-isomorphism.  But this easily follows from the distinguished triangle
  $$
   \Omega\kr(X_K)\wh{\otimes}^R_KF^r\B^+_{\dr} \to F^r(\Omega\kr(X_K)\wh{\otimes}^R_K\B^+_{\dr}) {\to} F^r(\Omega\kr(X_K)\wh{\otimes}^R_K(\B^+_{\dr}/F^r)).
  $$

  Concerning cohomology, we claim that  the cohomology  $ \wt{H}^i{\rm DR}(X_{C},r)$ is classical and that  it is an ${LB}$-space; for $i\geq r$, $\wt{H}^i{\rm DR}(X_{C},r)=0$. Recall that the de Rham complex is built from $LB$-spaces and that it has strict differentials.    The complex   ${\rm DR}(X_{C},r)$ is a complex of $LB$-spaces as well (use Section \ref{tensor}, point (5)). Its differentials are also strict. Indeed, 
  write the $i$'th differential as a composition (\ref{fixed}).  
  The first map in this composition  is a strict surjection: the surjection $\B^+_{\dr}/F^{r-i-1}\to \B^+_{\dr}/F^{r-i}$ has a continuous $K$-linear section since both spaces are $K$-Banach. The second map factors as
  $$
   \Omega^i(X_K) \wh{\otimes}_K(\B^+_{\dr}/F^{r-i-1})\verylomapr{d_i\otimes\id} \im d_i\wh{\otimes}_K(\B^+_{\dr}/F^{r-i-1})\hookrightarrow  \Omega^{i+1}(X_K) \wh{\otimes}_K(\B^+_{\dr}/F^{r-i-1}).
  $$
 Here, the first map is strict. Since a composition of a strict map and a strict injection is strict the $i$'th differential in ${\rm DR}(X_{C},r)$ is strict.
It follows that the 
  cohomology  $ \wt{H}^i{\rm DR}(X_{C},r)$ is classical, as wanted.  The claimed vanishing is now clear.

  Moreover, we easily compute that we have  a strict exact sequence
  $$
  0\to \Omega^i(X_C)/\im d_{i-1}\to {H}^i{\rm DR}(X_{C},r)\to H^i_{\dr}(X_K)\wh{\otimes}_K(\B^+_{\dr}/F^{r-i-1})\to 0
  $$
  Since $\Omega^i(X_C)/\im d_{i-1}$ is Hausdorff so is ${H}^i{\rm DR}(X_{C},r)$.  Since, by \cite[Theorem 1.1.17]{Em}, a Hausdorff quotient of a Hausdorff $LB$-space is an $LB$-space, to show that ${H}^i{\rm DR}(X_{C},r)$ is an $LB$-space it suffices  to show that so is $\ker \tilde{d}_i$, 
 for  $$\tilde{d}_i:  \Omega^i({X_K})\wh{\otimes}_K(\B^+_{\dr}/F^{r-i}) \to \Omega^{i+1}({X_K})\wh{\otimes}_K(\B^+_{\dr}/F^{r-i-1}).
 $$
  But we easily compute that there is a strict exact sequence
  $$
 0\to  \Omega^i({X_K})\wh{\otimes}_K(\gr^{r-i-1}_F\B^+_{\dr})\to  \ker \tilde{d}_i\to \ker d_i\wh{\otimes}_K(\B^+_{\dr}/F^{r-i-1})\to 0
  $$
  that is, in fact, split because the surjection $\B^+_{\dr}/F^{r-i-1}\to \B^+_{\dr}/F^{r-i}$ has a continuous $K$-linear section. It follows that, since $\ker d_i$ is a space of compact type (this follows from \cite[Prop. 1.1.41]{Em} and the fact  that $\Omega^i({X_K})$ is a space of compact type) the space  $\ker d_i\wh{\otimes}_K(\B^+_{\dr}/F^{r-i-1})$ is $LB$ (use Section \ref{tensor}, point (5)) and, finally, so is 
    $\ker \tilde{d}_i$, as wanted. 
\end{example}
\label{perl1}
 Let $X$ be affine or Stein. We can conclude from the above that our syntomic cohomology fits into the long exact sequence
$$
\to {H}^{i-1}{\rm DR}(X_{C},r)\stackrel{\partial}{\to} \wt{H}^i_{\synt}(X_{C},\Q_p(r))\to ({H}^i_{\hk}(X_0)\wh{\otimes}_K\wh{\B}^+_{\st})^{N=0,\phi=p^r}
\verylomapr{\iota_{\hk}\otimes\iota} {H}^i {\rm DR}(X_{C},r)\to 
$$
where all the terms but the syntomic one were shown to be classical and ${LB}$ or Fr\'echet, respectively. We  will show later that the syntomic cohomology  has these properties as well (see Proposition \ref{fd11}).
\begin{example}Assume that $X$ is affine or Stein and geometrically irreducible.
For $r=0$, from the above computations,  we obtain the isomorphism
$$\wt{H}^0_{\synt}(X_{C},\Q_p)\stackrel{\sim}{\to} (H^0_{\hk}(X_0)\wh{\otimes}_F\B^+_{\st})^{\phi=p,N=0}\simeq \B_{\crr}^{+,\phi=1}= \Q_p.$$
Hence $\wt{H}^0_{\synt}(X_{C},\Q_p)$ is classical. 

 For $r=1$, we obtain 
the following   exact sequence 
$$
   H^0 {\rm HK}(X_{C},1)\lomapr{\iota_{\hk}\otimes\iota}  \so(X_K) \wh{\otimes} _K C
    \to \wt{H}^1_{\synt}(X_{C},\Q_p(1))\to (H^1_{\hk}(X_0)\wh{\otimes}_F\wh{\B}^+_{\st})^{\phi=p,N=0}
   \to 0
$$
Since $H^0_{\hk}(X_0)=F$, we have $H^0 {\rm HK}(X_{C},1)=\B^{+,\phi=p}_{\crr}$ and the map    $H^0 {\rm HK}(X_{C},1)\lomapr{\iota_{\hk}\otimes\iota}  \so(X_K) \wh{\otimes} _K C$ is induced by the map $\iota: \B^+_{\crr}\to\B^+_{\dr}/F^1$. Since we have the fundamental sequence
$$
0\to \Q_p(1) \to \B^{+,\phi=p}_{\crr}\to \B^+_{\dr}/F^1\to 0
$$
we get the  exact sequence
$$
   0\to C\to \so(X_C)  \lomapr{\partial} \wt{H}^1_{\synt}(X_{C},\Q_p(1))\to (H^1_{\hk}(X_0)\wh{\otimes}_F\wh{\B}^+_{\st})^{\phi=p,N=0}
   \to 0
$$
Hence $\wt{H}^1_{\synt}(X_{C},\Q_p(1))$ is classical. 
\end{example}
\subsubsection{Fundamental diagram.} 
We will construct the fundamental diagram that syntomic cohomology fits into. We start with an example.
\begin{example}{\em Fundamental diagram; the case of $r=1$.} 
\label{fd1}
Assume that $X$ is affine or Stein and geometrically irreducible.  We claim that we have the following commutative diagram with strictly exact rows.
$$
\xymatrix{
0\ar[r] & \so(X_C)/C\ar[r]^-{\partial}\ar@{=}[d] & H^1_{\synt}(X_{C},\Q_p(1))\ar[d]^{\beta} \ar[r] & (H^1_{\hk}(X_0)\wh{\otimes}_F\B^+_{\st})^{\phi=p,N=0}\ar[r]\ar[d]^{\iota_{\hk}\otimes\theta} & 0\\
0\ar[r]& \so(X_C)/C \ar[r]^-d & \Omega^1(X_C)^{d=0} \ar[r] & H^1_{\dr}(X_K)\wh{\otimes}_K C\ar[r] & 0
}
$$
The top row is the strictly exact sequence from the Example \ref{perl1}. The bottom row is induced by  the natural exact sequence defining $ H^1_{\dr}(X_K)$. By Section \ref{tensor}, it is isomorphic to
the sequence
$$
0\to \so(X_C)/C \stackrel{d}{\to} \Omega^1(X_C)^{d=0} \to H^1_{\dr}(X_C)\to  0.
$$
Hence it is strictly exact by Lemma \ref{czesto}.
The map $\iota_{\hk}\otimes\theta$ is induced by the composition
$$ \R\Gamma_{\hk}(X_0)\wh{\otimes}_{F}\wh{\B}^+_{\st}\verylomapr{\iota_{\hk}\otimes\iota}   \R\Gamma_{\dr} (X_K)\wh{\otimes}_K\B^+_{\dr}\stackrel{\theta}{\to} \R\Gamma_{\dr} (X_K)\wh{\otimes}_K C .
$$
The map $\beta$ is induced by the composition (the fact that the first map  lands in $F^1$ is immediate from the definition of $\R\Gamma_{\synt}(X_{C},\Q_p(1))$)
\begin{align*}
\R\Gamma_{\synt}(X_{C},\Q_p(1)) & \to   F^1 (\R\Gamma_{\dr}(X_K)\wh{\otimes}^R_K\B^+_{\dr})
 \stackrel{\theta}{\to}
(0\to \Omega^1({X_K})\wh{\otimes}_KC\to \Omega^2({X_K})\wh{\otimes}_KC\to\cdots).
\end{align*}
Clearly they make the right  square in the above diagram commute. To see that the left square commutes as well it is best to consider the following diagram of maps of distinguished triangles
$$
\xymatrix{
\R\Gamma_{\synt}(X_{C},\Q_p(1))\ar[r]\ar@{.>}[d]^{\tilde{\beta}} & [\R\Gamma_{\hk}(X_0)\wh{\otimes}_F\wh{\B}^+_{\st}]^{\phi=p,N=0}
\ar[d]^{\iota_{\hk}\otimes\iota}\ar[r]^{\iota_{\hk}\otimes\iota} & (\R\Gamma_{\dr}(X_K)\wh{\otimes}_K\B^+_{\dr})/F^1\ar@{=}[d]\\
F^1(\R\Gamma_{\dr}(X_K)\wh{\otimes}_K\B^+_{\dr})\ar[r]\ar[d]^{\theta}&  \R\Gamma_{\dr}(X_K)\wh{\otimes}_K\B^+_{\dr}\ar[d]^{\theta}\ar[r] & (\R\Gamma_{\dr}(X_K)\wh{\otimes}_K\B^+_{\dr})/F^1\ar[d]^{\theta}\\
\Omega^{\geq 1}({X_K})\wh{\otimes}_K C[-1] \ar[r] & \Omega\kr(X_K)\wh{\otimes}_K C \ar[r] & \so({X_K})\wh{\otimes}_K C
}
$$
The map $\tilde{\beta}$ is the map on mapping fibers induced by the commutative right square.  We  have $\beta=\theta\tilde{\beta}$. 
It remains to check that the map $\so(X_K)\wh{\otimes}_K C\to \Omega^1(X_K)\wh{\otimes}_KC$ induced from the bottom row of the above diagram is equal to $d$ but this is easy.
\end{example}

 And here is the general case.
\begin{proposition} 
\label{fd11}Let $X$ be an affine or  a Stein  weak formal scheme.
Let  $r\geq 0$. There is a natural  map of strictly exact sequences
  $$
\xymatrix{
 0\ar[r] & \Omega^{r-1}(X_C)/\ker d\ar[r]^-{\partial}\ar@{=}[d] & H^r_{\synt}(X_{C},\Q_p(r))\ar[d]^{\beta} \ar[r] & (H^r_{\hk}(X_{0})\wh{\otimes}_F\B^+_{\st})^{N=0,\phi=p^r}
\ar[r]\ar[d]^{\iota_{\hk}\otimes\theta} & 0\\
 0\ar[r] & \Omega^{r-1}(X_C)/\ker d \ar[r]^-d & \Omega^r(X_C)^{d=0} \ar[r] & H^r_{\dr}(X_K)\wh{\otimes}_K C\ar[r] & 0
}
$$
Moreover, $\ker(\iota_{\hk}\otimes\theta)\simeq (H^r_{\hk}(X_{0})\wh{\otimes}_F\B^+_{\st})^{N=0,\phi=p^{r-1}}
$, $H^r_{\synt}(X_{C},\Q_p(r))$ is $LB$ or Fr\'echet, respectively, and the maps $\beta$, $\iota_{\hk}\otimes\theta$ are strict and have closed images.
  \end{proposition}
  \begin{proof} To deal easier with topological issues, we start with changing the period ring in the proposition   from $\B^+_{\st}$ to $\wh{\B}^+_{\st}$. That is, we will show that 
the natural map $
\iota_{\st}: {\B}^+_{\st}\to \wh{\B}^+_{\st}$ induces a commutative diagram (in $C_{\Q_p}$)
\begin{equation}
\label{sqq1}
\xymatrix{(H^r_{\hk}(X_0) \wh{\otimes}_{F}{\B}^+_{\st})^{N=0,\phi=p^i} \ar[r]^-{\iota_{\hk}\otimes\iota}\ar[d]^{1\otimes\iota_{\st}}_{\wr} &   
H^r_{\dr}(X_K) \wh{ \otimes}_{K}C \ar@{=}[d]\\
   (H^r_{\hk}(X_0)\wh{\otimes}_{F}\wh{\B}^+_{\st})^{N=0,\phi=p^i} \ar[r]^-{\iota_{\hk}\otimes\iota}&   H^r_{\dr}(X_K) \wh{ \otimes}_{K}C.
   }
\end{equation}
The fact that this diagram is commutative
 is clear  since the composition
$\B^+_{\st}\stackrel{\iota_{\st}}{\to} \wh{\B}^+_{\st}\stackrel{\iota}{\to} \B^+_{\dr} $ is just the original map $\iota:  {\B}^+_{\st}{\to} \B^+_{\dr}$. 
To see that the map $1\otimes\iota_{\st}$ is a strict isomorphism it suffices to look at the diagram
 $$
 \xymatrix{ & (H^r_{\hk}(X_0)\wh{\otimes}_{F}{\B}^+_{\st})^{N=0}\ar[d]^{\iota_{\st}} \\
  H^r_{\hk}(X_0)\wh{\otimes}_{F}\B^+_{\crr}\ar[r]^-{\wh{\beta}}_-{\sim} \ar[ru]^-{\beta} & (H^r_{\hk}(X_0)\wh{\otimes}_{F}\wh{\B}^+_{\st})^{N=0},
  }
  $$
  where $\wh{\beta}$ is the trivialization from (\ref{trivialization}). The map $\wh{\beta}$  is well-defined because $N$ is nilpotent on $H^r_{\hk}(X_0)$ and factorizes through the map $\iota_{\st}$ yielding the map $\beta$ and making the above diagram commute. The map $\wh{\beta}$ is a strict isomorphism. It follows that both $\beta$ and $\iota_{\st}$ are algebraic isomorphisms and this easily implies that $\iota_{\st}$ is a strict isomorphism, as wanted.

  We will now prove our  proposition with $\B^+_{\st}$ replaced by $\wh{\B}^+_{\st}$. The following map of exact sequences 
(where $\Omega^i$, $H^i_{\rm dR}$ and $H^i_{\rm HK}$ stand for
$\Omega^i(X_C)$, $H^i_{\rm dR}(X_K)$ and $H^i_{\rm HK}(X_0)$, respectively)
is constructed in an analogous way to the case of $r=1$ treated in the above example 
\begin{equation}
\label{zima22}
\xymatrix@C=.5cm{
 (H^{r-1}_{\hk}\wh{\otimes}_F\wh{\B}^+_{\st})^{N=0,\phi=p^r}\ar[r]\ar[d]^{\iota_{\hk}\otimes\theta}  & \Omega^{r-1}/d\Omega^{r-2}\ar[r]^-{\partial}\ar@{=}[d] & \wt{H}^r_{\synt}(X_{C},\Q_p(r))\ar[d]^{\beta} \ar[r] & (H^r_{\hk}\wh{\otimes}_F\wh{\B}^+_{\st})^{N=0,\phi=p^r}
\ar[d]^{\iota_{\hk}\otimes\theta} \to 0\\
0\to H^{r-1}_{\dr}\wh{\otimes}_KC \ar[r] &  \Omega^{r-1}/d\Omega^{r-2} \ar[r]^-d & \Omega^{r,d=0} \ar[r]^{\pi} & H^r_{\dr}\wh{\otimes}_K C\to 0
}
\end{equation}
To prove the first claim of the proposition it suffices to show that the map $\iota_{\hk}\otimes\theta$ in degree $r-1$ is surjective. For that we will need the following lemma.
  \begin{lemma}
  \label{lyon3}
Let $M$ be an effective\footnote{We call a $(\phi,N)$-module  $M$  {\em effective} if all the slopes of the Fobenius are $\geq 0$.} finite $(\phi,N)$-module over $F$. The sequence
$$
0\to (M\otimes_F\wh{\B}^+_{\st})^{\phi=p^j,N=0}\stackrel{t}{\to}(M\otimes_F\wh{\B}^+_{\st})^{\phi=p^{j+1},N=0}\lomapr{1\otimes\theta}M\otimes_FC
$$
is exact. Moreover, the right arrow is a surjection if the slopes of Frobenius are $\leq j$.
\end{lemma}
\begin{proof}Using the trivialization (\ref{trivialization}) and the fact that $\theta (u)=0$, we get the following commutative diagram
$$
\xymatrix{
0\ar[r] &  (M\otimes_F\wh{\B}^+_{\st})^{\phi=p^j,N=0}\ar[r]^{t} & (M\otimes_F\wh{\B}^+_{\st})^{\phi=p^{j+1},N=0}\ar[r]^-{1\otimes\theta} & M\otimes_FC\\
0\ar[r] &  (M\otimes_F\B^+_{\crr})^{\phi=p^j}\ar[u]^{\beta}_{\wr}\ar[r]^{t} & (M\otimes_F\B^+_{\crr})^{\phi=p^{j+1}}\ar[u]^{\beta}_{\wr}\ar[r]^-{1\otimes\theta} & M\otimes_FC\ar[u]^{\id}_{\wr}
}
$$
Hence it suffices to prove the analog of our lemma for the bottom sequence.

 First we will show that the  following sequence 
\begin{equation}
\label{lyon1}0\to (M\otimes_F\B^+_{\crr})^{\phi=p^j}\stackrel{t}{\to}(M\otimes_F\B^+_{\crr})^{\phi=p^{j+1}}\lomapr{1\otimes\theta}M\otimes_FC
\end{equation}
 is exact. Multiplication by $t$ is clearly injective. To show exactness in the middle it suffices to show that
$$
(M\otimes_FF^1\B^+_{
\crr})^{\phi=p^{j+1}}=(M\otimes_Ft\B^+_{
\crr})^{\phi=p^{j+1}},$$
where $F^1\B^+_{
\crr}:=\B^+_{\crr}\cap F^1\B^+_{\dr}.
$
Or that 
$(F^1\B^+_{
\crr})^{\phi=p^{j+1-\alpha}}=(t\B^+_{
\crr})^{\phi=p^{j+1-\alpha}}.
$
But this follows from the fact that 
$$(F^1\B^+_{
\crr})^{\phi=p^{j+1-\alpha}}\subset\{x\in\B^+_{\crr}|\theta(\phi^k(x))=0, \forall k\geq 0\}=t\B^+_{\crr}.
$$

 It remains to show that if the  Frobenius slopes of $M$ are $\leq j$ then the last arrow in the sequence (\ref{lyon1}) is a surjection. To see this, we note that all the terms 
 in the sequence  are $C$-points of  finite dimensional BC spaces\footnote{Which are called finite dimensional Banach Spaces in \cite{CB} and Banach-Colmez spaces in most of the literature.} and the maps can be lifted to maps  of such spaces. It follows that the  cokernel of multiplication by $t$ is a finite dimensional BC space. We compute its  Dimension \cite[5.2.2]{CN}:
\begin{align*}
\ddim (M\otimes_F\B^+_{\crr})& ^{\phi=p^{j+1}}-\ddim (M\otimes_F\B^+_{\crr})^{\phi=p^{j}}
 =\sum_{r_i\leq j+1}(j+1-r_i,1)-\sum_{r_i\leq j}(j-r_i,1)\\
& =((j+1)\dim_F M-t_N(M),\dim_F M)-(j\dim_F M-t_N(M),\dim_F M)
 =
(\dim_F M,0).
\end{align*}
Here   $r_i$'s are the slopes of Frobenius repeated with multiplicity, $t_N(M)=v_p(\det \phi)$, and the second equality follows from the fact that the slopes of Frobenius are $\leq j$.
Since this Dimension  is the same as the Dimension of the BC space corresponding to 
 $M\otimes_FC$, we get the surjection we wanted. 
\end{proof}

   Let us come back to the proof of Proposition \ref{fd11}. \\
  $\bullet$ {\em Assume first that $X$ is quasi-compact}. By the above lemma,  to prove that the map $\iota_{\hk}\otimes\theta$ is surjective in degree $r-1$ and that its kernel in degree $r$ is isomorphic to $ (H^r_{\hk}(X_{0})\wh{\otimes}_F\wh{\B}^+_{\st})^{\phi=p^{r-1},N=0}
$
it suffices to show that the slopes of Frobenius 
on 
$H^i_{\rm HK}(X_0)$ are ~$\leq i$. For that we use the weight spectral sequence (\ref{weights}) to reduce to showing that, for a smooth  scheme $Y$ over $k$, the slopes of Frobenius
on the (classical) rigid cohomology $H^j_{\rig}(Y/F)$ are $\leq j$. But this is well-known \cite[Th\'eor\`eme 3.1.2]{CS}. 

   We have shown that syntomic cohomology fits into an exact sequence 
   $$
   \xymatrix{
 0\ar[r] & \Omega^{r-1}(X_C)/\ker d\ar[r]^-{\partial} & \wt{H}^r_{\synt}(X_{C},\Q_p(r))\ar[r]^-{\pi_1} & (H^r_{\hk}(X_{0})\wh{\otimes}_F\B^+_{\st})^{N=0,\phi=p^r}
\ar[r] & 0
}
$$
Hence, since it is an extension of two classical object, it is classical. Since $H^r_{\dr}(X_K)$ is a finite dimensional vector space over $K$ the surjective map $\pi$ in the diagram (\ref{zima22}) has a section.
   Since the syntomic cohomology is the equalizer of the maps $\iota_{\hk}\otimes\theta$ and $\pi$ this section lifts to a section of the surjection $\pi_1$ above. Hence ${H}^r_{\synt}(X_{C},\Q_p(r))$ is an $LB$-space. 
    Since the map $\iota_{\hk}\otimes\theta$ lifts to a  map of finite Dimensional BC spaces it is strict and  has a closed image. It follows that so does the pullback map $\beta$. 

 $\bullet$ {\em Assume now that $X$ is Stein with a Stein covering $\{U_i\},i\in\N.$}
  Since the map $\iota_{\hk}\otimes\theta$ is the projective limit of the maps
$$(\iota_{\hk}\otimes\theta)_i: (H^{*}_{\hk}(U_{i})\wh{\otimes}_F\wh{\B}^+_{\st})^{N=0,\phi=p^r}\to
 H^{*}_{\dr}(]U_i[_X)\wh{\otimes}_KC 
$$
the computation above yields the statement on the kernel in degree $r$. For the surjectivity in degree $r-1$, it remains to show 
the vanishing of
$ H^1\holim_i(H^r_{\hk}(U_{i})\wh{\otimes}_F\wh{\B}^+_{\st})^{N=0,\phi=p^{r-1}}.
$
But this was shown in the proof of Lemma \ref{Lyon-again}.

   Since the maps $(\iota_{\hk}\otimes\theta)_i$  are strict and  have closed images and $H^1\holim_i(H^r_{\hk}(U_{i})\wh{\otimes}_F\wh{\B}^+_{\st})^{N=0,\phi=p^{r-1}}.$  it follows that the projective limit map  $\iota_{\hk}\otimes\theta$ inherits these properties and then so does the pullback map $\beta$.
 Finally, since the syntomic cohomology is the equalizer of the maps $\iota_{\hk}\otimes\theta$ and $\pi$ of Fr\'echet spaces it is Fr\'echet.  
\end{proof}
\begin{remark}Assume that $X$ is affine. 
The image of the map $\iota_{\hk}\otimes\theta$ in degree $r$ in the fundamental diagram is the $C$-points of a finite dimensional BC space that is the cokernel of the map:
$$  (H^r_{\hk}(X_{0})\otimes_F{\mathbb B}^+_{\st})^{N=0,\phi=p^{r-1}}\stackrel{t}{\to} (H^r_{\hk}(X_{0})\otimes_F{\mathbb B}^+_{\st})^{N=0,\phi=p^{r}}.
$$
Its Dimension is equal to 
\begin{align*}
\sum_{r_i\leq r}(r-r_i,1)-\sum_{r_i\leq r-1}(r-1-r_i,1)=\sum_{r_i \leq r-1}(1,0)+\sum_{r-1 <  r_i \leq r}(r-r_i,1).
\end{align*}
\end{remark}
\subsection{Crystalline  syntomic cohomology} 
The classical crystalline syntomic cohomology of Fontaine-Messing and the related period map to \'etale cohomology generalize easily to formal schemes. 
We define them and then modify this syntomic cohomology in the spirit of Bloch-Kato to make it more computable.
\subsubsection{Definition a la Fontaine-Messing}\label{definitionFM}
Let $X$ be a semistable $p$-adic formal scheme over $\so_K$. This means that, locally for the Zariski topology,  $X=\Spf(R)$, where  $R$ is the $p$-adic completion of an algebra \'etale over 
$
\so_K\{T_1,\dots,T_n\}/(T_{1}\cdots T_{m}-\varpi).
$
That is, we do not allow self-intersections. 
 We equip $X$ with the log-structure induced by the special fiber. Set $\overline{X}:=X_{\so_C}$.
 
  For $r\geq 0$, we have the {\em geometric syntomic cohomology} of  Fontaine-Messing \cite{FM}
\begin{align*}
\R\Gamma_{\synt}(\overline{X},\Z/p^n(r)) & :=[F^r\R\Gamma_{\crr}(\overline{X}_{n})\lomapr{\phi-p^r}\R\Gamma_{\crr}(\overline{X}_n)],\quad F^r\R\Gamma_{\crr}(\overline{X}_n):=\R\Gamma_{\crr}(\overline{X}_n,\sj^{[r]})\\
\R\Gamma_{\synt}(\overline{X},\Z_p(r))  &  :=\holim_n\R\Gamma_{\synt}(\overline{X},\Z/p^n(r)).
\end{align*}
Crystalline cohomology used here is the absolute one, i.e., over $\so_{F,n}$ and $\sj^{[r]}$ is the $r$'th level of the crystalline Hodge filtration sheaf. We have
\begin{align}
  \label{map1}
  \R\Gamma_{\synt}(\overline{X},\Z_p(r))_{\Q_p}  & =[F^r\R\Gamma_{\crr}(\overline{X})_{\Q_p}\lomapr{\phi-p^r}
  \R\Gamma_{\crr}(\overline{X})_{\Q_p}]\\
   & =[\R\Gamma_{\crr}(\overline{X})_{\Q_p}^{\phi=p^r}\to
  \R\Gamma_{\crr}(\overline{X})_{\Q_p}/F^r]\notag
  \end{align}
    and similarly with $\Z_p$ and $\Z/p^n$ coefficients. 

   The above geometric syntomic cohomology is related, via period morphisms,  to the \'etale cohomology of the rigid space $X_C$ (see below) and hence allows to describe the latter using differential forms. To achieve the same for pro-\'etale cohomology, we need to modify the definition of syntomic cohomology a bit. 
 Consider the complexes of sheaves on $X$ associated to the presheaves  ($U$ is an affine Zariski open in $X$ and $\overline{U}:=U_{\so_C}$)
   \begin{align*}
   \sa_{\crr}  & := (U\mapsto (\holim_n\R\Gamma_{\crr}(\overline{U}_n))_{\Q_p}), \quad F^r\sa_{\crr}:=(U\mapsto (\holim_nF^rR\Gamma_{\crr}(\overline{U}_n))_{\Q_p}),\\
    \sss(r) & :=(U\mapsto \R\Gamma_{\synt}(\overline{U},\Z_p(r))_{\Q_p}).
    \end{align*}
    We have 
    $$
    \sss(r)=[ F^r\sa_{\crr}\lomapr{\phi-p^r} \sa_{\crr}]=[ \sa_{\crr}^{\phi=p^r}\to \sa_{\crr}/F^r].
    $$
 We define  
    \begin{align*}
    \R\Gamma_{\crr}(\overline{X},\Q_p)  :=\R\Gamma(X,\sa_{\crr}),\quad     F^r\R\Gamma_{\crr}(\overline{X},\Q_p):=  \R\Gamma(X,F^r\sa_{\crr}),\quad \R\Gamma_{\synt}(\overline{X},\Q_p(r) ) :=  \R\Gamma(X,\sss(r)).
 \end{align*}
 Hence
 \begin{equation}
 \label{map2}
     \R\Gamma_{\synt}(\overline{X},\Q_p(r))=[F^r\R\Gamma_{\crr}(\overline{X},\Q_p)\lomapr{\phi-p^r}   \R\Gamma_{\crr}(\overline{X},\Q_p)]=[\R\Gamma_{\crr}(\overline{X},\Q_p)^{\phi=p^r}\to    \R\Gamma_{\crr}(\overline{X},\Q_p)/F^r].
 \end{equation}
 There is a natural map
 \begin{equation}
 \label{passage}
 \R\Gamma_{\synt}(\overline{X},\Z_p(r))_{\Q_p}\to\R\Gamma_{\synt}(\overline{X},\Q_p(r)).
 \end{equation}
 It is a quasi-isomorphism in the case $X$ is of finite type but not in general (since in the case of $\Z_p(r)$ we do all computations on $U$'s as above integrally and invert $p$ at the very end and in the case of 
 $\Q_p(r)$ we invert $p$ already on  each $U$).

      By proceeding just as in the case of overconvergent syntomic cohomology (and using crystalline embedding systems instead of dagger ones) we can equip both complexes in (\ref{passage}) 
  with a natural topology for which they   become  complexes of Banach  spaces over $\Q_p$  in the case $X$ is quasi-compact
  \footnote{We note that $\so_K$ being syntomic over $\so_F$, all the integral complexes in sight are
   in fact  $p$-torsion free.}.   We used here the simple fact that an exact sequence of Fr\'echet spaces is strictly exact. 
   
   The defining mapping fibers (\ref{map1}) and (\ref{passage}) are taken in $\sd(C_{\Q_p})$. Moreover, the  change of topology map in (\ref{passage}) is continuous (and a strict quasi-isomorphism  if $X$ is of finite type).  
  \subsubsection{Period map}\label{PR3}
We are interested in syntomic cohomology a la Fontaine-Messing because of the following comparison 
  \cite{Ts, CN}.
    \begin{proposition}
    \label{Ts1}
    Let $X$ be a  semistable finite type formal scheme over $\so_K$. 
    The Fontaine-Messing period map\footnote{We take the version of the Fontaine-Messing period map that is compatible with Chern classes.} \cite{FM}
    $$\alpha_{\rm FM}: \R\Gamma_{\synt}(\overline{X},\Q_p(r))\to \R\Gamma_{\eet}(X_{C},\Q_p(r))
    $$
    is a quasi-isomorphism after truncation $\tau_{\leq r}$.
    \end{proposition}
    
    We equip the pro-\'etale and \'etale cohomologies $\R\Gamma_{\proeet}(X_{C},\Q_p(r))$, and  $\R\Gamma_{\eet}(X_{C},\Q_p(r))
$ with a natural topology by proceeding as in the case of overconvergent rigid cohomology by using as local data compatible complexes of free $\Z/p^n$-modules \footnote{Such complexes can be found, for example, by taking the system of  \'etale hypercovers.}. If  
 $X$ is quasi-compact, we obtain in this way complexes of Banach spaces over $\Q_p$. 
\begin{corollary}
\label{fini1}
Let $X$ be a  semistable formal scheme over $\so_K$. 
    There is a natural  Fontaine-Messing period map
    \begin{equation}
    \label{topology1}
    \alpha_{\rm FM}: \R\Gamma_{\synt}(\overline{X},\Q_p(r))\to \R\Gamma_{\proeet}(X_{C},\Q_p(r))
    \end{equation}
   that  is a strict quasi-isomorphism after truncation $\tau_{\leq r}$.
\end{corollary}
\begin{proof}
Cover $X$ with quasi-compact  formal schemes and invoke  Proposition \ref{Ts1}; we obtain a quasi-isomorphism
$$
    \alpha_{\rm FM}:\tau_{\leq r} \R\Gamma_{\synt}(\overline{X},\Q_p(r))\to \tau_{\leq r}\R\Gamma_{\eet}(X_{C},\Q_p(r)).
$$
To see that it is  strict, it suffices to note that, locally,  the period map is a zigzag between complexes of  Banach spaces and invoke Lemma \ref{kolobrzeg-winter}.  

 It remains to show that, for a quasi-compact $X$,  the natural map
 $$
 \R\Gamma_{\eet}(X_{C},\Q_p(r))\to \R\Gamma_{\proeet}(X_{C},\Q_p(r))
 $$
 is a (strict) quasi-isomorphism.
 From
\cite[Cor. 3.17]{Sch1} we know that this is true with $\Z/p^n$-coefficients.
 This implies that we have a sequence of quasi-isomorphisms
 \begin{align*}
   \R\Gamma_{\eet}(X_{C},\Z_p(r)) & =\holim_n\R\Gamma_{\eet}(X_{C},\Z/p^n(r))
   \simeq \holim_n\R\Gamma_{\proeet}(X_{C},\Z/p^n(r)) \\ 
   & \simeq \R\Gamma_{\proeet}(X_{C},\holim_n\Z/p^n(r))
    \simeq \R\Gamma_{\proeet}(X_{C},\varprojlim_n\Z/p^n(r))=\R\Gamma_{\proeet}(X_{C},\Z_p(r)),
\end{align*}
where the third quasi-isomorphism follows from the fact that $\R\Gamma$ and $\holim$ commute and the fourth one 
follows from \cite[Prop. 8.2]{Sch1}. 

 It remains  to show that 
$$
 \R\Gamma_{\proeet}(X_C,\Z_p(r))\otimes \Q_p\stackrel{\sim}{\to}\R\Gamma_{\proeet}(X_C,\Q_p(r)).
$$
But, since $|X_C|$ is quasi-compact, the site $X_{C,\proeet}$  is coherent \cite[Prop. 3.12]{Sch1}. Hence its  cohomology commutes with colimits of abelian sheaves, yielding the above quasi-isomorphism.
\end{proof}

  \subsubsection{Definition a la Bloch-Kato}\label{PR4}
 Crystalline geometric syntomic cohomology a la Fontaine-Messing can be often described in a very simple way using filtered de Rham complexes and Hyodo-Kato cohomology (if the latter can be defined) and the period rings $\B^+_{\st},\B^+_{\dr}$.This was done for proper algebraic and analytic varieties in \cite{NN, CN}. In this section  we adapt the arguments from loc. cit. to the case of some non-quasi-compact rigid varieties. The de Rham term is the same, the Hyodo-Kato term is more complicated, and the role of the period ring $\B^+_{\st}$ is played by $\wh{\B}^+_{\st}$.
  
  Let $r\geq 0$. For a semistable formal scheme $X$ over $\so_K$, we define the {\em crystalline geometric syntomic cohomology a la Bloch-Kato} (as an object in $\sd(C_{\Q_p})$)
  $$
  \R\Gamma^{\rm BK}_{\synt}({X}_{\so_{C}},\Q_p(r)):=[[\R\Gamma_{\crr}(X/r^{\rm PD}_{\varpi},\Q_p) \wh{\otimes}^R_{r^{\rm PD}_{\varpi,\Q_p}}\wh{\B}^+_{\st}]^{N=0,\phi=p^r} \verylomapr{p_{\varpi}
  \otimes\iota}  (\R\Gamma_{\dr}(X_K) \wh{ \otimes}^R_{K}{\B}^+_{\dr})/F^r].
  $$
  Here $\R\Gamma_{\crr}(X/r^{\rm PD}_{\varpi},\Q_p)$ is defined in an analogous way to $\R\Gamma_{\crr}(X,\Q_p)$ (hence it is rational; the corresponding integral cohomology we will denote simply by 
  $\R\Gamma_{\crr}(X/r^{\rm PD}_{\varpi})$). The completed tensor product 
  $\R\Gamma_{\crr}(X/r^{\rm PD}_{\varpi},\Q_p) \wh{\otimes}^R_{r^{\rm PD}_{\varpi,\Q_p}}\wh{\B}^+_{\st}$ is defined in the following way (note that $r^{\rm PD}_{\varpi,\Q_p}$ is not a field hence we can not use the tensor product in the category of convex vector spaces): if $X$ is of finite type, we set
  $$
  \R\Gamma_{\crr}(X/r^{\rm PD}_{\varpi},\Q_p) \wh{\otimes}^R_{r^{\rm PD}_{\varpi,\Q_p}}\wh{\B}^+_{\st}:=(\R\Gamma_{\crr}(X/r^{\rm PD}_{\varpi}) \wh{\otimes}_{r^{\rm PD}_{\varpi}}\wh{\A}^+_{\st}){\otimes}^L{\Q_p},
  $$
  where the integral objects  are in the category $D(\Ind(PD_{\Q_p}))$; for a general $X$, we lift the above definition from formal schemes of finite type via the \'etale cohomological descent.
  \begin{proposition}
  \label{fini2}
There exists a functorial  quasi-isomorphism   in $\sd(C_{\Q_p})$
$$
\iota_{\rm BK}:\R\Gamma^{\rm BK}_{\synt}({X}_{\so_{C}},\Q_p(r)) \stackrel{\sim}{\to}\R\Gamma_{\synt}(X_{\so_{C}},\Q_p(r)).
$$
\end{proposition}
\begin{proof}
The comparison map $\iota_{\rm BK}$ will be induced by  a pair of maps $(\iota^1_{\rm BK},\iota^2_{\rm BK})$,   basically K\"unneth cup product maps,
 that make the following diagram commute (in $\sd(C_{\Q_p})$). 
$$
\xymatrix{
  [\R\Gamma_{\crr}(X/r^{\rm PD}_{\varpi},\Q_p) \wh{\otimes}^R_{r^{\rm PD}_{\varpi,\Q_p}}\wh{\B}^+_{\st}]^{N=0} \ar[r]^-{p_{\varpi}
  \otimes\iota}\ar[d]^{\iota^1_{\rm BK}}_{\wr} &   (\R\Gamma_{\dr}(X_K) \wh{ \otimes}^R_{K}{\B}^+_{\dr})/F^r\ar[d]^{\iota^2_{\rm BK}}_{\wr}\\
 \R\Gamma_{\crr}(\overline{X},\Q_p)\ar[r]     &   \R\Gamma_{\crr}(\overline{X},\Q_p)/F^r
}
$$
{\em {\rm (i)} Construction of the map $\iota^1_{\rm BK}$.} We may argue locally and assume that  $X$ is  quasi-compact. 
Consider the following maps in $\sd(C_F)$
\begin{align}
\label{comp1}
\R\Gamma_{\crr}(X/r^{\rm PD}_{\varpi},\Q_p)\widehat{\otimes}^R_{r^{\rm PD}_{\varpi,\Q_p}}\widehat{\B}^+_{\st}\stackrel{\cup}{\to}\R\Gamma_{\crr}(X_{\so_{C}}/r^{\rm PD}_{\varpi},\Q_p)\leftarrow
\R\Gamma_{\crr}(X_{\so_{C}},\Q_p).
\end{align}
We claim that the  cup product map 
is a quasi-isomorphism: indeed, 
the proof of an analogous result in the case of schemes  \cite[Prop. 4.5.4]{Ts} goes through in our setting. Recall the key points.
By (\ref{isom1}) and the fact that $\A_{\st,n}$ is flat over $r^{\rm PD}_{\varpi,n}$, it suffices to prove that the K\"unneth morphism
\begin{equation}
\label{int1}
\cup: \quad\R\Gamma_{\crr}(X_n/r^{\rm PD}_{\varpi,n})\otimes^L_{r^{\rm PD}_{\varpi,n}}\R\Gamma_{\crr}(\so^{\times}_{\ovk,n}/r^{\rm PD}_{\varpi,n}) {\to}\R\Gamma_{\crr}(X_{\so_{\ovk},n}/r^{\rm PD}_{\varpi,n})
\end{equation}
 is a quasi-isomorphism. By unwinding both sides one finds  a K\"unneth morphism for certain de Rham complexes. It is a quasi-isomorphism because these complexes   are ``flat enough''
 which follows from the fact that  the map $X_{\so_{\ovk,n}}\to \so_{K,n}^{\times}$ is log-syntomic. This finishes the argument.

  Both maps in (\ref{comp1}) are compatible with the
   monodromy operator $N$. Moreover, we have the distinguished triangle  \cite[Lemma 4.2]{Ka3}
\begin{equation}
\label{int2}
\R\Gamma_{\crr}(X_{\so_{\ovk,n}}/\so_{F,n})\to\R\Gamma_{\crr}(X_{\so_{\ovk,n}}/r^{\rm PD}_{\varpi,n})\stackrel{N}{\to} \R\Gamma_{\crr}(X_{\so_{\ovk,n}}/r^{\rm PD}_{\varpi,n}).
\end{equation}
It follows that the last map in (\ref{comp1}) is a quasi-isomorphism after taking the (derived) $N=0$ part.
Hence applying $N=0$ to the terms in (\ref{comp1}) we obtain a functorial quasi-isomorphism in $\sd(C_F)$ (for strictness, note that, rationally,  we worked only with complexes of Banach spaces)
\begin{equation}
\label{zurich1}
\iota_{\rm BK}^1:  [\R\Gamma_{\crr}(X/r^{\rm PD}_{\varpi},\Q_p) \wh{\otimes}^R_{r^{\rm PD}_{\varpi,\Q_p}}\wh{\B}^+_{\st}]^{N=0} \stackrel{\sim}{\to }\R\Gamma_{\crr}(X_{\so_{C}},\Q_p).
\end{equation}
{\em {\rm (ii)} Construction of the map $\iota^2_{\rm BK}$.} We may argue locally and assume that  $X$ is  quasi-compact. Consider the maps
 \begin{align}  
 \label{paris11}
 (\R\Gamma_{\crr}({X}_n/\so_{K,n}^{\times}){\otimes}^L_{\so_{K,n}}\R\Gamma_{\crr}(\so_{\ovk,n}^{\times}/\so_{K,n}^{\times}))/F^r
    \stackrel{\cup}{\to} \R\Gamma_{\crr}({X}_{\so_{\ovk,n}}/\so_{K,n}^{\times})/F^r
   {\leftarrow}\R\Gamma_{\crr}({X}_{\so_{\ovk,n}})/F^r.
 \end{align}
The  cup product map is a K\"unneth map and it is  a quasi-isomorphism for the same reason as the map  (\ref{int1}).
The second map -- the change of base map from $\so_F$ to $\so_K^{\times}$ -- is a  quasi-isomorphism (up to a universal  constant) by \cite[Cor. 2.4]{NN}.  
Rationally, the above maps  induce a map
 \begin{align*}
  \iota_{\rm BK}^2:  (\R\Gamma_{\dr}({X}_K)\wh{\otimes}^R_{K}{\B}^+_{\dr})/F^r
   \stackrel{\sim}{\to}\R\Gamma_{\crr}({X}_{\so_{C}},\Q_p)/F^r.
 \end{align*}
Since ${\B}^+_{\dr}$ is a Fr\'echet space, the natural map $$ 
(\R\Gamma_{\dr}({X}_K)\wh{\otimes}_{K}{\B}^+_{\dr})/F^r\to (\R\Gamma_{\dr}({X}_K)\wh{\otimes}^R_{K}{\B}^+_{\dr})/F^r
$$
is a strict quasi-isomorphism, hence so is the map
$\iota_{\rm BK}^2$.

  Compatibility of the maps $\iota_{\rm BK}^1,\iota_{\rm BK}^2$ can be  inferred from 
   the  natural commutative diagram
$$
\xymatrix{
[\R\Gamma_{\crr}({X}/r^{\rm PD}_{\varpi},\Q_p)\wh{\otimes}^R_{r^{\rm PD}_{\varpi,\Q_p}}\widehat{\B}^+_{\st}]^{N=0}\ar@/_70pt/[dd]_{\iota_{\rm BK}^1}\ar[d]^{\cup}_{\wr} 
\ar@/^30pt/[rr]^-{p_{\varpi}\otimes\iota} \ar[r]^-{p_{\varpi}\otimes p_{\varpi}}   &   
(\R\Gamma_{\crr}({X}/\so_K^{\times},\Q_p)\wh{\otimes}^R_K\R\Gamma_{\crr}(\so_{C}^{\times}/\so_K^{\times})_{\Q_p})/F^r\ar[d]^{\cup}_{\wr}
  & (\R\Gamma_{\dr}({X}_K)\wh{\otimes}^R_{K}{\B}^+_{\dr})/F^r
\ar[ddl]^{\iota_{\rm BK}^2}\ar[l]_-{\sim}\\
     [\R\Gamma_{\crr}({X}_{\so_{C}}/r^{\rm PD}_{\varpi},\Q_p)]^{N=0}\ar[r]^{p_{\varpi}}  & \R\Gamma_{\crr}({X}_{\so_{C}}/\so_K^{\times},\Q_p)/F^r\\
         \R\Gamma_{\crr}({X}_{\so_{C}},\Q_p)\ar[u]^{\wr}\ar[r]  &        \R\Gamma_{\crr}({X}_{\so_{C}},\Q_p)/F^r.\ar[u]^{\wr}
}
$$
\end{proof}

\section{Comparison of syntomic cohomologies}
We have defined two geometric syntomic cohomologies: the crystalline one and the overconvergent one. 
We will show now (Theorem~\ref{fini3})
that they are naturally isomorphic for Stein spaces and that, as a result, the $p$-adic pro-\'etale cohomology 
fits into a fundamental diagram (Theorem \ref{fdd}). We use this result to describe the pro-\'etale
cohomology of affine spaces, tori, and curves (see section~\ref{EXA}).
\subsection{Comparison morphism}
\label{compmorph}
 Let $X$ be a semistable weak formal scheme over $\so_K$. Let $\wh{X}$ be the associated  formal scheme. 
The purpose of this section is to prove that the change of convergence map  from overconvergent syntomic cohomology to crystalline syntomic cohomology is a strict quasi-isomorphism assuming that $X$ is Stein. 
 \begin{theorem}
 \label{fini3}Let $r\geq 0$.
There is a functorial  map  in $\sd(C_{\Q_p})$
$$
\iota_{\rig}: \R\Gamma_{\synt}(X_{C},\Q_p(r))\to\R\Gamma^{\rm BK}_{\synt}(\wh{X}_{\so_{C}},\Q_p(r)).
$$
It is a quasi-isomorphism if $X$ is Stein.
\end{theorem}
\begin{proof}
We will  induce 
the comparison map $\iota_{\rig}$ in our theorem by  a pair of maps  $(\iota^1_{\rig},\iota^2_{\rig})$ defined below that make the following diagram commute (in $\sd(C_{\Q_p})$)
$$
\xymatrix{
  [\R\Gamma_{\rig}(X_0/\so_F^0) \wh{\otimes}^R_{F}\wh{\B}^+_{\st}]^{\phi=p^r} \ar[r]^-{\iota_{\hk}\otimes\iota}\ar[d]^{\iota^1_{\rig}} &   (\R\Gamma_{\dr}(X_K) \wh{ \otimes}^R_{K}{\B}^+_{\dr})/F^r\ar[d]^{\iota_{\rig}^2}\\
  [\R\Gamma_{\crr}(\wh{X}/r^{\rm PD}_{\varpi},\Q_p) \wh{\otimes}^R_{r^{\rm PD}_{\varpi,\Q_p}}\wh{\B}^+_{\st}]^{\phi=p^r} \ar[r]^-{p_{\varpi}\otimes\iota} &   (\R\Gamma_{\dr}(\wh{X}_K) \wh{ \otimes}^R_{K}{\B}^+_{\dr})/F^r.
  }
$$
The map $\iota^1_{\rig}$ is compatible with monodromy. Both maps $\iota^1_{\rig},\iota^2_{\rig}$ are quasi-isomorphisms if $X$ is Stein.

\smallskip
\noindent
{\em {\rm (i)} Definition of maps $\iota^1_{\rig}$ and $\iota_{\rig}^2$}. 

\smallskip
$(\star )$ {\em Map $\iota_{\rig}^2.$} The map $\iota_{\rig}^2$, the easier of the two maps,  is just the map from de Rham cohomology of a weak formal scheme to de Rham cohomology of its completion; in the case $X$ is Stein, it is an isomorphism induced by the  canonical identification of coherent cohomology of a partially proper dagger space and its rigid analytic avatar \cite[Theorem 2.26]{GK0}.

\smallskip
$(\star )$ {\em Map $\iota_{\rig}^1.$} To define the map $\iota_{\rig}^1$, 
consider first the change of convergence map
\begin{equation}
\label{morph0}
\R\Gamma_{\rig}(X_0/\so_F^0)\widehat{\otimes}^R_{F}\widehat{\B}^+_{\st}\to \R\Gamma_{\crr}(X_0/\so_F^0,F)\widehat{\otimes}^R_{F}\widehat{\B}^+_{\st}.
\end{equation}
It is compatible with Frobenius and monodromy. 
We claim that if $X$ is Stein it is a strict quasi-isomorphism. It suffices to show this for the change of topology map
\begin{equation}
\label{morph01}
\R\Gamma_{\rig}(X_0/\so_F^0)\to \R\Gamma_{\crr}(X_0/\so_F^0,F). 
\end{equation}
But before we proceed to do that, a small digression about {\em convergent cohomology} that we will use. 
\begin{remark}({\em Review of convergent cohomology}).

Contrary to the case of rigid cohomology, the theory  of (relative) convergent cohomology is well developed   \cite{Og, Sh, Sh1}. Recall the key points. The set up is the following: the base $\sbb$ is a $p$-adic formal log-scheme over $\so_F$, $B:=\sbb_1$; we look at convergent cohomology of $X$ over $\sy$, where $X$ is a log-scheme over $B$ and $\sy$ is a $p$-adic formal log-scheme over $\sbb$.
\begin{enumerate}
\item There exist a convergent site defined in analogy with the crystalline site, where the role of PD-thickenings (analytically, objects akin to closed discs
 of a specific radius $<1$) is played by enlargements ($p$-adic formal schemes; analytically, closed discs of any radius $<1$). Convergent cohomology is defined as the cohomology of the rational structure sheaf on this site.
\item {\em Invariance under infinitesimal thickenings}. If $i:X\to X^{\prime}$ is a homeomorphic exact closed  immersion then the pullback functor $i^*$ induces a quasi-isomorphism on convergent cohomology \cite[0.6.1]{Og}, \cite[Prop. 3.1]{Sh1}.
\item {\em Poincar\'e Lemma}. It states that, locally,  convergent cohomology can be computed by de Rham complexes of convergent tubes in $p$-adic formal log-schemes log-smooth over the base (playing the role of PD-envelopes) \cite[Theorem 2.29]{Sh1}. Analytically, this means that the fixed closed discs used in the crystalline theory are replaced by open discs. 
\item There is a natural map from convergent cohomology to crystalline cohomology. It is a quasi-isomorphism for   log-schemes log-smooth over $\sy_1$ \cite[Theorem 2.36]{Sh1}. In particular\footnote{This can be easily seen by looking locally  at the de Rham complexes computing both sides.}, for a semistable scheme $X_0$ over $k^0$,  there is a  natural quasi-isomorphism
$$
\R\Gamma_{\conv}(X_0/\so_F^0)\stackrel{\sim}{\to} \R\Gamma_{\crr}(X_0/\so_F^0,F).
$$
\item 
\begin{proposition}{\rm(\cite[Theorem 5.3]{GK2})} 
\label{rig-conv}Let $X$ be a semistable weak formal scheme over $\so_K$. 
Assume that all irreducible components of $X_0$ are proper. Then
the natural morphism (induced by mapping weak formal log-schemes to their completions)
$$
\R\Gamma_{\rig}(X_0/\so_F^0)\to \R\Gamma_{\conv}(X_0/\so_F^0)
$$is a quasi-isomorphism.
\end{proposition}
\begin{proof}
Recall that we have  two compatible weight spectral sequences \cite[5.2, 5.3]{GK2}
\begin{align*}
E^{-k,i+k}_1 & =\bigoplus_{j\geq 0, j\geq -k}\bigoplus _{S\in \Theta_{2j+k+1}}H^{i-2j-k}_{\rig}(S/\so_F)\Rightarrow H^i_{\rig}(X_0/\so_F^0),\\
E^{-k,i+k}_1 & =\bigoplus_{j\geq 0, j\geq -k}\bigoplus _{S\in \Theta_{2j+k+1}}H^{i-2j-k}_{\conv}(S/\so_F)\Rightarrow H^i_{\conv}(X_0/\so_F^0)
\end{align*}
Here  $\Theta_j$ denotes the set of all intersections $S$ of $j$ different irreducible components of $X$ that are equipped with trivial log-structures. By assumptions, they are smooth and proper over $k$.
It suffices then to prove that the maps $H^*_{\rig}(S/\so_F)\to H^*_{\conv}(S/\so_F)$ are isomorphisms. Since $S$ is proper this is classical \cite[Prop. 1.9]{Ber}.
\end{proof}
\end{enumerate}
Hence, $\R\Gamma_{\conv}(X_0/\so_F^0)$ is a convergent version of $\R\Gamma_{\rig}(X_0/\so_F^0)$, i.e., a version where we use $p$-adic formal schemes and rigid spaces instead of weak formal schemes and dagger spaces.  
\end{remark}

Now, coming back to the change of topology map (\ref{morph01}), note that it  factors as
\begin{equation}
\label{morph02}
\R\Gamma_{\rig}(X_0/\so_F^0)\to \R\Gamma_{\conv}(X_0/\so_F^0)\stackrel{\sim}{\to} \R\Gamma_{\crr}(X_0/\so_F^0,F).
\end{equation}
We topologized the  convergent cohomology  in the same way we did  rigid cohomology. The second map is quasi-isomorphism because $X_0$ is semistable (hence, locally, admitting liftings). By Proposition \ref{rig-conv}, the first map is a quasi-isomorphism as well.  We claim that the composition (\ref{morph02})  is a strict quasi-isomorphism. We will reduce checking this to $X_0$ proper where it will be clear. Take the 
subschemes $\{U_i\}, \{Y_i\}, i\in\N,$ of $X_0$ as in Section \ref{assumptions}. We have strict quasi-isomorphisms
$$
\R\Gamma_{\rig}(X_0/\so_F^0)\stackrel{\sim}{\to}\holim_i\R\Gamma_{\rig}(U_i/\so_F^0)\stackrel{\sim}{\to}
\holim_i\R\Gamma_{\rig}(Y_i/\so_F^0).
$$
The first one by Example \ref{LF-tensor}, the second one trivially. We have similar strict quasi-isomorphisms for the crystalline cohomology
$$
\R\Gamma_{\crr}(X_0/\so_F^0,F)\stackrel{\sim}{\to}\holim_i\R\Gamma_{\crr}(U_i/\so_F^0,F)\stackrel{\sim}{\to}
\holim_i\R\Gamma_{\crr}(Y_i/\so_F^0,F).
$$
The second one is again trivial. The first one follows from the fact that it is a quasi-isomorphism of complexes of Fr\'echet spaces. It remains to show that the natural map
$$
\holim_i\R\Gamma_{\rig}(Y_i/\so_F^0)\to \holim_i\R\Gamma_{\crr}(Y_i/\so_F^0)
$$
is a strict quasi-isomorphism. Or that, so is the natural map
$$
\R\Gamma_{\rig}(Y_i/\so_F^0)\to \R\Gamma_{\crr}(Y_i/\so_F^0,F).
$$
We factor this map as
$$
\R\Gamma_{\rig}(Y_i/\so_F^0)\to \R\Gamma_{\conv}(Y_i/\so_F^0)\to \R\Gamma_{\crr}(Y_i/\so_F^0,F).
$$
Since the idealized log-scheme $Y_i$ is ideally log-smooth over $k^0$ the second map is a quasi-isomorphism. Since $Y_i$ is also  proper, so is the first map. This finishes the proof  that the map (\ref{morph01}) is a strict quasi-isomorphism.

  We will define now the  following  functorial quasi-isomorphism  in $\sd(C_{\Q_p})$
\begin{align}
\label{morph1}h_{\crr}: [\R\Gamma_{\crr}(X/r^{\rm PD}_{\varpi},\Q_p)
\widehat{\otimes}^R_{r^{\rm PD}_{\varpi,\Q_p}}\widehat{\B}^+_{\st}]^{\phi=p^r}{\to}[\R\Gamma_{\crr}(X_0/\so_F^0,F)\widehat{\otimes}^R_{F}\widehat{\B}^+_{\st}]^{\phi=p^r}.
\end{align}
    We may assume that  $X$ is  quasi-compact. 
Let $J$ be the kernel of the map $p_0: r^{\rm PD}_{\varpi}\to \so_F$, $T\mapsto 0$. The map $p_0$ is compatible with Frobenius and monodromy ("$\log T\mapsto\log 0$"). Consider the exact sequence
$$
0\to J_n\to r^{\rm PD}_{\varpi, n}\stackrel{p_0}{\to} \so_{F,n}\to 0.
$$
Tensoring it with $\wh{\A}_{\st,n}$, we get the following exact sequence
$$
0\to J_n\otimes_{r^{\rm PD}_{\varpi,n}}\wh{\A}_{\st,n}\to \wh{\A}_{\st, n}\stackrel{p_0}{\to} \so_{F,n}\otimes_{r^{\rm PD}_{\varpi,n}}\wh{\A}_{\st,n}\to 0
$$
We used here the fact that $\wh{\A}_{\st,n}$ is flat over $r^{\rm PD}_{\varpi,n}$. Going to limit with $n$, we get the exact sequence
\begin{equation}
\label{paris1}
0\to J\wh{\otimes}_{r^{\rm PD}_{\varpi}}\wh{\A}_{\st}\to \wh{\A}_{\st}\stackrel{p_0}{\to} \so_{F}\wh{\otimes}_{r^{\rm PD}_{\varpi}}\wh{\A}_{\st}\to 0
\end{equation}
Set $E:=J\wh{\otimes}_{r^{\rm PD}_{\varpi}}\wh{\A}_{\st}$.

   Tensoring the last sequence with $\R\Gamma_{\crr}(X_n/r^{\rm PD}_{\varpi,n})$ and $\R\Gamma_{\crr}(X_0/\so^0_{F,n})$, respectively,   we obtain the following distinguished triangles
   \begin{align*}
 &  E_n\otimes^L_{r^{\rm PD}_{\varpi,n}}\R\Gamma_{\crr}(X_n/r^{\rm PD}_{\varpi,n}) \to \wh{\A}_{\st,n}\otimes^L_{r^{\rm PD}_{\varpi,n}}\R\Gamma_{\crr}(X_n/r^{\rm PD}_{\varpi,n})\stackrel{p_0}{\to }\so_{F,n}{\otimes}_{r^{\rm PD}_{\varpi,n}}\wh{\A}_{\st,n}\otimes^L_{r^{\rm PD}_{\varpi,n}}\R\Gamma_{\crr}(X_n/r^{\rm PD}_{\varpi,n}),\\
 & E_n\otimes^L_{\so_{F,n}}\R\Gamma_{\crr}(X_0/\so^0_{F,n}) \to \wh{\A}_{\st,n}\otimes^L_{\so_{F,n}}\R\Gamma_{\crr}(X_0/\so^0_{F,n})\stackrel{p_0}{\to} \so_{F,n}{\otimes}_{r^{\rm PD}_{\varpi,n}}\wh{\A}_{\st,n}\otimes^L_{\so_{F,n}}\R\Gamma_{\crr}(X_0/\so^0_{F,n}).
   \end{align*}
   The last terms in these triangles are quasi-isomorphic. Indeed, by direct local computations we see that the natural map
   $$
   \R\Gamma_{\crr}(X_n/r^{\rm PD}_{\varpi,n})\otimes^L_{r^{\rm PD}_{\varpi,n}}\so_{F,n}\to \R\Gamma_{\crr}(X_0/\so^0_{F,n})
   $$ 
   is a quasi-isomorphism. Hence
   \begin{align*}
   \so_{F,n}{\otimes}_{r^{\rm PD}_{\varpi,n}}\wh{\A}_{\st,n}\otimes^L_{r^{\rm PD}_{\varpi,n}}\R\Gamma_{\crr}(X_n/r^{\rm PD}_{\varpi,n}) & \simeq \wh{\A}_{\st,n}\otimes^L_{r^{\rm PD}_{\varpi,n}}\R\Gamma_{\crr}(X_0/\so^0_{F,n})\\
 &\simeq \wh{\A}_{\st,n}\otimes^L_{r^{\rm PD}_{\varpi,n}}\so_{F,n}\otimes^L_{\so_{F,n}}\R\Gamma_{\crr}(X_0/\so^0_{F,n}).
   \end{align*}
   
   The complexes
   \begin{equation}
   \label{acyclic}
      [ E_{\Q_p}\wh{\otimes}^R_{r^{\rm PD}_{\varpi,\Q_p}}\R\Gamma_{\crr}(X/r^{\rm PD}_{\varpi},\Q_p)]^{\phi=p^r},\quad  [E_{\Q_p}\wh{\otimes}^R_{{F}}\R\Gamma_{\crr}(X_0/\so^0_{F},F)]^{\phi=p^r}
   \end{equation}
   are strictly acyclic\footnote{In fact, they are both isomorphic to the trivial complex.}: this is an immediate consequence of the fact that Frobenius $\phi$ is highly topologically nilpotent on $J$
   (hence $p^r-\phi$ is rationally invertible). This implies that the following maps
  \begin{equation}
  \label{zigzag}
   [\R\Gamma_{\crr}(X/r^{\rm PD}_{\varpi},\Q_p)\wh{\otimes}^R_{r^{\rm PD}_{\varpi,\Q_p}}\widehat{\B}^+_{\st}]^{\phi=p^r}\stackrel{p_0}{\to}
   [\R\Gamma_{\crr}(X_0/\so_F^0,F) \wh{\otimes}^R_{F}(\widehat{\B}^+_{\st}/E_{\Q_p})]^{\phi=p^r}
   \stackrel{1\otimes p_{0}}{\longleftarrow} [\R\Gamma_{\crr}(X_0/\so_F^0,F) \wh{\otimes}^R_{F}\widehat{\B}^+_{\st}]^{\phi=p^r}
\end{equation}
are strict quasi-isomorphisms. We define the map $h_{\crr}$  to be equal to the above zigzag. It is compatible with the monodromy operator  (for the first map in the zigzag use the fact
that  the monodromy  operator is defined by compatible residue maps).

   We define a map in $\sd(C_{\Q_p})$
\begin{equation}
\label{zurich2}
\iota_{\rig}^1:
[\R\Gamma_{\rig}(X_0/\so_F^0)  \wh{\otimes}^R_{F}\widehat{\B}^+_{\st}]^{\phi=p^r}{\to} [\R\Gamma_{\crr}(X/r^{\rm PD}_{\varpi},\Q_p)\widehat{\otimes}^R_{r^{\rm PD}_{\varpi,\Q_p}}\widehat{\B}^+_{\st}]^{\phi=p^r}
\end{equation}
as the composition of  the maps in (\ref{morph0}) and (\ref{morph1}).
    Both maps  being compatible with the
   monodromy operator  so is $\iota^1_{\rig}$.  

\smallskip
\noindent
{\em {\rm (ii)} Compatibility of the maps $\iota^1_{\rig},\iota^2_{\rig}$.} 
Let $r\geq 0$. The compatibility of the maps $\iota^1_{\rig},\iota^2_{\rig}$ can be shown by the  commutative diagram
 $$
\xymatrix@C=0.5pt{ [\R\Gamma_{\rig}(X_0/\so_F^0) \wh{\otimes}^R_{F}\wh{\B}^+_{\st}]^{\phi=p^r}   \ar[dd]_{\wr} \ar[rr]^-{\iota_{\hk}\otimes\iota}\ar@/_80pt/[ddd]^{\iota_{\rig}^1}_{\wr}     &  &  (\R\Gamma_{\dr}(X_K)\wh{\otimes}^R_{K}{\B}^+_{\dr})/F^r\ar[ddd]^{\iota_{\rig}^2}_{\wr}\\
 &     
 [\R\Gamma_{\rig}(\overline{X}_0/r^{\dagger})\wh{\otimes}^R_{F}\widehat{\B}^+_{\st}]^{\phi=p^r}\ar[lu]_{p_0\otimes 1}^{\sim}\ar[d]^{f_1}     \ar[ru]^-{p_{\varpi}\otimes\iota}     \\
  [\R\Gamma_{\crr}(X_0/\so_F^0,F)\wh{\otimes}^R_{F}\widehat{\B}^+_{\st}]^{\phi=p^r}
& 
 [\R\Gamma_{\rig}({X}_0/r^{\dagger})\wh{\otimes}^R_{F}\widehat{\B}^+_{\st}]^{\phi=p^r} \ar[luu]^{h_{\rig}}\ar[dl]^{f_2}  \ar[ruu]_-{p_{\varpi}\otimes\iota}  \\
[\R\Gamma_{\crr}(X/r^{\rm PD}_{\varpi},\Q_p)\wh{\otimes}^R_{r^{\rm PD}_{\varpi,\Q_p}}\widehat{\B}^+_{\st}]^{\phi=p^r}\ar[u]^{h_{\crr}}_{\wr} \ar[rr]^-{p_{\varpi}\otimes\iota}
   &  & (\R\Gamma_{\dr}(\wh{X}_K)\wh{\otimes}^R_{K}{\B}^+_{\dr})/F^r.
}
$$
Here the map $f_2$ is the change of convergence map defined by the composition
$$
f_2:\R\Gamma_{\rig}({X}_0/r^{\dagger})\to \R\Gamma_{\conv}({X}_0/\hat{r})\stackrel{\sim}{\leftarrow} \R\Gamma_{\conv}({X}_1/\hat{r})\to\R\Gamma_{\crr}(X/r^{\rm PD}_{\varpi},\Q_p),
$$
where $\hat{r}:=\so_F\{T\}$. The quasi-isomorphism is actually a natural isomorphism by the invariance under infinitesimal thickenings. The map $f_2$ is clearly compatible with the projection $p_{\varpi}$  and the map $\iota^2_{\rig}$.
The map $h_{\rig}$ is defined in the same way as the map $h_{\crr}$: we just replace ${\crr}$ by $\rig$ and $r^{\rm PD}_{\varpi}$ by $r^{\dagger}$.  It is clear that the maps $h_*$ are compatible. 

The map
   $f_1$
   is induced by  the composition 
   $$\R\Gamma_{\rig}(\overline{X}_0/r^{\dagger})= \R\Gamma_{\rig}((P_{\jcdot},V_{\jcdot})/r^{\dagger})\to  \R\Gamma_{\rig}(M^{\prime}_{\jcdot}/r^{\dagger})\stackrel{\sim}{\to}     \R\Gamma_{\rig}(M_{\jcdot}/r^{\dagger})\stackrel{\sim}{\leftarrow}    
  \R\Gamma_{\rig}(X_0/r^{\dagger}).
   $$
  The fact that the last quasi-isomorphism is strict needs a justification. We may assume that $X_0$ is affine and take a log-smooth lifting $Y$ of $X_0$ to $r^{\dagger}$. Since the sheaf of differentials of $Y_F$ is free we are reduced to showing strict acyclicity of the \v{C}ech complex of overconvergent functions for the covering corresponding to $\{M_{i}\},i\in I$. Using a dagger presentation of $Y_F$, this complex can be written as an inductive limit of \v{C}ech complexes for analogous  coverings of rigid analytic affinoids. The latter complexes being strictly acyclic (because they are acyclic and we have the Open Mapping Theorem for Banach spaces) and the inductive system being acyclic, the former complex is acyclic as well.
  For the above diagram we need  strictness of the last quasi-isomorphism  with terms tensored with $\wh{\B}^+_{\st}$ but this is automatic since we have taken derived tensor products. Finally, 
  it is easy to check (do it first without the period ring $\wh{\B}^+_{\st}$) that the map $f_1$ makes the two small adjacent triangles in the above diagram commute.
   \end{proof}
\subsection{Fundamental diagram}
Having the comparison theorem proved above, we can now deduce a fundamental diagram for pro-\'etale cohomology from the one for overconvergent syntomic cohomology. 
\begin{theorem}
\label{fdd}
Let $X$ be a Stein semistable weak formal scheme over $\so_K$. Let $r\geq 0$. There is  a natural map of strictly exact sequences of Fr\'echet spaces
  $$
\xymatrix@C=.6cm{
0\ar[r] &  \Omega^{r-1}(X_C)/\ker d\ar[r]\ar@{=}[d] & H^r_{\proeet}(X_{C},\Q_p(r))\ar[d]^{\tilde{\beta}} \ar[r] & (H^r_{\hk}(X_{0})\wh{\otimes}_F\B^+_{\st})^{N=0,\phi=p^r}\ar[r]\ar[d]^{\iota_{\hk}\otimes\theta} & 0\\
0\ar[r] &  \Omega^{r-1}(X_C)/\ker d \ar[r]^-d & \Omega^r(X_C)^{d=0} \ar[r] & H^r_{\dr}(X_C)\ar[r] & 0
}
$$
Moreover, the vertical maps  have closed images, and 
$
\ker\tilde{\beta}\simeq  (H^r_{\hk}(X_{0})\wh{\otimes}_F\B^+_{\st})^{N=0,\phi=p^{r-1}}$.
\end{theorem}
\begin{proof}
We define $\tilde{\beta}:=p^{-r}\beta\iota_{\rig}^{-1}\iota_{\rm BK}^{-1}\alpha_{\rm FM}^{-1}$, using Corollary \ref{fini1}, Proposition \ref{fini2}, and Theorem \ref{fini3}; the twist by $p^{-r}$ being added to make this map compatible with symbols.
The theorem  follows immediately from 
 Proposition \ref{fd11}.
\end{proof}
\begin{remark}The above diagram can be thought of as a one-way comparison theorem, i.e., the pro-\'etale cohomology $H^r_{\proeet}(X_{C},\Q_p(r))$ is the pullback of the diagram $$
(H^r_{\hk}(X_{0})\wh{\otimes}_F\B^+_{\st})^{N=0,\phi=p^r}\verylomapr{\iota_{\hk}\otimes\theta}H^r_{\dr}(X_C) \stackrel{\can}{\longleftarrow}\Omega^r(X_C)^{d=0}$$ built from the Hyodo-Kato cohomology and a piece of the de Rham complex. 
For a striking comparison, 
recall that if $X$ is a proper semistable formal scheme over $\so_K$ then the Semistable Comparison Theorem from \cite{CN} shows that we have the exact sequence
\begin{equation}
\label{proper}
0\to H^r_{\proeet}(X_{C},\Q_p(r))\to  (H^r_{\hk}(X_{0})\wh{\otimes}_F\B^+_{\st})^{N=0,\phi=p^r}\verylomapr{\iota_{\hk}\otimes\iota}  (H^r_{\dr}(X_K)\wh{\otimes}_K\B^+_{\dr})/F^r,
\end{equation}
i.e., the pro-\'etale cohomology $H^r_{\proeet}(X_{C},\Q_p(r))$ is the pullback of the diagram $$
(H^r_{\hk}(X_{0})\wh{\otimes}_F\B^+_{\st})^{N=0,\phi=p^r}\verylomapr{\iota_{\hk}\otimes\theta}(H^r_{\dr}(X_K)\wh{\otimes}_K \B^+_{\dr})/F^r {\longleftarrow} 0.$$ 
Of course, in this case the  \'etale and the pro-\'etale cohomologies agree. The sequence  (\ref{proper}) is obtained in an analogous  way to the top sequence in the fundamental diagram above. With the degeneration of the Hodge-de Rham spectral sequence and the theory of finite dimensional BC spaces forcing the injectivity on the left.
\end{remark}
\begin{remark}
The following 
commutative diagram illustrates the relationship between syntomic cohomology of $\Q_p(r)$  and $\Q_p(r-1)$
$$\xymatrix@C=.4cm@R=.5cm{
...\ar[d]^t\ar[r]&{\rm Syn}^{r-2}_{r-1}\ar[d]^t_{\wr}\ar[r]&{\rm HK}_{r-1}^{r-2}\ar[d]^t\ar[r]&{\rm DR}_{r-1}^{r-2}\ar[d]^t\ar[r]^{\partial}
&{\rm Syn}^{r-1}_{r-1}\ar[d]^t_{\wr}\ar[r]&{\rm HK}_{r-1}^{r-1}\ar[r]\ar[d]^t&0\ar[d]\ar[r]
\ar[d]&0\ar[d]\ar[r]&0\ar[d]\ar[r]&0\\
...\ar[r]\ar[d]&{\rm Syn}^{r-2}_r\ar[r]\ar[d]&{\rm HK}_r^{r-2}\ar[r]\ar[d]^{\theta}&{\rm DR}_r^{r-2}\ar[r]^{\partial}
\ar[d]^{\theta}&{\rm Syn}^{r-1}_r\ar[r]\ar[d]&{\rm HK}_r^{r-1}\ar[r]\ar[d]^{\theta}&{\rm DR}_r^{r-1}\ar[r]^{\partial}
\ar[d]^{\id}&{\rm Syn}^{r}_r\ar[r]\ar[d]^{\id}&{\rm HK}_r^{r}\ar[r]\ar[d]^{\id}&0\\
...\ar[r]&0\ar[r]&H^{r-2}_{\rm dR}\ar[r]^{\id}&H^{r-2}_{\rm dR}\ar[r]&0\ar[r]
&H^{r-1}_{\rm dR}\ar[r]
&{\rm DR}_r^{r-1}\ar[r]^{\partial}
&{\rm Syn}^{r}_r\ar[r]&{\rm HK}_r^{r}\ar[r]&0
}$$
Here we wrote ${\rm HK}^i_r,{\rm DR}^i_r,$ and ${\rm Syn}^i_r$ for the $i$'th cohomology of the complexes ${\rm HK}(X_{C},r),{\rm DR}(X_{C},r),$ and $\R\Gamma_{\synt}(X_{C},\Q_p(r))$, respectively.  We set $H^i_{\dr}:=H^i_{\dr}(X_C)$.

   We claim that the rows of the above diagram are strictly exact. Indeed, the two top rows  arise from the definition of syntomic cohomology $\R\Gamma_{\synt}(X_{C},\Q_p(r-1))$ and $\R\Gamma_{\synt}(X_{C},\Q_p(r))$; the map between them is the multiplication by  $t\in (\B_{\crr}^+)^{\varphi=p}\cap F^1\B_{\dr}$. These rows are clearly strictly exact. It suffices now to show  that the columns form short strictly exact sequences (with zeros at the ends). Indeed, for $i\leq r-1$, multiplication by $t$ induces an isomorphism
 (using comparison with pro-\'etale cohomology)
$${\rm Syn}^i_r\cong t{\rm Syn}^i_{r-1}$$
  as well as the following strictly exact sequences
\begin{align}
\label{lyon2}
0\to &  {\rm DR}^i_{r-1}\stackrel{t}{\to} {\rm DR}^i_r\to H^i_{\rm dR}(X_K)\wh{\otimes}_K C\to 0,\quad r\geq i+2,\\
0\to &  {\rm HK}^i_{r-1}\stackrel{t}{\to} {\rm HK}^i_r\to H^i_{\rm dR}(X_K)\wh{\otimes}_K C\to 0.\notag
\end{align}
The first strictly exact sequence  follows from  
 the strictly exact sequence 
$$0\to \gr^{r-i-1}_F\B^+_{\dr}\wh{\otimes}_K (\Omega^i(X_K)/d\Omega^{i-1}(X_K))
\to {\rm DR}^i_r\to (\B_{\dr}^+/F^{r-i-1})\wh{\otimes}_K H^i_{\rm dR}(X_K)\to 0;$$
the second one from Lemma \ref{lyon3}.
 \end{remark} 
\subsection{Examples}\label{EXA}
We will now illustrate Theorem \ref{fdd} with some simple examples.
\subsubsection{Affine space} Let $d\geq 1$. Let ${\mathbb A}^d_K$ be the  $d$-dimensional rigid analytic affine space over $K$.  Recall that 
$H^r_{\eet}({\mathbb A}_{C}^d,\Q_p)=0$ for $ r \geq 1$ \cite[Theorem 7.3.2]{Ber}. On the other hand, as the following proposition shows,  the pro-\'etale cohomology of
${\mathbb A}_{C}^d$ is highly nontrivial in nonzero degrees.
\begin{proposition}Let $r\geq 1$.
There is a   $\sg_K$-equivariant isomorphism  in $C_{\Q_p}$ (of Fr\'echet spaces)
$$
\Omega^{r-1}({\mathbb A}^d_C)/\ker d\stackrel{\sim}{\to} 
H^r_{\proeet}({\mathbb A}_{C}^d,\Q_p(r)).
$$
\end{proposition}
\begin{remark}
A simpler and more direct proof of this result (but still using syntomic cohomology) has been given in \cite{CN2}. 
See \cite{AC} for another proof working directly with the fundamental exact sequence in the pro-\'etale topology.
\end{remark}
\begin{proof}Let $\sa^d$ denote a semistable weak formal scheme over $\so_K$ such that $\sa^d_K\simeq {\mathbb A}^d_K$. We will explain below how such a model $\sa^d$ can be constructed.
By Theorem \ref{fdd}, we have a $\sg_K$-equivariant exact sequence (in $C_{\Q_p}$)
 $$
\xymatrix@C=.6cm{
0\ar[r] &  \Omega^{r-1}({\mathbb A}^d_C)/\ker d\ar[r] & H^r_{\proeet}({\mathbb A}^d_{C},\Q_p(r))\ar[r]& (H^r_{\hk}(\sa^d_{0})\wh{\otimes}_F\B^+_{\st})^{N=0,\phi=p^r}\ar[r]
& 0
}
$$
Recall that $H^r_{\dr}({\mathbb A}^d_K)=0$. Since, by the Hyodo-Kato isomorphism $H^r_{\hk}(\sa^d_0)\otimes_FK\simeq H^r_{\dr}({\mathbb A}^d_K)$, we have $H^r_{\hk}(\sa^d_0)=0$. Our proposition follows from the above exact sequence.

  It remains to show that we can construct a semistable weak formal scheme $\sa^d$ over $\so_K$ whose generic fiber is ${\mathbb  A}^d_K$. For $d=1$, we can define a model $\sa^1$ using Theorem 4.9.1 of \cite{FDP}. That theorem describes a construction of a formal semistable model for any analytic subspace ${\mathbb P}_K\setminus\sll^*$, where $\sll$ is an infinite compact subset of $K$-rational points of the projective line ${\mathbb P}_K$ and $\sll^*$ is the set of its limit points. The proof of the theorem   can be easily modified to yield a weak formal model. To define the model $\sa^1$ we want we  apply this theorem with $\sll=\{\infty\}\cup \{\varpi^n|n\in\Z, n\leq 0\}$. We note that the special fiber of $\sa^1$ is a half line  of projective lines.
  
    To construct a model $\sa^d$ for $d>1$, first we consider the $d$-fold product $Y$ of the logarithmic weak formal scheme associated to $\sa^1$. Product is taken over $\so_K^{\times}$. It is not  a semistable scheme but it is log-smooth over $\so_K^{\times}$. Hence its singularities can be resolved using combinatorics of monoids describing the log-structure. In fact, using Lemma 1.9 of \cite{SW},  one can define a canonical ideal sheaf of $Y$ that needs to be blown-up to obtain a semistable model $\sa^d$ we want. 
\end{proof}
\subsubsection{Torus}
 Let $d\geq 1$. Let ${\mathbb G}^d_{m,K}$ be the  $d$-dimensional rigid analytic torus over $K$. 
Let $\sy^d$ denote a semistable weak formal scheme over $\so_K$ such that $\sy^d_K\simeq {\mathbb G}^d_{m,K}$. Such a model $\sy^d$ exists.
For $d=1$, we can define a model $\sy^1$ using Theorem 4.9.1 of \cite{FDP}; just as in the case of ${\mathbb A}^1_K$ above. More specifically, 
to define the model $\sy^1$ we want we  apply this theorem with $\sll=\{\infty, 0\}\cup \{\varpi^n|n\in\Z\}$. We note that the special fiber of $\sy^1$ is a line  of projective lines.
  To construct a model $\sy^d$ for $d>1$, we use products as above.
  
  To compute the pro-\'etale cohomology, we will use    Theorem \ref{fdd}. To make it  explicit, we need to compute $
  (H^r_{\hk}(\sy^d_{0})\wh{\otimes}_F\B^+_{\st})^{N=0,\phi=p^r}$. 
For $d=1$,
  we have
  $$
  H^r_{\dr}({\mathbb G}_{m,K})=\begin{cases}
  K  & \mbox{ if } r=0,\\
  c_1^{\dr}(z) K&  \mbox{ if } r=1,\\
  0  & \mbox{  if } r>1.
  \end{cases}
  $$
  Here $z$ is a coordinate of the torus and $c_1^{\dr}(z)$ is its de Rham Chern class, i.e. $dz/z$ 
(see Appendix \ref{symbols}). 
  For $d>1$, we can use the K\"unneth formula to compute that $H^r_{\dr}({\mathbb G}^d_{m,K})$ is a $K$-vector space of dimension $\binom{d}{r}$ generated by 
  the tuples $c_1^{\dr}(z_{i_1})\cdots c_1^{\dr}(z_{i_r})$. Similarly, 
  $H^r_{\hk}(\sy^d_{0})$ is an $F$-vector space of dimension $\binom{d}{r}$ generated by the tuples $c_1^{\hk}(z_{i_1})\cdots c_1^{\hk}(z_{i_r})$. By Lemma \ref{compatibility}, the Hyodo-Kato and the de Rham symbols are compatible under the Hyodo-Kato map $\iota_{\hk}$.

Since $\phi (c_1^{\hk}(z_{i_j}))=pc_1^{\hk}(z_{i_j})$ and $N(c_1^{\hk}(z_{i_j}))=0$, we get  that    $$
(H^r_{\hk}(\sy^d_{0})\wh{\otimes}_F\B^+_{\st})^{N=0,\phi=p^r}=H^r_{\hk}(\sy^d_{0})^{\phi=p^r}=\wedge^r\Q_p^d
$$  
 and that it  is a $\Q_p$-vector space of dimension $\binom{d}{r}$ generated by the tuples $c_1^{\hk}(z_{i_1})\cdots c_1^{\hk}(z_{i_r})$.
  Hence, Theorem \ref{fdd} gives us   a map of $\sg_K$-equivariant exact sequences (in $C_{\Q_p}$)
 $$
\xymatrix{0\ar[r] &  \Omega^{r-1}({\mathbb G}^d_{m,C})/\ker d\ar@{=}[d]\ar[r] &
 H^r_{\proeet}({\mathbb G}^d_{m,C},\Q_p(r))\ar[r]\ar@{^(->}[d]^{\tilde{\beta}} & \wedge^r\Q_p^d\ar[r] \ar@{^(->}[d]^{\can}& 0\\
0\ar[r] &  \Omega^{r-1}({\mathbb G}^d_{m,C})/\ker d
\ar[r]^-d & \Omega^r({\mathbb G}^d_{m,C})^{d=0} \ar[r] & 
\wedge^rC^d\ar[r] & 0
}
$$

Recall, for comparison, that $H^r_{\eet}({\mathbb G}^d_{m,C},\Q_p(r))\simeq \wedge^r\Q_p^d
$, 
  a $\Q_p$-vector space generated by the tuples $c_1^{\eet}(z_{i_1})\cdots c_1^{\eet}(z_{i_r})$.

\subsubsection{Curves}
Let $X$ be a Stein curve over $K$ with a semistable model $\sx$ over $\so_K$. The diagram from Theorem \ref{fdd} takes the following form\footnote{We note here that the conditions of that theorem are always satisfied for curves.}
$$\xymatrix@R=.5cm@C=.6cm{
 0\ar[r]&C^{\pi_0(X)}\ar@{=}[d]\ar[r]&
\mathcal{O}(X_C)\ar[r]^-{\rm exp}\ar@{=}[d]& H^1_{\proeet}(X_C, \Q_p(1))\ar[r]\ar[d]^-{\rm dlog}&
(H^1_{\hk}(\sx_0)\wh{\otimes}_F\B^+_{\st})^{\phi=p,N=0}\ar[r]\ar[d]^-{\rm \iota_{\hk}\otimes\theta}&0\\
0\ar[r]&C^{\pi_0(X)}\ar[r]&
\mathcal{O}(X_C)\ar[r]^-{d} & \Omega^1(X_C) \ar[r]&
H^1_{\rm dR}(X)\wh{\otimes}_KC\ar[r]&0
}$$        
\section{Pro-\'etale cohomology of Drinfeld half-spaces}
We will use the fundamental diagram of Theorem \ref{fdd} 
to compute the $p$-adic pro-\'etale cohomology of Drinfeld half-spaces (Theorem~\ref{PROET}).
This boils down to understanding the Hyodo-Kato cohomology groups as $(\varphi,N)$-modules
and as representations of ${\rm GL}_{d+1}(K)$.  The latter can be done by using
the comparison with de Rham cohomology and results of Schneider-Stuhler (see Theorem~\ref{SSmain});
the computation of $\varphi$ and $N$ uses an explicit description a la Iovita-Spiess (see Theorem~\ref{IS}
and Lemma~\ref{genIS})
of Hyodo-Kato cohomology of Drinfeld half-spaces
in terms of symbols of rational hyperplanes.
\subsection{Drinfeld half-spaces and their standard formal models}\label{formal-models}
    Let $K$ be  a finite extension of $\Q_p$.  Let ${\mathbb H}^d_K$, $d\geq 1$, be the $d$-dimensional Drinfeld half-space over $K$:
 the $K$-rigid space that is the complement in ${\mathbb P}^d_K$ of all $K$-rational hyperplanes. If $\mathcal{H}=\mathbb{P}((K^{d+1})^*)=\mathbb{P}^d(K)$ is the  space of $K$-rational hyperplanes
         in $K^{d+1}$ (this is a profinite set), we have 
   $${\mathbb H}^d_K={\mathbb P}^d_K\setminus \cup_{H\in \mathcal{H}} H.$$ 
 The group  $G:={\rm GL}_{d+1}(K)$ acts on it. 
 We will drop the subscript $K$ if there is no danger of confusion. 
 
  ${\mathbb H}^d_K$ is a rigid analytic Stein space hence also a dagger analytic Stein space. It has a (standard) $G$-equivariant semistable weak formal model $\wt{\mathbb H}^d_K$ \cite[6.1]{GK2} (that is Stein).
Recall that the set of vertices of the Bruhat-Tits building ${\rm BT}$ of ${\rm PGL}_{d+1}(K)$ is the set of homothety classes of lattices in $K^{d+1}$. It corresponds to the set of irreducible components of $Y:=\wt{\mathbb H}^d_{K,0}$.
For $s\geq 0$, let ${\rm BT}_s$ denote the Bruhat-Tits building truncated at $s$, i.e., the simplicial subcomplex of ${\rm BT}$ supported on the vertices $v$ such that the combinatorial distance $d(v,v_0)\leq s$, $v_0=[\so_K^{d+1}]$. Here, for a lattice $L$, $[L]$ denotes the homothety class of $L$. 
 Let $Y_s$ denote the union of the irreducible components corresponding to the vertices of ${\rm BT}_s$. It is a closed subscheme of $Y$ that we equip with the induced log-structure. We will sometimes write $Y_{\infty}$ for the whole special fiber $Y$. We denote by 
 $Y_s^{\circ}:=Y_s\setminus (Y_s\cap \overline{(Y\setminus Y_{s})})$, where the bar denotes closure.
We have immersions $Y_{s-1}\subset Y_s^{\circ}\subset Y_s$, where the first one is closed and the second one is open. 
 \subsection{Generalized Steinberg representations} 
 We will briefly review the definitions and basic properties of the generalized Steinberg representations that we will need.
 \subsubsection{Locally constant special representations}
 \label{intro5}
 Let $B$ be the upper triangular Borel subgroup of $G$ and 
   $\Delta=\{1,2,...,d\}$. We identify the Weyl group $W$ of $G$ with the group of permutations of $\{1,2,...,d+1\}$ and with the subgroup of permutation matrices
   in $G$. Then $W$ is generated by the elements $s_i=(i, i+1)$ for $i\in \Delta$. 
   
   For each subset 
   $J$ of $\Delta$ we let: 
   
   $\bullet$ $W_J$ be the subgroup of $W$ generated by the $s_i$ with $i\in J$. 
   
   $\bullet$ $P_J=BW_JB$, the parabolic subgroup of $G$ generated by $B$ and $W_J$.

   $\bullet$ $X_J=G/P_J$, a compact topological space. 
    
    If $A$ is an abelian group and $J\subset \Delta$, let 
    $${\rm Sp}_J(A)=\frac{{\rm LC}(X_J, A)}{\sum_{i\in \Delta\setminus J} {\rm LC}(X_{J\cup \{i\}}, A)},$$
       where ${\rm LC}$ means locally constant functions with values in $A$ (automatically with compact support since the $X_J$'s are compact). This is 
   a smooth $G$-module over $A$ and we have a natural isomorphism 
   ${\rm Sp}_J(A)={\rm Sp}_J(\mathbf{Z})\otimes A$. For $J=\emptyset$ we obtain the usual Steinberg representation with coefficients in $A$, while for $J=\Delta$ we have ${\rm Sp}_J(A)=A$ (since $X_J$ is a point). For $r\in \{0,1,...,d\}$ we use the simpler notation 
        $${\rm Sp}_r={\rm Sp}_{\{1,2,...,d-r\}}.$$For $r> d$, we set ${\rm Sp}_r=0$.

  \begin{proposition} \label{GK irred}
   If $A$ is a field of characteristic $0$ or $p$ then 
 the ${\rm Sp}_J(A)$'s (for varying $J$) are the irreducible constituents of ${\rm LC}(G/B, A)$, each occurring with multiplicity $1$.    
  \end{proposition}
  
  \begin{proof} This is due to Casselman in characteristic $0$ (see \cite[X, Theorem 4.11]{BW})
  and to Grosse-Kl\"onne \cite[Cor. 4.3]{GKD} in characteristic $p$. 
    \end{proof}
    
    \begin{remark} The proposition does not hold for $A$ a field of characteristic $\ell\ne p$, see
     \cite[III, Theorem 2.8]{MF}. 
           \end{remark}
       
        The rigidity in characteristic $p$ given by the previous theorem has consequences in characteristic $0$ that will be very useful to us later on.
     \begin{corollary}\label{Uniquelattice}
   If $J$ is a subset of $\Delta$, then 
    ${\rm Sp}_J(\so_K)$ is, up to a $K^*$-homothety, the 
    unique $G$-stable $\so_K$-lattice in ${\rm Sp}_J(K)$.
    \end{corollary}
     \begin{proof}
     This follows easily from Proposition \ref{GK irred} and the fact that ${\rm Sp}_J(\so_K)$ is finitely generated over 
   $\so_K[G]$, see \cite[Cor. 4.5]{GKD} for the details. 
\end{proof}
       
       \subsubsection{Topology}
       
         If $\Lambda$ is a topological ring, then ${\rm Sp}_J(\Lambda)$ has a natural topology: the space 
         $X_J$ being profinite, we can write $X_J=\varprojlim_{n} X_{n,J}$ for  finite sets 
         $X_{n,J}$ and then 
          ${\rm LC}(X_J, \Lambda)=\varinjlim_{n} {\rm LC}(X_{n,J}, \Lambda)$, each 
          ${\rm LC}(X_{n,J}, \Lambda)$ being a finite free $\Lambda$-module endowed with the natural topology.       
          In particular, if $\Lambda$ is a finite extension of $\Q_p$, this exhibits ${\rm Sp}_J(\Lambda)$ as an inductive limit
          of finite dimensional $\Lambda$-vector spaces, and the corresponding topology is the strongest locally convex topology on the 
          $\Lambda$-vector space ${\rm Sp}_J(\Lambda)$, which is an LF-space. 
          
          Let $M^*:={\rm Hom}_{\rm cont}(M,\Lambda)$ for any topological $\Lambda$-module $M$, and equip $M^*$ with the weak topology.
          Then ${\rm Sp}_J(\Lambda)^*$ is naturally isomorphic to $\varprojlim_{n} {\rm LC}(X_{n,J}, \Lambda)^*$, i.e. a countable inverse limit of finite free $\Lambda$-modules.
          In particular, if  $L$ is a finite extension of $\Q_p$ then 
          ${\rm Sp}_J(L)^*$ is a nuclear FrŽ\'echet space (in fact a countable product of Banach spaces) and ${\rm Sp}_J(\so_L)^*$ is a compact $\so_L$-module, which is torsion free. 
          Therefore ${\rm Sp}_J(\so_L)^*\otimes L$ has a natural structure of a weak dual of an $L$-Banach space.           
\subsubsection{Continuous special representations}     
Consider now the corresponding 
continuous special representation 
$${\rm Sp}^{\rm cont}_J(\Lambda)=\frac{\mathcal{C}(X_J, \Lambda)}{\sum_{\alpha\in \Delta\setminus J} \mathcal{C}(X_{J\cup \{\alpha\}}, \Lambda)}.$$
Arguing as above, we see that, for any finite extension $L$ of $\Q_p$, the space ${\rm Sp}^{\rm cont}_J(L)$ has a natural structure of an $L$-Banach space, with the 
unit ball given by ${\rm Sp}^{\rm cont}_J(\so_L)$. The action of $G$ on all these spaces is continuous and we can recover ${\rm Sp}_J(L)$ from ${\rm Sp}^{\rm cont}_J(L)$ as the space of 
smooth vectors (for the action of~$G$). 
     
The rigidity in characteristic $p$ given by Proposition \ref{GK irred} and Corollary \ref{Uniquelattice} yields:     
\begin{corollary}\label{Uniquelattice1}
Let $J$ be a subset of $\Delta$ and $L$ a finite extension of $\Q_p$.
       
{\rm a)} The universal unitary completion of ${\rm Sp}_J(L)$ is 
${\rm Sp}^{\rm cont}_J(L)$.
     
{\rm b)} The space of $G$-bounded vectors in ${\rm Sp}_J(L)^*$ is ${\rm Sp}^{\rm cont}_J(L)^*$.
\end{corollary}
  \begin{proof}
a) Follows from Corollary \ref{Uniquelattice} and the fact that 
${\rm St}^{\rm cont}_J(\so_L)$
is the $p$-adic completion of ${\rm Sp}_J(\so_L)$ (which in turn uses that
${\rm Sp}_J(A)={\rm Sp}_J(\mathbf{Z})\otimes A$ for all $A$, and this is a free $A$-module).
  
b) Follows by duality from a). 
\end{proof}
  
\begin{remark}
One can also define a locally analytic generalized Steinberg representation ${\rm Sp}_J^{\rm an}(L)$ for any finite extension 
$L$ or $\Q_p$ (or any closed subfield of complex numbers). It is naturally a space of compact type, whose dual is a nuclear Fr\'echet space.
It contains ${\rm Sp}_J(L)$ as a closed subspace (it is closed because  it is the space of vectors killed by the Lie algebra of $G$). The dual of ${\rm Sp}_J^{\rm an}(L)$ 
surjects onto the dual of ${\rm Sp}_J(L)$ and contains the dual of ${\rm Sp}^{\rm cont}_J(L)$ as a dense subspace. The big difference is that 
${\rm Sp}_J^{\rm an}(L)$ is  topologically reducible as a $G$-module. Its Jordan-H\"older constituents are described in \cite{SO}.
\end{remark}
    
\subsection{Results of Schneider-Stuhler}
 
We recall the cohomological interpretation of the representations ${\rm Sp}_r(\mathbf{Z})$, following \cite{SS}. 
Recall that $\mathcal{H}$ is the space of $K$-rational hyperplanes
in $K^{d+1}$. For $r\in \{1,2,...,d\}$ we define simplicial profinite sets $Y_{\cdot}^{(r)}$, $\mathcal{T}^{(r)}_{\cdot}$
as follows: 
         
$\bullet$ $Y_s^{(r)}$ is the closed subset of $\mathcal{H}^{s+1}$ consisting of tuples $(H_0,...,H_s)\in \mathcal{H}^{s+1}$ with $$\dim_K(\sum_{i=0}^s K\ell_{H_i})\leq r,$$
where $\ell_{H_i}\in (K^{d+1})^*$ is any equation of $H_i$. 
              
$\bullet$  $\mathcal{T}_{s}^{(r)}$ is the set of flags 
$W_0\subset...\subset W_s$ in $(K^{d+1})^*$ for which $\dim_K W_i\in \{1,...,r\}$ for all 
$i$.  This set has a natural profinite topology.
        
In both cases the face/degeneracy maps are the obvious ones, i.e. 
omit/double one hyperplane in a tuple, resp. a vector subspace in a flag.  With the topology forgotten, $\mathcal{T}^{(d)}_{\cdot}$  is the Tits\footnote{For instance, for $d=1$ this is the set of ends of the tree.} (not Bruhat-Tits!) building of $G$.

           The following result is due to Schneider and Stuhler: 
               
        \begin{proposition} \label{SSsimplicial} For all $r\in \{1,2,...,d\}$ we have natural isomorphisms (where $\wt{H}$ denotes reduced cohomology)
        $$\wt{H}^{r-1}(|\mathcal{T}^{(r)}_{\cdot}|, \mathbf{Z})\simeq \wt{H}^{r-1}(|Y^{(r)}_{\cdot}|, \mathbf{Z})\simeq {\rm Sp}_r(\mathbf{Z}).$$
                \end{proposition}
         
         \begin{proof} The 
         isomorphism $\wt{H}^{r-1}(|\mathcal{T}^{(r)}_{\cdot}|, \mathbf{Z})\simeq \wt{H}^{r-1}(|Y^{(r)}_{\cdot}|, \mathbf{Z})$ is proved in \cite[Ch. 3, Prop 5]{SS}. To identify these objects with ${\rm Sp}_r(\mathbf{Z})$, assuming for simplicity $r>1$ from now on, consider the clopen subset $\mathcal{NT}_s^{(r)}\subset \mathcal{T}_{s}^{(r)}$ of $\mathcal{T}_{s}^{(r)}$ consisting of flags $W_0\subset...\subset W_s$ for which all inclusions are strict. Using the obvious isomorphism $\wt{H}^{r-1}(|\mathcal{NT}^{(r)}_{\cdot}|, \mathbf{Z})\simeq \wt{H}^{r-1}(|\mathcal{T}^{(r)}_{\cdot}|, \mathbf{Z})$
the result follows from 
         the exact sequence\footnote{Recall that if $S_{\cdot}$ is any simplicial profinite set,
         then $H^*(|S_{\cdot}|, \mathbf{Z})=H^*({\rm LC}(S_{\cdot}, \mathbf{Z}))$, where $|S_{\cdot}|$ is the geometric realisation of
       $S_{\cdot}$ and ${\rm LC}(S_{\cdot}, \mathbf{Z})$ is the complex 
       $({\rm LC}(S_s, \mathbf{Z}))_{s}$, the differentials being given by the alternating sum
       of the maps induced by face maps in $S$.}
         
                  $${\rm LC}( \mathcal{NT}^{(r)}_{r-2}, \mathbf{Z})\to {\rm LC}(\mathcal{NT}^{(r)}_{r-1},\Z)\to H^{r-1}(|\mathcal{NT}^{(r)}_{\cdot}|, \mathbf{Z})\to 0$$         
         and the 
         identifications
         $$\mathcal{NT}^{(r)}_{r-1}\simeq X_{\{1,2,...,d-r\}}, \quad \mathcal{NT}^{(r)}_{r-2}\simeq \coprod_{i=d-r+1}^d X_{\{1,...,d-r, i\}}$$
\end{proof}
             
  \begin{remark} \label{USEFUL} 
 For all $r\in \{1,2,...,d\}$ and all $q$ there are natural isomorphisms 
       $$H^q(|\mathcal{NT}^{(r)}_{\cdot}|, \mathbf{Z})\simeq H^q(|\mathcal{T}^{(r)}_{\cdot}|, \mathbf{Z})\simeq H^q(|Y^{(r)}_{\cdot}|, \mathbf{Z})$$
       and these spaces are nonzero only for $q=0, r-1$, with $H^0$ being given by $\mathbf{Z}$ for $r>1$ and by 
       ${\rm LC}(\mathbf{P}((K^{d+1})^*), \mathbf{Z})$ for $r=1$. See \cite[Ch. 3, Prop. 6]{SS} for the details.

   \end{remark}
  
          The following theorem is one of the main results of \cite{SS}. See also \cite{Orl} for a different argument (at least for a) and the compactly supported analogue of b)). 
     
\begin{theorem}{\rm (Schneider-Stuhler)} \label{SSmain}
Let  $r\geq 0$. 
     
{\rm   a)}  For a  prime $\ell\ne p$, there are natural isomorphisms of $G\times {\sg}_K$-modules 
$$H^r_{\eet}({\mathbb H}^d_C, \Q_{\ell}(r))\simeq {\rm Sp}_r(\Z_\ell)^*\otimes \Q_{\ell},\quad H^r_{\proeet}({\mathbb H}^d_C, \Q_{\ell}(r))\simeq {\rm Sp}_r(\Q_{\ell})^*.$$
          
{\rm    b)}  There is a  natural isomorphism of $G$-modules 
     $$H^r_{\rm dR}({\mathbb H}^d_K)\simeq {\rm Sp}_r(K)^*.  $$
\end{theorem}

\begin{proof} Let $H^*$ be any of the cohomologies occuring in the theorem. It has the properties required by Schneider-Stuhler \cite[Ch. 2]{SS}. The crucial among them  is the homotopy invariance property: if $D$ is the $1$-dimensional open unit disk then, for any smooth affinoid 
$X$,  the projection $X\times D\to X$ induces a  natural isomorphism $H^*(X)\stackrel{\sim}{\to}H^*(X\times D)$.
For de Rham cohomology this is very simple (see the discussion preceding Prop. 3 in \cite[Ch. 2]{SS});
for $\ell$-adic \'etale and pro-\'etale cohomologies this follows from the ``homotopy property''
 of $\ell$-adic \'etale cohomology with respect to  a closed disk \cite[proof of Theorem  6.0.2]{SS}, and the fact that $\ell$-adic \'etale and pro-\'etale cohomologies are the same on affinoids.
  
   We recall very briefly the key arguments, without going into the rather involved combinatorics. If $H\in \mathcal{H}$ and $n\geq 1$, let 
   $U_n(H)$ be the open polydisk in the affine space ${\mathbb P}_K^d\setminus H$ given by\footnote{We use unimodular representatives for points of projective space and for linear forms giving equations of $H$.}
   $|\ell_H(z)|>|\pi|^n$. The open subsets $U_n=\cap_{H\in \mathcal{H}} U_n(H)$ form a Stein covering of ${\mathbb H}^d_K$ and $U_n=\cap_{H\in \mathcal{H}_n} U_n(H)$, for a 
    finite subset $\mathcal{H}_n$ of $\mathcal{H}$, in bijection with ${\mathbb P}^d((\so_K^{d+1}/\pi^n)^*)$. 
   Writing $H^*(X,U)$ for the ``cohomology with support in $X\setminus U$'' (more precisely, the 
  derived functors of the functor ``sections vanishing on $U$''), a formal argument (see the discussion following \cite[Ch.3, Cor. 5]{SS}) gives a spectral sequence 
   $$E_1^{-j, i}(n)=\bigoplus_{H_{0},...,H_{j}\in \mathcal{H}_n} H^i({\mathbb P}_K^d, U_n(H_{0})\cup...\cup U_n(H_{j}))\Rightarrow H^{i-j}({\mathbb P}^d_K, U_n).$$
   
   Now,  $U_n(H_{0})\cup...\cup U_n(H_{j})$ is a locally trivial fibration over a projective space, whose fibers are open polydisks  \cite[Ch.1, prop.6]{SS}. Using the homotopy invariance of cohomology, one computes $H^i({\mathbb P}_K^d, U_n(H_{0})\cup...\cup U_n(H_{j}))$, in particular this is always equal to $A=H^0({\rm Sp}(K))$ or $0$ (with a simple combinatorial recipe allowing to distinguish the two cases). The spectral sequence simplifies greatly and\footnote{This is allowable as all modules involved are finite over the Artinian ring $A$.} letting $n\to \infty$ gives (using also Proposition \ref{SSsimplicial} and Remark \ref{USEFUL}) a spectral sequence 
   $$E_2^{-j,i}\Rightarrow H^{i-j}({\mathbb P}^d_K, {\mathbb H}^d_K),$$
   where $$E_2^{-j,i}={\rm Hom}_{\mathbf{Z}} (H^j(|\mathcal{T}_{\cdot}^{(\frac{i}{2})}|, \mathbf{Z}), A)$$
   if $i\in [2, 2d]$ is even and $j\in \{0, \frac{i}{2}-1\}$, and $0$ otherwise. The analysis of this spectral sequence combined with Proposition 
   \ref{SSsimplicial} yields the cohomology groups $H^i({\mathbb P}^d_K, {\mathbb H}^d_K)$. The result follows from the exact sequence 
   $$...\to H^i({\mathbb P}^d_K)\to H^i({\mathbb H}^d_K)\to H^{i+1}({\mathbb P}^d_K, {\mathbb H}^d_K)\to H^{i+1}({\mathbb P}^d_K)\to....$$
   
   \end{proof}  
  Combining Theorem \ref{SSmain} and Corollary \ref{Uniquelattice1} yields:

  \begin{corollary}\label{Gbounded}
   The space of $G$-bounded vectors in $H^r_{\rm dR}({\mathbb H}^d_K)$
   is isomorphic to ${\rm Sp}^{\rm cont}_r(K)^*$. 
     \end{corollary}
   \subsection{Generalization of Schneider-Stuhler} We will extend the results of Schneider-Stuhler to Hyodo-Kato cohomology. To do that we will use the description of the isomorphisms
   in Theorem \ref{SSmain} via symbols. 
      \subsubsection{Results of Iovita-Spiess} 
    All the isomorphisms in Theorem \ref{SSmain} are rather abstract, but following Iovita-Spiess \cite{IS} one can make them quite 
    explicit as follows. 
   Let ${\rm LC}^{c}(\mathcal{H}^{r+1}, \mathbf{Z})$ be the space of locally constant functions 
       $f: \mathcal{H}^{r+1}\to \mathbf{Z}$ such that, for all $H_0,...,H_{r+1}\in \mathcal{H}$,
       $$f(H_1,...,H_{r+1})-f(H_0, H_2,..., H_{r+1})+...+(-1)^{r+1}f(H_0,...,H_{r})=0$$
       and moreover, if 
       $\ell_{H_i}$ are linearly dependent, then     
       $f(H_0,...,H_r)=0$ (i.e., $f$ vanishes on $Y_r^{(r)}$). Define analogously $\mathcal{C}^{c}(\mathcal{H}^{r+1}, \mathbf{Z})$.
      It is not difficult to see that we have a natural isomorphism (see the proof of Proposition
      \ref{SSsimplicial} for the notation used below)
       $$\widetilde{H}^{r-1}(|\mathcal{NT}^{(r)}_{\cdot}|, \mathbf{Z})\simeq {\rm LC}^{c}(\mathcal{H}^{r+1}, \mathbf{Z})$$ and, in particular, (using Proposition \ref{SSsimplicial})
       a natural  isomorphism $${\rm Sp}_r(\mathbf{Z})\simeq {\rm LC}^{c}(\mathcal{H}^{r+1}, \mathbf{Z}).$$
       
         If $S$ is a profinite set and $A$ an abelian group, let $D(S,A)={\rm Hom}({\rm LC}(S,\mathbf{Z}), A)$ be the space of $A$-valued locally constant distributions on 
         $S$. If $L$ is a discrete valuation nonarchimedean field  let $M(S, L)$ be  the    space of  $L$-valued measures, i.e., bounded $L$-valued distributions. It has a natural topology that is finer than the subspace topology induced from $D(S,L)$ \cite[Ch. 4]{IS}.
         
          The inclusion ${\rm LC}^{c}(\mathcal{H}^{r+1}, \mathbf{Z})\subset {\rm LC}(\mathcal{H}^{r+1}, \mathbf{Z})$ gives rise to a continuous strict surjection 
           $$D(\mathcal{H}^{r+1}, A)\to {\rm Hom}({\rm Sp}_r(\mathbf{Z}), A).$$
           Define the space $D(\mathcal{H}^{r+1}, A)_{\rm deg}$ of degenerate distributions as the kernel of this map.
        Combining this with the previous theorem we obtain  surjections: 
        \begin{align*}
        & D(\mathcal{H}^{r+1}, K)\to H^s_{\rm dR}({\mathbb H}^d_K), \quad M(\mathcal{H}^{r+1}, \Q_{\ell})\to H^r_{\eet}({\mathbb H}^d_C, \Q_{\ell}(r)),\\
         & D(\mathcal{H}^{r+1}, \Q_{\ell})\to H^r_{\proeet}({\mathbb H}^d_C, \Q_{\ell}(s)).
         \end{align*}

        These surjections can be made explicit as follows. For each $(H_0,...,H_r)\in \mathcal{H}^{r+1}$, the invertible functions (on ${\mathbb H}^d_K$)
         $\frac{\ell_{H_1}}{\ell_{H_0}},..., \frac{\ell_{H_r}}{\ell_{H_0}}$ give rise (either by taking ${\rm dlog}$ and wedge-product or by taking the corresponding symbols in \'etale cohomology and then cup-product) to a symbol 
         $[H_0,...,H_r]$ living in $H^r_{\rm dR}({\mathbb H}^d_K)$, resp. in $H^r_{\eet}({\mathbb H}^d_C, \Q_{\ell}(r))$, resp. $H^r_{\proeet}({\mathbb H}^d_C, \Q_{\ell}(r))$. For example, for de Rham cohomology 
         $[H_0,...,H_r]$ is the class of the closed $r$-form 
        $${\rm dlog}\frac{\ell_{H_1}}{\ell_{H_0}}\wedge...\wedge {\rm dlog}\frac{\ell_{H_r}}{\ell_{H_0}}$$
in $H^r_{\rm dR}({\mathbb H}^d_K)$.

    One shows that the following regulator map is well-defined
$$
r_{\dr}: D(\mathcal{H}^{r+1}, K)\to H^r_{\dr}({\mathbb H}^d_K), \quad \mu\mapsto \int_{\mathcal{H}^{r+1}} [H_0,...,H_r] \mu(H_0,...,H_r).$$ The problem here is that the map $(H_0,...,H_r)\mapsto [H_0,...,H_r]$ is not locally constant on ${\mathbb H}^d_K$; however it is so on $U_n$ (see the proof of Theorem \ref{SSmain} for the notation), for all $n$, which makes it possibe to give a meaning to the integral.
The same integral  works for $\ell$-adic \'etale and pro-\'etale cohomologies 
yielding 
the regulator maps
\begin{align*}
 r_{\eet}: M(\mathcal{H}^{r+1}, \Q_{\ell})\to H^r_{\eet}({\mathbb H}^d_C,\Q_{\ell}(r)),\quad  r_{\proeet}: D(\mathcal{H}^{r+1}, \Q_{\ell})\to H^r_{\proeet}({\mathbb H}^d_C,\Q_{\ell}(r)).
 \end{align*}
This can be easily seen in the case of \'etale cohomology. For the pro-\'etale cohomology, the key point is that we can write $$H^r_{\proeet}({\mathbb H}^d_C, \Q_{\ell}(r))=\varprojlim_{n} H^r_{\proeet}(U_{n,C}, \Q_{\ell}(r)),$$ where
$H^r_{\proeet}(U_{n,C}, \Q_{\ell}(r))$ is finite dimensional 
and the map $$\mathcal{H}^{r+1}\to H^r_{\proeet}({\mathbb H}^d_C, \Q_{\ell}(r))
\to H^r_{\proeet}(U_{n,C}, \Q_{\ell}(r)),\quad 
(H_0,...,H_r)\mapsto [H_0,...,H_r],$$ 
is locally constant for all $n$, by arguing as in \cite[Lemma 4.4]{IS}.   All these regulators are continuous.

    One can show that the above  maps induce the isomorphisms in Theorem \ref{SSmain} by imitating the arguments in \cite{IS}.
    \begin{theorem}{\rm (Iovita-Spiess, \cite[Theorem 4.5]{IS})}
\label{IS}
The following diagram of Fr\'echet $G$-spaces commutes
$$
\xymatrix{
0\ar[r] & D(\mathcal{H}^{r+1}, K)_{\deg}\ar[r] & D(\mathcal{H}^{r+1}, K)\ar@{->>}[rd]_{\can}\ar[r]^{r_{\dr}} & H^r_{\dr}({\mathbb H}^d_K)\ar[r] & 0\\
 & & & {\rm Sp}_r(K)^*\ar@{-}[u]^{\wr}
}
$$
and the sequence is strictly exact. 
Similarly for $\ell$-adic \'etale and pro-\'etale cohomologies.
\end{theorem}
\subsubsection{Generalization of the results of Iovita-Spiess}
Set $X:=\wt{{\mathbb H}}^d_K$ and $Y:=\wt{{\mathbb  H}}^d_{K,0}$. The above results of Iovita-Spiess can be generalized to Hyodo-Kato cohomology.
\begin{lemma}
\label{genIS}Let $r\geq 0$. There are natural isomorphisms of Fr\'echet spaces
$$
 H^r_{\hk}(Y)\simeq {\rm Sp}_r(F)^*,\quad H^r_{\hk}(Y)^{\phi=p^r}\simeq {\rm Sp}_r(\Q_p)^*
$$
that are compatible with the isomorphism $H^r_{\dr}(X_K)\simeq {\rm Sp}_r(K)^*$ 
from Theorem \ref{SSmain} and the natural maps $ {\rm Sp}_r(\Q_p)^*\to {\rm Sp}_r(F)^*\to  {\rm Sp}_r(K)^*$.
\end{lemma}
\begin{proof}       We start with $ H^r_{\hk}(Y)$. Consider the following diagram  
   $$
     \xymatrix{ 
  & D(\sh^{r+1},K)\ar[rd]_{r_{\dr}} \ar[rrd]^{\can}\\
  D(\sh^{r+1},F)\ar[ru]^{\can}\ar[r]^{r_{\hk}} \ar@{->>}[dr]_{\can} & H^r_{\hk}(Y)\ar@{^(->}[r]^-{\iota_{\hk}} & H^r_{\dr}(X_K) \ar@{-}[r]^{\sim} & {\rm Sp}_r(K)^*\\
  &  {\rm Sp}_r(F)^*\ar@{^(->}[urr]_{\can}\ar@{-->}[u]^f
  }
  $$
  Here the regulator map $r_{\hk}$ is defined in an analogous way to the map $r_{\dr}$ but by using the overconvergent Hyodo-Kato Chern classes $c^{\hk}_1$ defined in the Appendix.  It is continuous. 
The outer diagram clearly commutes.  The small triangle commutes by Theorem \ref{IS}. The square  commutes by Lemma \ref{compatibility}. Chasing the diagram we construct the broken arrow, a continuous map 
$f: {\rm Sp}_r(F)^*\to H^r_{\hk}(Y)$ that makes the left bottom triangle commute; it is easy to check that it makes the right bottom triangle commute as well. This implies that the map $f$ is injective. Since $H^r_{\hk}(Y)$ is topologically irreducible as a $G$-module  (use the Hyodo-Kato isomorphism), it is also surjective (use the fact that $ {\rm Sp}_r(F)^*$ is closed in ${\rm Sp}_r(K)^*$).

  The argument for $H^r_{\hk}(Y)^{\phi=p^r}$  is similar. But first we need to show that the natural map 
  \begin{equation}
  \label{paris12}
  H^r_{\hk}(Y)^{\phi=p^r}\otimes_{\Q_p}F\to H^r_{\hk}(Y)
  \end{equation}
   is an injection. We compute
 \begin{align*}
 H^r_{\hk}(Y)^{\phi=p^r}\otimes_{\Q_p}F & \simeq (\varprojlim_s H^r_{\hk}(Y_s))^{\phi=p^r}\otimes_{\Q_p}F\simeq 
 (\varprojlim_s H^r_{\hk}(Y_s)^{\phi=p^r})\otimes_{\Q_p}F\\
  & \simeq \varprojlim_s (H^r_{\hk}(Y_s)^{\phi=p^r}\otimes_{\Q_p}F)
 \hookrightarrow \varprojlim_s H^r_{\hk}(Y_s)\simeq H^r_{\hk}(Y),
 \end{align*}   
  as wanted.       For the injection above we have used the fact  that all $H^r_{\hk}(Y_s)$ are finite dimensional over $F$.    
                   
 We look now  at the commutative diagram
$$
\xymatrix{D(\sh^{r+1},F) \ar[rr]^{\can}    \ar[rd]^{r_{\hk}}&  &  {\rm Sp}_r(F)^*\ar[dl]^{f}\\
   & H^r_{\hk}(Y)  \\
   & H^r_{\hk}(Y)^{\phi=p^r}\ar@{^(->}[u]^{\can}\\
   D(\sh^{r+1},\Q_p)\ar@{->>}[rr]^{\can}\ar[uuu]^{\can}\ar[ru]^{r_{\hk}} & &  {\rm Sp}_r(\Q_p)^*\ar@{^(->}[uuu]^{\can}\ar@{-->}[lu]^{f^{\prime}}
}
$$
The key point is that, as shown in Section \ref{diffsymbols},  the map $r_{\hk}$ restricted to $D(\sh^{r+1},\Q_p)$ factors through $H^r_{\hk}(Y)^{\phi=p^r}$. 
Arguing as above we construct the continuous map $f^{\prime}$. It is clearly injective. It is surjective by (\ref{paris12}). 
\end{proof}

 \subsection{Pro-\'etale cohomology} \label{kolokwak}
 We are now ready to compute the $p$-adic pro-\'etale cohomology of ${\mathbb H}^d_C$.
 Let $r\geq 0$. Since  the linearized Frobenius on $H^r_{\hk}(Y)$ is equal to the multiplication by $q^r$, where $q= |k|$.
  \cite[Cor. 6.6]{GK2} and  $N\phi=p\phi N$ \cite[Prop. 5.5]{GK2}, the monodromy operator is trivial on $H^r_{\hk}(Y)$.
Hence the first isomorphism below is Galois equivariant.
 \begin{align*}
 (H^r_{\hk}(Y)\wh{\otimes}_F\B^+_{\st})^{N=0,\phi=p^r} & \simeq (H^r_{\hk}(Y)\wh{\otimes}_{F}\B^{+}_{\crr})^{\phi=p^r}        
 \simeq (H^r_{\hk}(Y)^{\phi=p^r}\wh{\otimes}_{\Q_p}\B^{+}_{\crr})^{\phi=p^r}\\
 & \simeq
  H^r_{\hk}(Y)^{\phi=p^r}\wh{\otimes}_{\Q_p}\B^{+,\phi=1}_{\crr}\simeq
  {\rm Sp_r(\Q_p})^*\wh{\otimes}_{\Q_p}\B_{\crr}^{+,\phi=1}= {\rm Sp_r(\Q_p})^*.
  \end{align*}
  The second isomorphism follows from the proof of Lemma \ref{genIS}, the fourth one -- from this lemma itself, and  the third one is clear.
  Using the above isomorphisms and Lemma \ref{genIS},  the map
  $\iota_{\hk}\otimes\theta: (H^r_{\hk}(Y)\wh{\otimes}_F\B^+_{\st})^{N=0,\phi=p^r} \to H^r_{\dr}(X_K)\wh{\otimes}_K C$ can be identified with the natural map 
  ${\rm Sp}_r(\Q_p)^*\to {\rm Sp}_r(K)^*\wh{\otimes}_KC$. 
  
  Similarly, we compute  that
  \begin{align*}
 (H^r_{\hk}(Y)\wh{\otimes}_F\B^+_{\st})^{N=0,\phi=p^{r-1}} & \simeq H^r_{\hk}(Y)^{\phi=p^r}\wh{\otimes}_{\Q_p}\B^{+,\phi=p^{-1}}_{\crr}= 0.
\end{align*}
Combined with Theorem \ref{fdd}, these  yield  the following theorem.
\begin{theorem}\label{PROET}
Let $r\geq 0$.
    There is a natural map of strictly exact sequences of  $G\times \sg_K$-Fr\'echet spaces  (over $\Q_p$) 
$$
 \xymatrix{
 0\ar[r] & \Omega^{r-1}({\mathbb H}^d_C)/\ker d\ar[r]\ar@{=}[d] & H^r_{\proeet}({\mathbb H}^d_C,\qp(r))\ar@{^(->}[d]^{\tilde{\beta}}\ar[r] &
  {\rm Sp}_r (\Q_p)^*\ar[r]\ar@{^(->}[d]^{\can} & 0\\
0\ar[r] & \Omega^{r-1}({\mathbb H}^d_C )/\ker d \ar[r]^-{d}  & \Omega^r ({\mathbb H}^d_C)^{d=0}\ar[r] & {\rm Sp}_r(K)^*\wh{\otimes}_KC\ar[r] & 0
 }
 $$
 where the vertical maps are  closed immersions. 
 \end{theorem}
\section{\'Etale cohomology of Drinfeld half-spaces} \label{etale}
The purpose of this chapter is to compute the $p$-adic \'etale cohomology of the Drinfeld half-space (Theorem~\ref{comp19}). 
Using a  comparison theorem (see Proposition \ref{fini11} below) 
this reduces to the computation of the Fontaine-Messing syntomic cohomology. 
The latter then is transformed into a syntomic cohomology built from the crystalline Hyodo-Kato cohomology and the integral de Rham cohomology (see Section~\ref{ZURICH}). 
These differential cohomologies can be computed explicitly (Corollary~\ref{amtrak} and~Theorem~\ref{Steinberg})
due to the fact that the standard formal model of the Drinfeld half-space is pro-ordinary and the sheaves of integral differentials are acyclic (a result of Grosse-Kl\"onne).

 Throughout this chapter we work in the category of pro-discrete modules (see Section \ref{pro-discrete} for a quick review and Remark \ref{nassau1} for how the topology is defined on the various cohomology complexes).
\subsection{Period isomorphism}
\begin{proposition}
\label{fini11}
Let $X$ be a  semistable formal scheme over $\so_K$.  Let $r\geq 0$, $\overline{X}:=X_{\so_C}$.
    There is a natural  Fontaine-Messing period map
    \begin{equation*}
    \label{topology11}
    \alpha_{\rm FM}: \R\Gamma_{\synt}(\overline{X},\Z_p(r)){\otimes}^L\Q_p\to \R\Gamma_{\eet}(X_{C},\Z_p(r)){\otimes}^L\Q_p
    \end{equation*}
   that  is a strict quasi-isomorphism (in $\sd(C_{\Q_p})$) after truncation $\tau_{\leq r}$.
\end{proposition}
\begin{proof}
Let $Y$ be a semistable finite type formal scheme over $\so_K$.  Fontaine-Messing in \cite[III.5]{FM} have defined an integral period map 
$$
 \wt{\alpha}_{\rm FM}:\R\Gamma_{\synt}(\overline{Y},\Z_p(r))\to \R\Gamma_{\eet}(Y_{C},\Z_p(r)^{\prime}),
$$
where $\Z_p(r)^{\prime}:=(p^aa!)^{-1}\Z_p(r)$, for $r=(p-1)a+b, a,b\in\Z, 0\leq b\leq p-1$. 
The map $ \tau_{\leq r}\wt{\alpha}_{\rm FM}$  is a $p^N$-quasi-isomorphism for a universal constant $N=N(r)$. This means that its  kernel and cokernel  on cohomology groups in degrees $0\leq i\leq r$ are annihilated by $p^N$. It follows that the cone of $\wt{\alpha}_{\rm FM}$ has cohomology annihilated by $p^N$, $N=N(r)$, as well. 

We define $$
\wt{\alpha}_{\rm FM}:\R\Gamma_{\synt}(\overline{X},\Z_p(r))\to \R\Gamma_{\eet}(X_{C},\Z_p(r)^{\prime})
$$ by cohomological descent from the above $\wt{\alpha}_{\rm FM}$. The above local arguments imply that 
 $(\tau_{\leq r} \wt{\alpha}_{\rm FM}){\otimes}^L{\Q_p}$ is a quasi-isomorphism.  We set ${\alpha}_{\rm FM}:=p^{-r}(\wt{\alpha}_{\rm FM}{\otimes}^L{\Q_p})$.  This twist by $p^{-r}$
is necessary to make the period morphism compatible with Chern classes.

 It remains to show that $(\tau_{\leq r} \wt{\alpha}_{\rm FM}){\otimes}^L{\Q_p}\simeq\tau_{\leq r} (\wt{\alpha}_{\rm FM}{\otimes}^L\Q_p)$ and that this quasi-isomorphism  is strict. 
 \begin{remark}\label{nassau1} 
Before  doing that, let us recall how topology is defined on the domain and target of $\wt{\alpha}_{\rm FM}$. Locally, for a quasi-compact \'etale open $U\to X$, we get complexes of (topologically free) $\Z_p$-modules with $p$-adic topology. For a quasi-compact \'etale hypercovering $\{U_i\}_{i\in I}$, of $X$, we take the total complex of the \v{C}ech complex of such complexes. Hence in every degree we have a product of $\Z_p$-modules with $p$-adic topology. The functor $(-){\otimes}\Q_p$ from Section \ref{pro-discrete}, by Proposition \ref{acyclic-integral},  
associates to these complexes of pro-discrete $\Z_p$-modules complexes of locally convex $\Q_p$-vector spaces by tensoring them degree-wise with $\Q_p$ and taking the induced topology. 
These new complexes represent the domain and target of $\wt{\alpha}_{\rm FM}{\otimes}^L\Q_p$. We note that we have a strict quasi-isomorphism
$\tau_{\leq r}(\wt{\alpha}_{\rm FM}{\otimes}^L\Q_p)\simeq
(\tau_{\leq r}\wt{\alpha}_{\rm FM}){\otimes}^L\Q_p$ (again use Proposition \ref{acyclic-integral}).
 \end{remark}
 
 Now, note that the map $\wt{\alpha}_{\rm FM}$, being a $p^N$-isomorphism on cohomology, has a $p^N$-inverse in $D(\Ind(PD_{\Z_p}))$, i.e., there exists a map $\wt{\beta}:\tau_{\leq r}\R\Gamma_{\eet}(X_{C},\Z_p(r)^{\prime})\to \tau_{\leq r} \R\Gamma_{\synt}(\overline{X},\Z_p(r))$ such that $\wt{\alpha}\wt{\beta}=p^N$ and $\wt{\beta}\wt{\alpha}=p^N$ (not the same $N=N(r)$, of course). It follows that $(\tau_{\leq r} \wt{\alpha}_{\rm FM}){\otimes}^L{\Q_p}$ has an inverse in $D(C_{\Q_p})$, hence it is strict. 
\end{proof}
\subsection{Cohomology of differentials}
We gather in this section computations of various bounded differential cohomologies of the Drinfeld half-space.
Let 
$$\wt{X}:=\wt{{\mathbb H}}^d_K,\quad {X}:=(\wt{{\mathbb H}}^d_K)^{\wedge}$$
 be the standard weak formal model, resp. formal model, of the Drinfeld half-space ${\mathbb H}^d_K$. It is equipped with an action of $G={\rm GL}_{d+1}(K)$ compatible with the natural action on the generic fiber.
Let 
$$Y:=X_0,\quad  \overline{Y}:=Y_{\overline{k}}.$$  
Let $F^0$ be the set of irreducible components of the special fiber $Y$. They are isomorphic smooth projective schemes over $k$ that we see as log-schemes with the log-structure induced from $Y$. Let $T$ be the central irreducible component of $Y$, i.e., the irreducible component with  stabilizer $K^*{\rm GL}_{d+1}(\so_K)$.
It is obtained from the projective space ${\mathbb P}^d_k$ by first blowing up all $k$-rational points, then blowing up the strict transforms of $k$-rational lines, etc. For $0\leq j\leq d-1,$ let $\sv^j_0$ be the set of all $k$-rational linear subvarieties $Z$ of ${\mathbb P}^d_k$ with $\dim (Z)=j$ and let $\sv_0:=\bigcup_{j=0}^{d-1}\sv^j_0$. The set $\sv$ of all strict transforms in $T$ of elements of $\sv_0$ is a set of divisors of $T$; together with the canonical log-structure of the log-point $k^0$, it induces the log-structure on $T$. 
 
   Let $\wt{\theta}_0,\ldots,\wt{\theta}_d$ be the standard projective coordinate functions on ${\mathbb P}^d_k$ and on $T$. For $i,j\in\{0,\ldots,d\}$ and $g\in G$ we call $g\dlog (\wt{\theta}_i/\wt{\theta}_j) $ a {\em standard logarithmic differential 1-form} on $T$; exterior products of such forms we call {\em standard logarithmic differential forms} on $T$. 
 
 \subsubsection{Cohomology of differentials on irreducible components}As proved by Grosse-Kl\"onne the sheaves of differentials on $T$ are  acyclic and the standard logarithmic differential forms generate the $k$-vector space of global differentials. 
\begin{proposition}{\rm (\cite[Theorem 2.3, Theorem 2.8]{GKI},\cite[Prop. 1.1]{GKB})}
\label{GK5}
\begin{enumerate}
\item $H^i(T,\Omega^j)=0$, $i>0,j\geq 0$.
\item  The $k$-vector space $H^0(T,\Omega^j)$, $j\geq 0$, is generated by  standard logarithmic forms. In particular, it is killed by $d$.
\item $H^i_{\crr}(T/\so_F^0)$ is torsion free and 
$$
H^i_{\crr}(T/\so_F^0)\otimes_{\so_F}k=H^i_{\dr}(T)=H^0(T,\Omega^i_T).
$$
\end{enumerate}
\end{proposition}
We note here that, the underlying scheme of $T$ being smooth, the crystalline cohomology $H^i_{\crr}(T/\so_F^0)=H^i_{\crr}(T^{\prime}/\so_F)$, where $T^{\prime}$ is the underlying scheme of $T$  equipped with the log-structure given by the elements of $\sv$.

 For $0\leq j\leq d$, let ${\mathbb L}^j_T$ be the $k$-vector subspace of $\Omega^j_T(T^0)$, $T^0:=T\setminus \cup_{V\in {\mathcal V}}V$,  generated by all $j$-forms $\eta$ of the type
 $$
 \eta=y_1^{m_1}\cdots y_j^{m_j}\dlog y_1\wedge \cdots\wedge \dlog y_j
 $$
 with $m_i\in\Z$ and $y_i\in\so(T^0)^*$ such that $y_j=\wt{\theta}_j/\wt{\theta}_0$ for an isomorphism of $k$-varieties ${\mathbb P}^d_k\simeq{\rm Proj} (k[\wt{\theta}_0,\cdots,\wt{\theta}_d])$. By Theorem \ref{GK5}, $H^0(T,\Omega^j)$ is the $k$-vector subspace of ${\mathbb L}^j_T$ generated by all $j$-forms $\eta$ as above with $m_i=0$ for all $0\leq i\leq j$. 

 Let ${\mathbb L}^j_T$, resp.  ${\mathbb L}^{j,0}_T$, be the constant sheaf on $T$ with values ${\mathbb L}^j_T$, resp. $H^0(T,\Omega^j_T)$.
 For a non-empty subset $S$ of ${\mathcal V}$ such that $E=\cap_{V\in S}V$ is non-empty, define the subsheaf ${\mathbb L}^j_E$ of $\Omega^j_T\otimes\so_E$ as the image of the composite
 $$
 {\mathbb L}^j_T\to \Omega^j_T\to \Omega^j_T\otimes\so_E.
 $$
 \begin{proposition}{\rm (\cite[Theorem 1.2]{GKB})}
 \label{standard}
The canonical maps
 $$
 {\mathbb L}^{j,0}_T\hookrightarrow {\mathbb L}^j_T\hookrightarrow \Omega^j_T, \quad 
{\mathbb L}^j_E\hookrightarrow  \Omega^j_T\otimes\so_E$$
 induce  isomorphisms on Zariski cohomology groups.
 \end{proposition}
 
\subsubsection{Cohomology of differentials on $X$ and truncations of $Y$}
We quote an important result of Grosse-Kl\"onne proving acyclicity of the sheaves of differentials on $X$ and vanishing of the differential on their global sections.
\begin{proposition}{\rm (\cite[Theorem 4.5]{GKI},\cite[Prop. 4.5]{GKB})}\label{GK01} 
Let $j\geq 0$.
 \begin{enumerate}
 \item We have  topological isomorphisms\footnote{Here and below, cohomology $H^*$ without a subscript denotes Zariski cohomology. All the groups are profinite. This is because they are limits of cohomologies of the truncated log-schemes $Y_s$ described below that are ideally log-smooth and proper.}
\begin{align*}
& H^i(X,\Omega^j_X)=0\quad{\rm and}\quad  H^i(X,\Omega^j_X\otimes_{\so_K}k)=0,\quad i >0,\\
& H^0(X,\Omega^j_X)\otimes_{\so_K}k\simeq H^0(X,\Omega^j_X\otimes_{\so_K}k).
\end{align*}
\item
$d=0$ on 
$H^0(X,\Omega^j_X).$
\end{enumerate}
\end{proposition}
\begin{corollary}Let $j\geq 0$. We have a topological isomorphism $H^j_{\dr}(X)\stackrel{\sim}{\leftarrow} H^0(X,\Omega^j_X)$. In particular, these groups are torsion-free.
\end{corollary}

The above theorem  can be generalized to the idealized log-schemes $Y_s, s\in\N,$ defined in Section \ref{formal-models} in the following way.
\begin{proposition}
\label{general-princeton} Let $j\geq 0$, 
 $s\in\N$.
 \begin{enumerate}
 \item $H^i(Y_s,\Omega^j)=0$ for $i>0$.
\item 
$d=0$ on $H^0(Y_s,\Omega^j).
$
 \end{enumerate}
\end{proposition}
\begin{proof} For the first claim, the argument is analogous to the one of Grosse-Kl\"onne for $s=\infty$.  We will sketch it briefly. 
Take  $s\neq\infty$.
 Since $\Omega^j_{Y_s}$ is locally free over $\so_{Y_s}$, we have the Mayer-Vietoris exact sequence
$$
0\to \Omega^j_{Y_s}\to \bigoplus_{Z\in F^0_s}\Omega^j_{Y_s}\otimes\so_Z\to \bigoplus_{Z\in F^1_s}\Omega^j_{Y_s}\otimes\so_Z\to\cdots
$$
where $F^r_s$ is the set of non-empty intersections of $(r+1)$ pairwise distinct irreducible components of  $Y_s$ and is a finite set (which is also  the set of $r$-simplices of ${\rm BT}_s$). 
 By \cite[Cor. 1.6]{GKI},
$H^i(Z,\Omega^j_{Y_s}\otimes\so_Z)=0,$ $i>0,$ for every $Z\in F^r_s$. 
Hence
to show that $H^i(Y_s,\Omega^j_{Y_s})=0$, $i>0$, we need to prove that
$H^i({\rm BT}_s,\sff)=0$, for $i>0$, where $\sff$ is the coefficient system on ${\rm BT}_s$ defined by $\sff(Z)=H^0(Y,\Omega^j_{Y_s}\otimes\so_Z)$, for $Z\in F^r_s$.
      We will use for that an analog of Grosse-Kl\"onne's acyclicity condition.  For a lattice chain in ${\rm BT}_s$
  $$
  \varpi L_r\subsetneq L_1\subsetneq \cdots\subsetneq L_r
  $$
  we call the ordered $r$-tuple $([L_1],\ldots,[L_r])$,  a pointed $(r-1)$-simplex (with underlying $(r-1)$-simplex the unordered set $\{[L_1],\ldots,[L_r]\}$). Denote it by $\wh{\eta}$
  and consider the set $$
  N_{\wh{\eta}}=\{[L]\mid\varpi L_r\subsetneq L\subsetneq L_1\}.
  $$
  We note that $N_{\wh{\eta}}$ is a subset of vertices of ${\rm BT}_s$. 
  A subset $M_0$ of $N_{\wh{\eta}}$ is called {\em stable} if, for all $L,L^{\prime}\in M_0$, the intersection   $L\cap L^{\prime}$ also lies in $M_0$.

\begin{lemma}Let $\sff$ be a cohomological coefficient system on ${\rm BT}_s$.
Let $1\leq r \leq d$. Suppose that for any pointed $(r-1)$-simplex $\wh{\eta}\in {\rm BT}_s$ with underlying $(r-1)$-simplex $\eta$ and for any stable subset $M_0$ of $N_{\wh{\eta}}$  the following subquotient complex of the cochain complex $C({\rm BT}_s,\sff)$ with values in $\sff$ is exact
$$
\sff(\eta)\to \prod_{z\in M_0}\sff(\{z\}\cup\eta)\to\prod_{z,z^{\prime}\in M_0,\{z,z^{\prime}\}\in\F^1_s}\sff(\{z,z^{\prime}\}\cup \eta).
$$
Then the $r$-th cohomology group $H^s({\rm BT}_s,\sff)$ of $C({\rm BT}_s,\sff)$ vanishes. 
\end{lemma}
\begin{proof}
For ${\rm BT}$ this is the main theorem of \cite{GKA}. The argument used in its proof \cite[Theorem 1.2]{GKA} carries over to our case: when applied to a cocycle from ${\rm BT}_s$, the recursive procedure of producing a coboundary in the proof of Theorem 1.2 in  loc. cit. ``does not leave'' ${\rm BT}_s$.
\end{proof}
Hence it suffices to check that the above condition is satisfied for our $\sff$. But this was checked in \cite[Cor. 1.6]{GKI}.

 The second claim of the proposition follows from the diagram 
 $$
 \xymatrix{
  H^0(Y_s, \Omega^j_{Y_s})\ar@{^(->}[r] &  \bigoplus_{Z\in F^0_s}H^0(Z,\Omega^j_{Y_s}\otimes\so_Z)  
  }
  $$ 
   and Proposition \ref{GK5}.
\end{proof}
\subsubsection{Ordinary log-schemes}
A quick review of basic facts concerning   ordinary log-schemes.

  Let  $W_n\Omega_{Y}\kr$ denotes the de Rham-Witt complex of $Y/k^0$ \cite{Hy}.
Recall first \cite[Prop. II.2.1]{Ill} that if $T$ is a log-smooth and proper log-scheme over $k^0$, for a perfect field $k$ of positive characterstic $p$,  then $H^i_{\eet}(T,W_n\Omega^j)$ is of finite length and we have
  $
  \rg_{\eet}(T,W\Omega^j)  \stackrel{\sim}{\to} \holim_n\rg_{\eet}(T,W_n\Omega^j)$ for $W\Omega^j:=\varprojlim_n W_n\Omega^j$.   It follows that
    $H^i_{\eet}(T,W\Omega^j)  \stackrel{\sim}{\to} \varprojlim_nH^i_{\eet}(T,W_n\Omega^j).$  The module $M_{i,j}$ of $p$-torsion of this group is annihilated by a power of $p$ and $H^i_{\eet}(T,W\Omega^j)/M_{i,j}$ is a free $\so_F$-module of finite type \cite[Theorem II.2.13]{Ill}. However, $H^0_{\eet}(T,W\Omega^j)$ is itself a  free $\so_F$-module of finite type
     \cite[Cor. II.2.17 ]{Ill}.
   On the other hand, the complex $\rg_{\eet}(T,W\Omega\kr)$ is perfect and 
    $\rg_{\eet}(T,W\Omega\kr)\otimes^L_{\so_F}\so_{F,n}\simeq \rg_{\eet}(T,W_n\Omega\kr)$ \cite[Theorem II.2.7]{Ill}.


Let $V$ be  a fine   (idealized) log-scheme over $k^0$ that is of Cartier type. 
We have  the subsheaves  of boundaries and cocycles of $\Omega^j_V$ (thought of as sheaves on  $V_{\eet}$) 
$$
  B^j_V:=\im(d:\Omega^{j-1}_V\to \Omega^j_V),\quad  Z^j_V:=\ker(d:\Omega^j_V\to\Omega^{j+1}_V).
$$
{\em Assume now that $V$ is proper and log-smooth}. Recall  that  it is called {\em ordinary} if for all $i,j\geq 0$, $H^i_{\eet}(V,B^j)=0$ (see \cite{BK}, \cite{IR}). 

  We write
$W_n\Omega^r_{{-},\log}$ for  the de Rham-Witt sheaf of logarithmic forms.
\begin{proposition}{\rm (\cite[Theorem 4.1]{Lor})}
\label{rzeszow5}
  The following conditions are equivalent (we write $\overline{V}$ for $V_{\overline{k}}$).
  \begin{enumerate}
  \item $V/k^0$ is ordinary.
  \item For $i,j\geq 0$, the inclusion $\Omega^j_{\overline{V},\log} \subset \Omega^j_{\overline{V}}$ induces a canonical isomorphism of  $\overline{k}$-vector spaces
  $$
  H^i_{\eet}(\overline{V},
  \Omega^j_{\log})\otimes_{{\mathbf F}_p}\overline{k}\stackrel{\sim}{\to} H^i_{\eet}(\overline{V},\Omega^j).
  $$
  \item For $i,j,n\geq 0$, the canonical maps 
  \begin{align*}
  H^i_{\eet}(\overline{V},W_n\Omega^j_{\log})\otimes_{\Z/p^n}W_n(\overline{k})   & \to H^i_{\eet}(\overline{V},W_n\Omega^j),\\
    H^i_{\eet}(\overline{V},W\Omega^j_{\log})\otimes_{\Z_p}W(\overline{k})  & \to H^i_{\eet}(\overline{V},W\Omega^j),
\end{align*}
where $W\Omega^r_{\log}:=\varprojlim_nW_n\Omega^r_{\log}$,
  are  isomorphisms.
  \item For $i,j\geq 0$, the de Rham-Witt Frobenius 
  $$
  {\rm F}: H^i_{\eet}(V,W\Omega^j)\to H^i_{\eet}(V,W\Omega^j)
  $$
  is an isomorphism.
  \end{enumerate}
  \end{proposition}
  \begin{example}
  The above result  implies that, by the Projective Space Theorem, projective spaces are ordinary, and, more generally, so are projectivizations of vector bundles  \cite[Prop. 1.4]{Ill1}. This implies, by the blow-up diagram,  the following:
  \begin{proposition}{\rm (\cite[Prop. 1.6]{Ill1})}
  \label{Illusie1}
Let $X$ be a proper smooth scheme over $k$. Let $Y\subset X$ be a smooth closed subscheme, $\wt{X}$ the blow-up of $Y$ in $X$. Then $X$ and $Y$ are ordinary if and only if $\wt{X}$ is ordinary.
\end{proposition}
And this, in turn, by the weight spectral sequence, implies the following: 
\begin{proposition}{\rm (\cite[Prop. 1.10]{Ill1})}
\label{Illusie2}
Assume that $k=\overline{k}$. Let $Y$ be a semistable scheme over $k$.
Assume that it is a union of irreducible components  $Y_i$,  $1 \leq  i  \leq r$ such that 
for all  $I \subset \{1,\cdots,r\}$, the intersection $Y_I$ is smooth and ordinary.
Then $Y$, as a log-scheme over $k^0$,  is ordinary.
\end{proposition}
\begin{proof}As suggested by Illusie in \cite[Rem. 2.8]{Ill1}, this can be proved using the weight spectral sequence
\begin{align*}
\label{weights-crr}
E^{-k,i+k}_1 & =\bigoplus_{j\geq 0, j\geq -k}H^{i-s-j}_{\eet}(Y_{2j+k+1},W\Omega^{s-j-k})(-j-k)\Rightarrow H^{i-s}_{\eet}(Y,W\Omega^s).
\end{align*}
Here  $Y_{t}$ denotes the intersection of $t$ different irreducible components of $Y$ that are equipped with the trivial log-structure. 
 Such spectral sequences were constructed in  \cite[3.23]{Mok}, \cite[4.1.1]{Nak}.  They are Frobenius equivariant (the Tate twist $(-j-k)$ refers to the twist of Frobenius by $p^{j+k}$) \cite[Theorem 9.9]{Nak};
 hence, without the Twist twist,  compatible with the de Rham-Witt Frobenius $F$. 
 
 Now, by assumptions, all the schemes $Y_t$ are smooth and  ordinary. It follows, by Proposition \ref{rzeszow5},  that the Frobenius $F$ induces an isomorphism on $E^{-k,i+k}_1$. Hence also on the abutment $H^{i-s}_{\eet}(Y,W\Omega^s_Y)$, as wanted.
\end{proof}
\end{example}

 {\em We drop now the assumption that $V$ is proper.} Recall that we have   the Cartier isomorphism
$$
C: Z^j/  B^j \stackrel{\sim}{\to}\Omega^j,\quad 
x^p\dlog y_1\wedge\ldots\wedge\dlog y_j \mapsto x\dlog y_1\wedge\ldots\wedge\dlog y_j.
$$
\begin{lemma}
\label{final1}Assume that $H^i_{\eet}(V,\Omega^j)=0$ and that $d=0$ on $H^0_{\eet}(V,\Omega^j)$ for all $i\geq 1$ and $j\geq 0$.
Then $V$ is   ordinary \cite[4]{Lor}, i.e., for $i,j\geq 0$, we have 
$H^i_{\eet}(V,  B^j)=0.$ 
 \end{lemma}
\begin{proof}
Consider  the exact sequences
\begin{equation}
\label{sweden1}
0\to    B^j\to  Z^j\stackrel{f}{\to} \Omega^j\to 0,\quad
0\to   Z^j\to\Omega^j\to  B^{j+1}\to 0,
\end{equation}
where the map $f$ is the composition $Z^j\to Z^j/B^j\stackrel{\sim}{\to} \Omega^j$ of the natural projection and the Cartier isomorphism.
Since
$H^i_{\eet}(V,\Omega^j)=0,i >0,$
the first exact sequence yields  the isomorphisms
 \begin{equation}
 \label{sweden0}
H^i_{\eet}(V,  B^j)\stackrel{\sim}{\to} H^i_{\eet}(V, Z^j),\quad i\geq 2.
\end{equation}
It also yields the long exact sequence
\begin{align}
\label{sweden3}
0\to H^0_{\eet}(V,  B^j) {\to} H^0_{\eet}( V,Z^j){\to} H^0_{\eet}(V,\Omega^j)\stackrel{\partial}{\to} H^1_{\eet}(V,  B^j){\to} H^1_{\eet}(V, Z^j)\to 0.
\end{align}

   Since $d=0$ on $H^0_{\eet}(V,\Omega^j)$ and hence the natural map $H^0_{\eet}(V, Z^j)\to H^0_{\eet}(V,\Omega^j)$ is an isomorphism, 
   the second exact sequence from (\ref{sweden1}) yields 
the isomorphisms (since we assumed $H^i_{\eet}(V,\Omega^j)=0$ for $i>0$)
\begin{equation}
\label{sweden2}
H^i_{\eet}(V,  B^{j+1})\stackrel{\sim}{\to} H^{i+1}_{\eet}( V,Z^j),\quad i\geq 0.
\end{equation}

   To prove the lemma, we will argue by increasing induction on $j$; the case of $j=0$ being trivial  since $  B^0=0$. 
Assume thus that our lemma is true for $j$ and all $i\geq 0$.  We will show that this implies that it is true for $j+1$ and all $i\geq 0$. 
   Since $H^1_{\eet}( V, B^j)=0$ by assumption, the exact sequence (\ref{sweden3}) implies that
$H^1_{\eet}(V, Z^j)=0$.  And this implies, by (\ref{sweden0}), that $H^i_{\eet}(V, Z^r)=0,i\geq 1$.
 This, in turn,  yields, by (\ref{sweden2}), that 
   $H^i_{\eet}( V, B^{j+1})=0$, $i\geq 0$. 
   This concludes the proof of  the  lemma.
   \end{proof}

\subsubsection{${\mathbb H}_K$ as a pro-ordinary log-scheme}
It follows from Lemma~\ref{final1} and Proposition~\ref{general-princeton} that:
\begin{corollary}
\label{gen-ordinary}
The idealized log-schemes $Y_s, s\in\N\cup \{\infty\}$, are ordinary.
\end{corollary}
\begin{remark}
Proposition \ref{Illusie2} and Proposition \ref{Illusie1} show that the underlying scheme of $Y_s,$ for $s<\infty,$ is (classically) ordinary by using the weight spectral sequence. One should be able to prove Corollary \ref{gen-ordinary} in an analogous way.
\end{remark}

\begin{lemma}\label{degenerate}
For $i\geq 1, j\geq 0$, we have 
\begin{enumerate}
\item $H^i_{\eet}(Z, W_n\Omega^j)=0$, for $Z=Y,T$,
\item $d=0$  on $H^0_{\eet}(T,W_n\Omega^j)$.
\item For $V=Y,\overline{Y}$, 
the following sequence is strictly exact\footnote{Do not confuse $V$ with the Verschiebung in $V^n$.}
$$0\to H^0(V, \Omega^j)\stackrel{V^n}{\to} H^0(V, W_{n+1}\Omega^j)\to H^0(V, W_n\Omega^j)\to 0,
$$
\end{enumerate}
\end{lemma}
\begin{proof}
For claim (1), we start with $Z=Y$.  We have subsheaves 
$$0=B_0^j\subset B^j_1\subset...\subset Z^j_1\subset Z^j_0=\Omega^j_Y$$
such that $B_1^j=B^j_Y$, $Z_0^j=\Omega^j_Y$,  
$Z_1^j=Z^j_Y$ and for all $n$ we have inverse Cartier isomorphisms
$$C^{-1}:B_n^j\stackrel{\sim}{\to} B^j_{n+1}/B^j_1, \quad C^{-1}: Z^j_n\stackrel{\sim}{\to} Z^j_{n+1}/B^j_1.$$
By Proposition \ref{GK01} and Lemma \ref{final1}, we have  $H^i_{\eet}(Y,B^j_1)=H^i_{\eet}(Y,Z^j_1)=0$ for $i>0$, thus 
the same holds with $B^j_1$ and $Z^j_1$ replaced by $B^j_n$ and $Z^j_n$. 
On the other hand, define $R_n^j$ by the exact sequence 
\begin{equation} \label{eyes0}
0\to R_n^j\to B_{n+1}^j\oplus Z^{j-1}_n\to B^j_1\to 0,
\end{equation}
the last map being $(C^n, dC^{n-1})$. By the previous discussion, we have $H^i_{\eet}(Y,R_n^j)=0$ for $i>0$.
      Hyodo and Kato prove \cite[Theorem 4.4]{HK} that we have an exact sequence 
           \begin{equation}
           \label{eyes1}
           0\to \frac{\Omega^j\oplus \Omega^{j-1}}{R_n^j}\to W_{n+1}Ê\Omega^j\to W_n\Omega^j\to 0.
           \end{equation}
           Note that $ \frac{\Omega^j\oplus \Omega^{j-1}}{R_n^j}$ does not have higher cohomology since 
           each of $\Omega^j, \Omega^{j-1}, R_n^j$ has this property (use Proposition \ref{GK01}). Using the previous exact sequence, the result follows by induction on $n$ (using that $W_1\Omega^j\simeq \Omega^j$).
           
           In the case of $Z=T$ we argue in a similar way using Proposition \ref{GK5} instead of Proposition \ref{GK01}.
           
           For claim (2), since 
$\Gamma_{\eet}(T,W_n\Omega^j)\hookrightarrow \Gamma_{\eet}(T_{\overline{k}},W_n\Omega^j)$, we can pass to $T_{\overline{k}}$. But then, by ordinarity of $T_{\overline{k}}$, we have (see Proposition \ref{rzeszow5})
$$
H^0_{\eet}(T_{\overline{k}},W_n\Omega^j)\simeq H^0_{\eet}(T_{\overline{k}},W_n\Omega^j_{\log})\otimes_{\Z/p^n}W_n(\overline{k})
$$
and the latter group  clearly has a trivial differential. 

 To prove claim (3), we note first that Lemma \ref{final1} applies to both $Y$ and $\overline{Y}$. For $Y$ this follows from  Proposition \ref{GK01}. For $\overline{Y}$, we use Corollary \ref{gen-ordinary} to write down a sequence of quasi-isomorphisms
 \begin{align*}
 \R\Gamma(\overline{Y},\Omega^j) & \simeq \holim_s\R\Gamma(\overline{Y}^{\circ}_s,\Omega^j)\simeq \holim_s\R\Gamma(\overline{Y}_s,\Omega^j)\\
 & \simeq  \holim_sH^0(\overline{Y}_s,\Omega^j)\simeq \varprojlim_sH^0(\overline{Y}_s,\Omega^j).
 \end{align*}
 It follows that $H^i(\overline{Y},\Omega^j)=0$ for $i>0$. To see that $d=0$ on $H^0(\overline{Y},\Omega^j)$ we use the embedding
 $H^0(\overline{Y},\Omega^j)\hookrightarrow \prod_{C\in F^0}H^0(\overline{C},\Omega^j)$ and Proposition \ref{GK5}. 
 
   Now, set $V=Y,\overline{Y}$. By Lemma \ref{final1}, we have
   $H^i_{\eet}(V,B^j)=0$ $i,j\geq 0$.
   Note that,  by (\ref{eyes1}), 
  we have the exact sequence
$$0\to (H^0_{\eet}(V, \Omega^j)\oplus H^0_{\eet}(V, \Omega^{j-1}))/H^0_{\eet}(V, R_n^j)\verylomapr{(V^n,dV^n)}  H^0_{\eet}(V, W_{n+1}\Omega^j)\to H^0_{\eet}(V, W_n\Omega^j)\to 0.$$
 It remains to show that the natural map from $H^0_{\eet}(V, \Omega^j)$  to the leftmost term is an isomorphism, or that, the natural map
 $H^0_{\eet}(V, R_n^j)\to H^0_{\eet}(V, \Omega^{j-1})$ is an isomorphism.
 The exact sequence (\ref{eyes0})
                   yields that the natural map $H^0_{\eet}(V, R_n^j)\to  H^0_{\eet}(V, Z_n^{j-1})$ is an isomorphism. It remains thus to show  that so is the natural map
                   $ H^0_{\eet}(V, Z_n^{j-1})\to H^0_{\eet}(V, \Omega^{j-1})$.
                   
                   For that it suffices to show that the natural maps
                   $ H^0_{\eet}(V, Z_{n+1}^{j-1})\to H^0_{\eet}(V, Z^{j-1}_n)$, $n\geq 0$, are isomorphisms.  
We will argue by induction 
                   on $n\geq 0$. Since $d=0$ on $H^0_{\eet}(V, \Omega^{j-1})$ this is clear for $n=0$. Assume now that this is true for $n-1$. We will show it for $n$ itself. 
                  Consider
 the commutative diagram
  $$
  \xymatrix{
  H^0_{\eet}(V,Z^{j-1}_{n+1})\ar[r]^{C}_{\sim} \ar[d]_{\can} & H^0_{\eet}(V,Z^{j-1}_{n})\ar[d]^{\wr}_{\can}\\
  H^0_{\eet}(V,Z^{j-1}_n)\ar[r]^{C}_{\sim} & H^0_{\eet}(V,Z^{j-1}_{n-1})
  }
  $$
  The top and bottom  isomorphisms follow from the isomorphism $C^{-1}: Z^j_i\stackrel{\sim}{\to}Z^j_{i+1}/B^j_1$, $i\geq 0$. The right vertical map is an isomorphism by the inductive assumption. We get that the left vertical map is an isomorphism, as wanted.
  
   Finally, to see that the exact sequence in claim (3) is strictly exact note that for $Y$ this follows from compactness of $H^0(Y,\Omega^j)$ and $H^0(Y,W_{n+1}\Omega^j)$ and for $\overline{Y}$ this follows from the case of $Y$ by \'etale base change.
 \end{proof}

   \subsubsection{Cohomology of differentials II}
We will need a generalization of the above results and a more careful discussion of topological issues. 
\begin{proposition} \label{lyon11}
Let  $j\geq 0$. Let $S$ be a topological  $\so_K$-module and let $R$ be a topological $W(k)$- or $W(\overline{k})$-module. Assume that $S$ and $R$ are topologically orthonormalizable.
\begin{enumerate}
\item The following natural maps are strict quasi-isomorphisms (in $\sd(\Ind(PD_?))$, $?=K,F,\Q_p$)
\begin{align*}
& H^0(X,\Omega^j_{X,n})\wh{\otimes}_{\so_{K,n}} S_n\stackrel{\sim}{\to}
\R\Gamma(X,\Omega^j_{X,n})\wh{\otimes}_{\so_{K,n}} S_n,\\
&    H^0_{\eet}(Y,W_n\Omega^j)   \wh{\otimes}_{\so_{F,n}}R_n
      \overset{\sim}{\to }\R\Gamma_{\eet}(Y,W_n\Omega^j)   \wh{\otimes}_{\so_{F,n}}R_n,\\   
&  \R\Gamma_{\eet}(\overline{Y},W_n\Omega^j_{\log})   \wh{\otimes}_{\Z/p^n}R_n
   \overset{\sim}{\to }\R\Gamma_{\eet}(\overline{Y},W_n\Omega^j)   \wh{\otimes}_{W_n(\overline{k})}R_n,\\
    & H^0_{\eet}(\overline{Y},W_n\Omega^j_{\log})   \wh{\otimes}_{\Z/p^n}R_n\overset{\sim}{\to }
   \R\Gamma_{\eet}(\overline{Y},W_n\Omega^j_{\log})   \wh{\otimes}_{\Z/p^n}R_n.
    \end{align*}
\item  $d=0$ on
 $H^0_{\eet}({X},\Omega^j_{{X},n})\wh{\otimes}_{\so_{K,n}} S_n$ and on $H^0_{\eet}(Y,W_n\Omega^j_{Y}) \wh{\otimes}_{\Z/p^n}R_n . $
 \item The following natural map is a strict quasi-isomorphism
 $$
 \bigoplus_{j\geq r}H^0(X,\Omega^j_{X,n})[-j]\stackrel{\sim}{\to}F^r\R\Gamma_{\dr}(X_n), \quad r\geq 0.
 $$
\end{enumerate}
\end{proposition}
\begin{remark}
The completed tensor products for the above complexes of pro-discrete modules can be made
more explicit using a Stein covering $\{U_i\},\,i\in\N$ of $Y$. For example:
$$
\R\Gamma_{\crr}(Y/\so_{F,n}^0)  \wh{\otimes}_{\so_{F,n}}\widehat{\A}_{\st,n}\simeq\holim_i (\R\Gamma_{\crr}(U_{i}/\so_{F,n}^0)  {\otimes}_{\so_{F,n}}\widehat{\A}_{\st,n}).
$$
Note that $\widehat{\A}_{\st,n}$ has discrete topology. 
\end{remark}

\begin{proof}
Note that the last claim follows from the previous two claims. 

 In the rest of the  proof, to lighten the notation, we will write simply  $\R\Gamma(Z,\Omega\kr_n):=\R\Gamma(Z,\Omega\kr_{Z,n})$ for the de Rham cohomology of the log-scheme $Z_n$.
 We have the spectral sequence
  $$
  E_2^{q,i}=\wt{H}^q\holim_s (\wt{H}^i_{}({Y}^{\circ}_s,\Omega^j_n){\otimes}_{\so_{K,n}} S_n)\Rightarrow \wt{H}^{q+i}_{}(\R\Gamma(X,\Omega^j_n)\wh{\otimes}_{\so_{K,n}} S_n).
  $$
Since the pro-systems
 $$\{\wt{H}^i_{}({Y}^{\circ}_s,\Omega^j_{n}){\otimes}_{\so_{K,n}} S_n\},s\geq 0,\quad
\{\wt{H}^i_{}({Y}_s,\Omega^j_{n}){\otimes}_{\so_{K,n}} S_n\},s\geq 0,$$
are equivalent (and  $\wt{H}^i_{}({Y}_s,\Omega^j_{n})$ is classical and of finite type since $Y_s$ is  ideally log-smooth and proper over $k^0$), they both  have  trivial $\wt{H}^q\holim_s$, $q>0$. Hence the spectral sequence degenerates and we have
\begin{align*}
\wt{H}^i_{}(\R\Gamma(X,\Omega^j_{n})\wh{\otimes}_{\so_{K,n}} S_n)  \simeq \varprojlim_s(\wt{H}^i_{}({Y}^{\circ}_s,\Omega^j_{n}){\otimes}_{\so_{K,n}} S_n)\simeq \varprojlim_s(H^i_{}({Y}_s,\Omega^j_{n}){\otimes}_{\so_{K,n}} S_n).
\end{align*}
In particular, it is classical.

Moreover, using a basis $\{e_{\lambda}\}, \lambda\in I,$ of $S_n$ over $\so_{K,n}$, we get an  embedding 
 \begin{align*}
 \varprojlim_s(\wt{H}^{i}(Y^{\circ}_s,\Omega^j_{n}){\otimes}_{\so_{K,n}}S_n) \hookrightarrow   \prod_{\lambda\in I} H^{i}_{}(X,\Omega^j_{n})e_{\lambda}
\end{align*}
Since the latter groups are trivial for $i>0$, by Proposition \ref{GK01}, the vanishing of  $\wt{H}^i_{}(\R\Gamma(X,\Omega^j_{n})\wh{\otimes}_{\so_{K,n}} S_n)$ follows. This embedding also shows that  $d=0$ on $H^0$ in part (2) of the proposition.

 The proof for the second map in part (1) of the proposition is analogous with Lemma \ref{degenerate} replacing Proposition \ref{GK01}.
 
  For the proof for the third map in part (1) of the proposition,  consider now the sequence of strict quasi-isomorphisms
\begin{align*}
 \R\Gamma_{\eet}(\overline{Y},W_n\Omega^j_{\log})  \wh{\otimes}_{\Z/p^n}R_n  & =
\holim_s(\rg_{\eet}(\overline{Y}^{\circ}_s,W_n\Omega^j_{\log})  {\otimes}_{\Z/p^n}R_n)\simeq \holim_s(\rg_{\eet}(\overline{Y}_s,W_n\Omega^j_{\log})  {\otimes}_{\Z/p^n}R_n) \\
& \overset{\sim}{\to} \holim_s(\rg_{\eet}(\overline{Y}_s,W_n\Omega^j)  {\otimes}_{W_n(\overline{k})}R_n)\simeq 
\holim_s(\rg_{\eet}(\overline{Y}^{\circ}_s,W_n\Omega^j_{})  {\otimes}_{W_n(\overline{k})}R_n)\\
& =  \R\Gamma_{\eet}(\overline{Y},W_n\Omega^j)    \wh{\otimes}_{W_n(\overline{k})}R_n.
\end{align*}
The second and the fourth strict quasi-isomorphisms  are clear. The third strict quasi-isomorphism follows from the fact that, by Corollary \ref{gen-ordinary},  the log-scheme 
$\overline{Y}_s$ is ordinary and we have Proposition \ref{rzeszow5}.


 For the fourth strict quasi-isomorphism in part (1) of the proposition, use  the second and the third one to reduce to showing that we have a natural topological isomorphism
 $$
 H^0_{\eet}(\overline{Y},W_n\Omega^j_{\log})   \wh{\otimes}_{\Z/p^n}R_n\simeq  H^0_{\eet}(\overline{Y},W_n\Omega^j)   \wh{\otimes}_{W_n(\overline{k})}R_n.
  $$
  But, by Proposition \ref{rzeszow5}, we have topological isomorphisms
  \begin{align*}
   H^0_{\eet} & (\overline{Y},W_n\Omega^j_{\log})   \wh{\otimes}_{\Z/p^n}R_n=\varprojlim_s (H^0_{\eet}(\overline{Y}^{\circ}_s,W_n\Omega^j_{\log})   {\otimes}_{\Z/p^n}R_n)
   \simeq    \varprojlim_s (H^0_{\eet}(\overline{Y}_s,W_n\Omega^j_{\log})   {\otimes}_{\Z/p^n}R_n)\\
 &   
   \simeq \varprojlim_s (H^0_{\eet}(\overline{Y}_s,W_n\Omega^j)   {\otimes}_{W_n(\overline{k})}R_n)\simeq
   \varprojlim_s (H^0_{\eet}(\overline{Y}^{\circ}_s,W_n\Omega^j)   {\otimes}_{W_n(\overline{k})}R_n)      =
    H^0_{\eet}(\overline{Y},W_n\Omega^j)   \wh{\otimes}_{W_n(\overline{k})}R_n.
        \end{align*}

  It remains to show that $d=0$ on   
   $ H^0_{\eet}({Y},W_n\Omega^j)  \wh{\otimes}_{\so_{F,n}}R_n $. Assume first that $R$ is a $W(\overline{k})$-module. Arguing as above we obtain the embedding (notation as above)
  $$
       H^0_{\eet}(Y,W_n\Omega^j)    \wh{\otimes}_{\so_{F,n}}R_n  \overset{\sim}{\to}H^0_{\eet}({\overline{Y}},W_n\Omega^j_{\log}) \wh{\otimes}_{\Z/p^n}\R_n\hookrightarrow\prod_{\lambda\in I} H^0_{\eet}({\overline{Y}},W_n\Omega^j_{\log}) e_{\lambda}.       
$$
$d=0$ follows. If $R$ is only a $W(k)$-module, we 
 write
$$  H^0_{\eet}({Y},W_n\Omega^j) \wh{\otimes}_{\so_{F,n}}R_n\hookrightarrow 
 H^0_{\eet}(Y,W_n\Omega^j)  \wh{\otimes}_{\so_{F,n}}(W_n(\overline{k})\otimes_{\so_{F,n}} R_n)$$  to  obtain $d=0$ in this case as well. 
   \end{proof}
   
\begin{corollary}
\label{amtrak}
\begin{enumerate}
  \item For $j\geq 0$, we have a canonical topological isomorphism\footnote{More specifically, topological isomorphism of  projective limits of $\overline{k}$-vector spaces of finite rank.}
  $$
  H^0_{\eet}(\overline{Y},
  \Omega^j_{\log})\wh{\otimes}_{{\mathbf F}_p}\overline{k}\stackrel{\sim}{\to} H^0_{\eet}(\overline{Y},\Omega^j).
  $$
  \item For $j,n\geq 0$, the canonical maps 
  \begin{align*}
  H^0_{\eet}(\overline{Y},W_n\Omega^j_{\log})\wh{\otimes}_{\Z/p^n}W_n(\overline{k})   & \to H^0_{\eet}(\overline{Y},W_n\Omega^j),\\
    H^0_{\eet}(\overline{Y},W\Omega^j_{\log})\wh{\otimes}_{\Z_p}W(\overline{k})  & \to H^0_{\eet}(\overline{Y},W\Omega^j)
\end{align*}
  are  topological isomorphisms\footnote{More specifically, topological isomorphisms of  projective limits of $W_?(\overline{k})$-modules, free and of finite rank.}. In higher degrees all the above cohomology groups are trivial.
  \item The cohomologies $\wt{H}^i(X, \Omega^j_X)$ and $\wt{H}^j_{\dr}(X)$ are classical, 
  $
  {H}^i(X, \Omega^j_X)=0 $ for $i>0$, and 
  $ {H}^j_{\dr}(X)\simeq {H}^0(X, \Omega^j_X). $
   \end{enumerate}
\end{corollary}
\begin{proof}
The first two quasi-isomorphisms are actually included in the above proposition. For the third quasi-isomorphism, both sides are nontrivial only in degree zero: 
by Lemma \ref{degenerate} and the second isomorphism of this corollary, the projective systems  $\{H^0_{\eet}(\overline{Y},W_n\Omega^j_{\log})\wh{\otimes}_{\Z/p^n}W_n(\overline{k})\}_n$   and $\{H^0_{\eet}(\overline{Y},W_n\Omega^j)\}_n$  are Mittag-Leffler.   In degree zero we pass, as usual, to the limit over the truncated subschemes of the special fiber and there, since these subschemes are ordinary,  we have a term-wise isomorphism,  as wanted.

 For the cohomology $\wt{H}^i(X, \Omega^j_X)$, the fact that it is classical follows from the fact that the cohomology $\wt{H}^i(X_n, \Omega^j)$ is classical and nontrivial only for $i=0$, which was proved in 
 Proposition \ref{lyon11}, and the fact that the natural maps $H^0(X_{n+1},\Omega^j)\to H^0(X_n,\Omega^j)$ are surjective: a direct consequence, via Proposition \ref{lyon11}, of Lemma \ref{degenerate}.

 For the cohomology $\wt{H}^j_{\dr}(X)$, by Proposition \ref{lyon11}, we have
 $
 \wt{H}^j_{\dr}(X_n)\stackrel{\sim}{\leftarrow} H^0(X_n,\Omega^j).
 $ It follows that, since the maps $H^0(X_{n+1},\Omega^j)\to H^0(X_n,\Omega^j)$ are surjective,   $\wt{H}^j_{\dr}(X)$ is classical (by a Mittag-Leffler argument). 
 \end{proof}
\begin{remark}There is an alternative argument which proves Proposition \ref{lyon11} and which does not use 
 ordinarity of the truncated log-scheme $Y_s$.  It starts with proving the above corollary. We present it in the Appendix.
\end{remark}

\subsubsection{de Rham cohomologies of the model and the generic fiber}
Proposition~\ref{injectdR} below will be crucial in understanding the de Rham cohomology of the model and
its variants.

 Define the map 
$$
\iota_Y: H^i_{\eet}(Y,W\Omega\kr_Y)\simeq H^i_{\crr}(Y/\so^0_F)\to H^i_{\crr}(Y/\so^0_F,F)\stackrel{\sim}{\leftarrow}H^i_{\rig}(Y/\so^0_F)\stackrel{\iota_{\hk}}{\longrightarrow}H^i_{\rig}(Y/\so^{\times}_K).
$$
\begin{proposition}\label{injectdR} 
\begin{enumerate}
\item The above map induces an injection
$$
\iota_Y: H^i_{\eet}(Y,W\Omega\kr){\otimes}_{\so_F}K\hookrightarrow H^i_{\rig}(Y/\so^{\times}_K).
$$
\item The canonical map
$$
 H^i_{\dr}(X){\otimes}_{\so_K}K\rightarrow H^i_{\dr}(X_K)
$$
is injective.
\end{enumerate}
\end{proposition}
\begin{proof}

   For the first claim, it suffices to show that we have a commutative diagram
$$
\xymatrix{
H^i_{\eet}(Y,W\Omega\kr)\ar@{^(->}[r]^-{\alpha}\ar[d]^{\iota_Y} & \prod_{j\in\N} H^i_{\eet}(C_j,W\Omega\kr)\ar@{^(->}[d]^{\prod_j\iota_{C_j}}\\
H^i_{\rig}(Y/\so^{\times}_K)\ar[r] & \prod_{j\in\N} H^i_{\rig}(C_j/\so^{\times}_K),
}
$$
where ${C_j},j\in\N$, is the set of irreducible components of $Y$ and  the map $\iota_{C_j}$ is defined in an analogous way to the map $\iota_Y$ but by replacing the Hyodo-Kato map by the composition
$$
H^i_{\rig}(C_j/\so^0_F)\stackrel{\sim}{\to} H^i_{\rig}(C_j^0/\so^0_F)\lomapr{\iota_{\hk}} H^i_{\rig}(C_j^0/\so^{\times}_K)\stackrel{\sim}{\leftarrow} H^i_{\rig}(C_j/\so^{\times}_K).
$$
Since the Hyodo-Kato map is compatible with Zariski localization the above diagram commutes.

  We claim that we have  natural isomorphisms 
$$
H^0_{\eet}(Y,W\Omega^i)\stackrel{\sim}{\to} H^i_{\eet}(Y,W\Omega\kr),\quad H^0_{\eet}(C_j,W\Omega^i)\stackrel{\sim}{\to} H^i_{\eet}(C_j,W\Omega\kr).
$$
Indeed, set $Z=Y,C_j$. We have $H^0_{\eet}(Z,W\Omega^i)=\varprojlim_nH^0_{\eet}(Z,W_n\Omega^i)$. Since, by Proposition \ref{lyon11} and Lemma \ref{degenerate},
  $$\rg_{\eet}(Z,W_n\Omega\kr)\simeq
 \oplus_jH^0_{\eet}(Z,W_n\Omega^j)[-j],
 $$
this implies that
 $$
 H^i_{\eet}(Z,W\Omega\kr)\simeq \varprojlim_nH^i_{\eet}(Z,W_n\Omega\kr)\simeq \varprojlim_nH^0_{\eet}(Z,W_n\Omega^i), 
 $$
 as wanted. In particular, there is no torsion. 

   It follows that the maps $\iota_{C_j}$ in the above diagram are injections: they are isomorphisms after tensoring the domains with $K$ and the domains are torsion-free.
The map $\alpha$ is an injection because so is, by definition, the map $\alpha^{\prime}$ in the commutative diagram
$$
\xymatrix{
H^0_{\eet}(Y,W\Omega^i)\ar[drr]_-{\alpha^{\prime}}\ar[r]^-{\alpha} &  \prod_{j\in\N} H^0_{\eet}(C_j,W\Omega^i)\ar[r] &  \prod_{j\in\N} H^0_{\eet}(C^0_j,W\Omega^i)\\
 & & H^0_{\eet}(Y_{\tr},W\Omega^i_{Y_{\tr}})\ar[u]^{\wr},
}
$$
where $Y_{\tr}$ denotes the nonsingular locus of $Y$.

  We note that the above computation shows also that the natural map $H^i_{\crr}(Y/\so^0_F){\otimes}_{\so_F}F\to H^i_{\crr}(Y/\so^0_F,F)$
  is an injection. This will be useful in proving the second claim of the proposition. 
  Using the diagram (\ref{symbols-diag}) we can form a commutative diagram
  $$
  \xymatrix{
  H^i_{\dr}(X){\otimes}_{\so_K}K\ar[rr]^{\can} & & H^i_{\dr}(X_K)\\
  H^i_{\crr}(Y/\so_F^0){\otimes}_{\so_F}K\ar[u]^{\iota_{\hk}}_{\wr} \ar@{^(->}[r] &  H^i_{\crr}(Y/\so_F^0,F)\otimes_{F}K &  H^i_{\rig}(Y/\so_F^0)\otimes_{F}K\ar[u]^{\iota_{\hk}}_{\wr}\ar[l]_-{\sim}
  }
  $$
  Here the first map $\iota_{\hk}$ is the bounded Hyodo-Kato isomorphism described in the Appendix. Since the first bottom map is an injection so is the top map, as wanted.
\end{proof}
\subsection{Relation to Steinberg representations}
We proved in the previous section that, for all $i>0$, the spaces $H^i_{\eet}(Y,W\Omega^r)$ and  $H^i_{\eet}(\overline{Y},W\Omega^r_{\log})$ vanish. The purpose of this section is to prove the following result describing the corresponding spaces for $i=0$ in terms of generalized Steinberg representations. 
\begin{theorem}
\label{Steinberg}
Let $r\geq 0$.
  \begin{enumerate}
  \item We have  natural  isomorphisms of locally convex topological $\Q_p$-vector  spaces (more precisely, weak duals of Banach spaces)
  \begin{enumerate}
  \item $H^0(Y, W\Omega^r){\otimes}_{\so_F}F\simeq H^r(Y, W\Omega\kr){\otimes}_{\so_F}F\simeq {\rm Sp}^{\cont}_r(F)^*$,
  \item  $H^0_{\eet}(Y, W\Omega^r_{{\rm log}}){\otimes}\Q_p\simeq {\rm Sp}^{\cont}_r(\Q_p)^*$,
  \item $ H^0(X,\Omega^r){\otimes}_{\so_K}K\simeq H^r_{\dr}(X){\otimes}_{\so_K}K\simeq {\rm Sp}_r^{\cont}(K)^*$,
\item $H^0_{\eet}(\overline{Y},W\Omega^r_{{\rm log}}){\otimes}\Q_p\simeq {\rm Sp}_r^{\cont}(\Q_p)^*.
$
 \end{enumerate}
    They are compatible with the canonical maps between Steinberg representations and with the isomorphisms 
  $$
  H^r_{\dr}(X_K)\simeq {\rm Sp}_r(K)^*,\quad H^r_{\hk}(X)\simeq {\rm Sp}_r(F)^*
  $$
  from Theorem \ref{IS} and Lemma \ref{genIS}. 
  \item We have  natural  isomorphisms of pro-discrete $\Z_p$-modules
  \begin{enumerate}
  \item $H^0_{\eet}(Y,W\Omega^r)\simeq H^r_{\eet}(Y,W\Omega\kr)\simeq {\rm Sp}^{\cont}_r(\so_F)^*$ and $H^0(Y, \Omega^r)\simeq {\rm Sp}_r(k)^*$,
\item  $H^0_{\eet}(Y,W\Omega^r_{\log})\simeq {\rm Sp}^{\cont}_r(\Z_p)^*$ and $H^0_{\eet}(Y, \Omega^r_{{\rm log}})\simeq {\rm Sp}_r(\mathbf{F}_p)^*$,
\item $
H^0(X,\Omega^r)\simeq H^r_{\dr}(X)\simeq {\rm Sp}^{\cont}_r(\so_K)^*,
$
\item  $H^0_{\eet}(\overline{Y},W\Omega^r_{{\rm log}})\simeq {\rm Sp}_r^{\cont}({\mathbf Z}_p)^*$ and 
$H^0_{\eet}(\overline{Y},\Omega^r_{{\rm log}})\simeq {\rm Sp}_r({\mathbf F}_p)^*$.
\end{enumerate}
They are compatible with the canonical maps between Steinberg representations and with the above isomorphisms.
\end{enumerate}
\end{theorem}
\begin{proof}
Consider the  following diagram
\begin{equation}
\label{zima2}
\xymatrix{  & H^r_{\eet}(Y,W\Omega\kr)\ar@{^(->}[r]^-{\iota_Y} & H^r_{\hk}(Y)\\
 D(\mathcal{H}^{r+1}, \so_F)\ar[ru]^{r_{\hk}} \ar@{^(->}[r]^{\can}\ar@{->>}[d]_{\can} & D(\mathcal{H}^{r+1}, F)\ar[ru]^{r_{\hk}} \ar@{->>}[d]_{\can} \\
 {\rm Sp}^{\cont}_r(\so_F)^*\ar@{-->}[uur]_{r_{\hk}}\ar@{^(->}[r]^{\can}  & {\rm Sp}_r(F)^*\ar[uur]_{\sim}
}
\end{equation}
The bottom square clearly commutes. The first (continuous) regulator $r_{\hk}$ is defined by integrating the crystalline Hyodo-Kato Chern classes $c_1^{\hk}$ defined in the Appendix.  By Section \ref{diffsymbols} it makes the top square commute.
 The right triangle commutes by Lemma \ref{genIS}. 
It follows that there exists a broken arrow (we will call it $r_{\hk}$ as well) that makes the left triangle commute. This map is continuous and  also clearly makes the adjacent square commute. Hence it is an injection. 
We will prove that it is an isomorphism after inverting $p$.

 The above combined with Proposition \ref{injectdR} and Theorem \ref{genIS} gives the   embeddings 
             $$     
               {\rm Sp}_r^{\rm cont}(\so_F)^*{\otimes}_{\so_F} F\stackrel{r_{\hk}}{\hookrightarrow} H^0_{\eet}(Y, W\Omega^r){\otimes}_{\so_F} F\stackrel{f}{\hookrightarrow } H^r_{\rm \hk}(Y)\simeq {\rm Sp}_r(F)^*.$$
   Their composite is the canonical embedding.
     The image of the map $f$ must be in the subspace of $G$-bounded vectors of ${\rm Sp}_r(K)^*$, since 
     $H^0_{\eet}(Y, W\Omega^r)$ is compact (it is naturally an inverse limit of 
             finite free $\so_F$-modules). That  subspace is identified with 
       ${\rm Sp}_r^{\rm cont}(F)^*\simeq  {\rm Sp}_r^{\rm cont}(\so_F)^*{\otimes}_{\so_F} F$ by Corollary \ref{Gbounded}. It follows that the map $r_{\hk}$ is an isomorphism.

         In fact, the above map $r_{\hk}$ is already an integral isomorphism (as stated in part (2a)). To see this, consider the commutative diagram
  $$
  \xymatrix{
   {\rm Sp}_r^{\rm cont}(\so_F)^*\ar@{->>}[d] \ar@{^(->}[r]^{r_{\hk}} & H^0_{\eet}(Y, W\Omega^r)\ar@{->>}[d]\\
    {\rm Sp}_r^{\rm cont}(k)^* \ar[r]^{r_{\hk}} & H^0_{\eet}(Y, \Omega^r)
  }
  $$
  $H^0_{\eet}(Y, W\Omega^r)$ is a $G$-equivariant lattice in 
  $H^0_{\eet}(Y, W\Omega^r)_{\Q_p}\simeq {\rm Sp}_r^{\rm cont}(F)^*$ hence, by Corollary \ref{Uniquelattice}, it is homothetic  to ${\rm Sp}_r^{\rm cont}(\so_F)^*$. It follows that
  $H^0_{\eet}(Y, \Omega^r)\simeq {\rm Sp}_r^{\rm cont}(k)^*$ is irreducible. Moreover, the bottom map $r_{\hk}$ is nonzero: by construction of the top map $r_{\hk}$, the symbol $\dlog z_1\wedge\cdots\wedge \dlog z_r$ for coordinates $z_1,\ldots, z_r$ of ${\mathbb P}^d_K$ is in the image. 
 It follows that it is an isomorphism hence so is the top map $r_{\hk}$ as well. Moreover, the latter is a topological isomorphism since the domain is compact. It follows that its
  rational version is a topological isomorphism as well, which proves part (1a) of the theorem.
 
 The proof of part (1b) is very similar to the proof of part (1a), so we will be rather brief.  Consider the commutative diagram
\begin{equation}
\label{zima3}
\xymatrix{  H^r_{\eet}(Y,W\Omega\kr_{\log})\ar@{^(->}[rr]^{f}  & & H^r_{\eet}(Y,W\Omega\kr)^{\phi=p^r} \\
& D(\mathcal{H}^{r+1}, \Z_p)\ar[ru]^{r_{\hk}}\ar@{->>}[d]_{\can}\ar[lu]_{r_{\log}}  \\
  & {\rm Sp}^{\cont}_r(\Z_p)^*\ar[uur]_{\sim}\ar@{-->}[uul]^{r_{\log}}
}
\end{equation}
Here the (continuous) regulator $r_{\log}$ is defined by integrating the crystalline logarithmic de Rham-Witt  Chern classes $c_1^{\log}$ defined in the Appendix.  
Arguing as above we can construct the broken arrow, which is again a continuous map,   making the whole diagram commute. 
 It easily follows that both maps $f$ and $r_{\log}$ are isomorphisms.  Now, to prove that they are topological isomorphisms we argue  first integrally, as for part (2b), and then rationally as for part (2a).
 
 For part (1c) one repeats the argument starting with  the  following commutative diagram
\begin{equation}
\label{zima1}
\xymatrix{ & H^r_{\dr}(X)\ar@{^(->}[r]^{\can}& H^r_{\dr}(X_K)\ar@{-}[ddl]^{\sim}\\
 D(\mathcal{H}^{r+1}, \so_K)\ar[ru]^{r_{\dr}} \ar@{^(->}[r]^{\can}\ar@{->>}[d]_{\can} & D(\mathcal{H}^{r+1}, K)\ar[ru]^{r_{\dr}}\ar@{->>}[d]_{\can}\\
{\rm Sp}_r(\so_K)^*\ar@{-->}[uur]_{r_{\dr}}\ar@{^(->}[r]^{\can} & {\rm Sp}_r(K)^*,
}
\end{equation}
where the continuous (bounded) regulator $r_{\dr}$ is defined by integrating the integral de Rham  Chern classes $c_1^{\dr}$ defined in the Appendix,
and using the fact that, 
  by Proposition \ref{injectdR}, we have
  $$
  H^r_{\dr}(X){\otimes}_{\so_K}K\hookrightarrow H^r_{\dr}(X) \simeq {\rm Sp}_r(K)^* 
  $$
  and, 
  by Proposition \ref{lyon11}, we have $H^r_{\dr}(X)\simeq H^0(X,\Omega^r)$. The integral part (2c) follows as above.

Parts (1d) and (2d) follow from parts (1b) and (2b) and the following lemma. 
\begin{lemma}
\label{change-of-fields}
For $n\geq 1$, we have  canonical topological isomorphisms
$$H^0_{\eet}(Y,W_n\Omega^r_{{\rm log}})\stackrel{\sim}{\to}H^0_{\eet}(\overline{Y},W_n\Omega^r_{{\rm log}}),\quad H^0_{\eet}(Y,W\Omega^r_{{\rm log}})\stackrel{\sim}{\to}H^0_{\eet}(\overline{Y},W\Omega^r_{{\rm log}}).
$$
\end{lemma}
\begin{proof} It suffices to prove the first isomorphism and, since both sides satisfy $p$-adic devissage, it suffices to do it for $n=1$. 
We have $H^0_{\eet}(\overline{Y}, \Omega^r_{\log})\stackrel{\sim}{\to}H^0_{\eet}(\overline{Y}, \Omega^r)^{C=1}$. On the other hand, by \'etale base change, we have a topological isomorphism 
 $H^0_{\eet}(\overline{Y}, \Omega^r)\simeq H^0_{\eet}(Y, \Omega^r)\wh{\otimes}_{k}\overline{k}$.
And parts  (2a) and (2b) of the theorem show that the natural map
    $H^0_{\eet}(Y, \Omega^r_{\log})\wh{\otimes}_{\mathbf{F}_p} k
    \to H^0_{\eet}(Y, \Omega^r)$ is a topological  isomorphism. Hence, since $C=1$ on $H^0_{\eet}(Y, \Omega^r_{\log})$, we obtain topological isomorphisms
    $$H^0_{\eet}(\overline{Y}, \Omega^r_{\log}) \stackrel{\sim}{\leftarrow}H^0_{\eet}(\overline{Y}, \Omega^r)^{C=1}\stackrel{\sim}{\leftarrow}(H^0_{\eet}(Y, \Omega^r_{\log})\wh{\otimes}_{\mathbf{F}_p} \overline{k})^{C=1}\stackrel{\sim}{\leftarrow}H^0_{\eet}(Y, \Omega^r_{\log}),$$
    as wanted.
  \end{proof}
\end{proof}
\begin{remark}
\label{noc1}
Consider the commutative diagram
$$
\xymatrix{ & H^r_{\eet}(Y,W\Omega\kr)\ar@{-->}[r]^-{\iota_{\hk}} & H^r_{\dr}(X)\\
 D(\mathcal{H}^{r+1}, \so_F)\ar[ru]^{r_{\hk}} \ar@{^(->}[r]^{\can}\ar@{->>}[d]_{\can} & D(\mathcal{H}^{r+1}, \so_K)\ar[ru]^{r_{\dr}} \ar@{->>}[d]_{\can} \\
 {\rm Sp}^{\cont}_r(\so_F)^*\ar[uur]_{\sim}\ar@{^(->}[r]^{\can}  & {\rm Sp}^{\cont}_r(\so_K)^*\ar[uur]_{\sim}
}
$$
The dashed arrow is defined to make the diagram commute. It is continuous.  It can be thought of as an integral Hyodo-Kato map. Compatibilities used in the proof of Theorem \ref{Steinberg} ensure that it is compatible  with the bounded and the overconvergent Hyodo-Kato maps. 
Because the natural map 
$ {\rm Sp}^{\cont}_r(\so_F)^*\otimes_{\so_F}\so_K\stackrel{\sim}{\to} {\rm Sp}^{\cont}_r(\so_K)^*$ is an isomorphism, 
we get the integral Hyodo-Kato (topological) isomorphism
$$
\iota_{\hk}: H^r_{\eet}(Y,W\Omega\kr)\otimes_{\so_F}\so_K\stackrel{\sim}{\to}  H^r_{\dr}(X).
$$
\end{remark}
\subsection{Computation of syntomic cohomology}\label{ZURICH}
We will prove in this section that the geometric syntomic cohomology of $X$ can be computed using the logarithmic de Rham-Witt cohomology.  To simplify the notation we will write $(-)_{\Q_p}$ for $(-){\otimes}^L\Q_p$.
\subsubsection{Simplification of syntomic cohomology}
Let $X$ now be a semistable Stein formal scheme over $\so_K$. 
\begin{lemma}\label{zurich11}
Let $r\geq 0$. There exist compatible natural strict quasi-isomorphisms
\begin{align*}
\iota_{\crr}: \quad & [\R\Gamma_{\crr}(\overline{X})_{\Q_p}]^{\phi=p^r}\stackrel{\sim}{\to}[(\R\Gamma_{\crr}(X_0/\so_{F}^0)  \wh{\otimes}_{\so_{F}}\widehat{\A}_{\st})_{\Q_p}]^{N=0,\phi=p^r},\\
\iota_{\crr}: \quad & [\R\Gamma_{\crr}(X)_{\Q_p}]^{\phi=p^r}\stackrel{\sim}{\to}[\R\Gamma_{\crr}(X_0/\so_{F}^0)_{\Q_p}]  ^{N=0,\phi=p^r}.
\end{align*}
\end{lemma}
\begin{proof}By (\ref{comp1}),  (\ref{int1}),
we have a natural strict quasi-isomorphism 
$$
\iota_{\rm BK}^1:\quad    [\R\Gamma_{\crr}(X_n/r^{\rm PD}_{\varpi,n}) \wh{\otimes}_{r^{\rm PD}_{\varpi,n}}\wh{\A}_{\st,n}]^{N=0} \stackrel{\sim}{\to} \R\Gamma_{\crr}(\overline{X}_n).
$$
To check the strictness one can look locally and there everything is discrete. 
We can also adapt the proof of Theorem \ref{fini3} (and proceed as in the proof of Proposition \ref{fini11}: note that both crystalline cohomology complexes are in the bounded derived category) to construct a natural strict quasi-isomorphism
$$
h_{\crr}: \quad  
 [(\R\Gamma_{\crr}(X/r^{\rm PD}_{\varpi})\widehat{\otimes}_{r^{\rm PD}_{\varpi}}\widehat{\A}_{\st})_{\Q_p}]^{\phi=p^r}\stackrel{\sim}{\to}
 [(\R\Gamma_{\crr}(X_0/\so_{F}^0)  \wh{\otimes}_{\so_{F}}\widehat{\A}_{\st})_{\Q_p}]^{\phi=p^r}. 
$$
In fact, it suffices to note that the complexes (\ref{acyclic}) used in that proof have cohomology annihilated by $p^N$, for a constant $N=N(d,r)$, $d=\dim X_0$.

   Define the first map in the lemma by $\iota_{\crr}:=h_{\crr}\iota_{\rm BK}^{-1}$. The definition of the second map $\iota_{\crr}$ is analogous (but easier: there is no need for the zigzag in the definition of $h_{\crr}$).
\end{proof}
Let $r\geq 0$. 
From the maps in (\ref{paris11}) we induce  a natural strict quasi-isomorphism 
 \begin{align}
 \label{sunny}
  \iota_{\rm BK}^2:   (\R\Gamma_{\dr}(X)\wh{\otimes}_{\so_K}\A_{\crr,K})_{\Q_p}/F^r
   \stackrel{\sim}{\to}\R\Gamma_{\crr}(\overline{X})_{\Q_p}/F^r.
 \end{align}
Set $\gamma_{\hk}:=(\iota_{\rm BK}^2)^{-1}\iota_{\crr}^{-1}$. The above discussion  yields the following strict quasi-isomorphism 
\begin{align}
\label{reduction}
\R\Gamma_{\synt}(\overline{X},\Z_p(r))_{\Q_p}
 & \stackrel{\sim}{\to}
   \xymatrix{[[(\R\Gamma_{\crr}(X_0/\so_{F}^0)  \wh{\otimes}_{\so_{F}}\widehat{\A}_{\st})_{\Q_p}]^{N=0,\phi=p^r}\ar[r]^-{\gamma_{\hk}}
    &  (\R\Gamma_{\dr}({X})\wh{\otimes}_{\so_K}{\A}_{\crr,K})_{\Q_p}/F^r]}.
\end{align}
By construction,  it is compatible with its pro-analog (\ref{map2}), i.e., we have a natural continuous map of distinguished triangles, where all the vertical maps are the canonical maps
$$
\xymatrix{
 \R\Gamma_{\synt}(\overline{X},\Z_p(r))_{\Q_p}\ar[d]
 \ar[r] & [(\R\Gamma_{\crr}(X_0/\so_{F}^0)  \wh{\otimes}_{\so_{F}}\widehat{\A}_{\st})_{\Q_p}]^{N=0,\phi=p^r}\ar[d]\ar[r]^-{\gamma_{\hk}}
    &  (\R\Gamma_{\dr}({X})\wh{\otimes}_{\so_K}{\A}_{\crr,K})_{\Q_p}/F^r\ar[d]\\
 \R\Gamma_{\synt}(\overline{X},\Q_p(r))
 \ar[r] & [\R\Gamma_{\crr}(X_0/\so_{F}^0,F)  \wh{\otimes}^R_{{F}}\widehat{\B}^+_{\st}]^{N=0,\phi=p^r}\ar[r]^-{\gamma_{\hk}}
    &  (\R\Gamma_{\dr}({X_K})\wh{\otimes}^R_{K}{\B}^+_{\dr})/F^r.
    }
$$
\subsubsection{Computation of the Hyodo-Kato part}
We come back now to the Drinfeld half-space.
\begin{lemma}
\label{formulas1}The cohomology of $[(\R\Gamma_{\crr}(Y/\so_{F}^0)   \wh{\otimes}_{\so_{F}}\widehat{\A}_{\st})_{\Q_p}]^{N=0,\phi=p^r}$ is classical and we have 
  the natural topological isomorphisms
  \begin{align}
  \label{formulas}
   H^r([(\R\Gamma_{\crr}(Y/\so_{F}^0)  & \wh{\otimes}_{\so_{F}}\widehat{\A}_{\st})_{\Q_p}]^{N=0,\phi=p^r})\simeq H^0_{\eet}(\overline{Y},W\Omega^r_{\log}){\otimes}{\Q_p},\\
  H^{r-1}([(\R\Gamma_{\crr}(Y/\so_{F}^0)  & \wh{\otimes}_{\so_{F}}\widehat{\A}_{\st})_{\Q_p}]^{N=0,\phi=p^r})\simeq  (H^0_{\eet}(\overline{Y},W\Omega^{r-1}_{\log})
  \wh{\otimes}_{\Z_p}{\A}^{\phi=p}_{\crr}){\otimes}{\Q_p}.\notag
  \end{align}
  \end{lemma}
  \begin{proof}
    By Proposition \ref{lyon11},  we have  strict quasi-isomorphisms\footnote{Strictly speaking, the quasi-isomorphisms in that proposition are modulo $p^n$ but it is easy to get the $p$-adic result by going to the limit and using Mittag-Leffler as in Corollary \ref{amtrak}.}
 \begin{align*}
     \R\Gamma_{\crr}(Y/\so_{F}^0)   \wh{\otimes}_{\so_{F}}\widehat{\A}_{\st} & \simeq \R\Gamma_{\eet}(Y,W\Omega\kr)   \wh{\otimes}_{\so_{F}}\widehat{\A}_{\st}
  \overset{\sim}{\leftarrow} H^0_{\eet}(Y,W\Omega\kr)   \wh{\otimes}_{\so_{F}}\widehat{\A}_{\st}\\
  & \overset{\sim}{\leftarrow}\bigoplus_{i\geq 0}H^0_{\eet}(Y,W\Omega^i)   \wh{\otimes}_{\so_{F}}\widehat{\A}_{\st}[-i]
   \simeq  \bigoplus_{i\geq 0}H^0_{\eet}(\overline{Y},W\Omega^i)   \wh{\otimes}_{W(\overline{k})}\widehat{\A}_{\st}[-i].
   \end{align*}
The first claim of the lemma follows.

 Let $M$ be a finite type free $(\phi,N)$-module over  $W(\overline{k})$. Note that $N$ is nilpotent.  We claim
 that we have a short exact sequence
 $$
 0\to M\otimes_{W(\overline{k})} \A_{\crr}\to M\otimes_{W(\overline{k})}\wh{\A}_{\st}\stackrel{N}{\to}M\otimes_{W(\overline{k})}\wh{\A}_{\st}\to 0.
 $$
Indeed,  if $N=0$, this is clear from the short exact sequence (\ref{Breuil}). For a general $M$, we argue by induction on $m$ such that $N^m=0$ using the short exact sequence
 $$
 0\to M_0\to M\stackrel{N}{\to }M_1\to 0,
 $$
 where $M_0=\ker N, M_1=\im N$. $M_0,M_1$ 
are  finite type free $(\phi,N)$-modules such that $N^{m-1}=0$. It follows that we have  strict quasi-isomorphisms (reduce to the truncated log-schemes $Y_s$ and pass to the limit)
   \begin{align*}
  [\R\Gamma_{\crr}(Y/\so_{F}^0)    \wh{\otimes}_{\so_{F}}\widehat{\A}_{\st}]^{N=0} & \simeq \bigoplus_{i\geq 0}[H^0_{\eet}(\overline{Y},W\Omega^i)   \wh{\otimes}_{W(\overline{k})}\widehat{\A}_{\st}]^{N=0}[-i]\simeq 
   \bigoplus_{i\geq 0}H^0_{\eet}(\overline{Y},W\Omega^i)   \wh{\otimes}_{W(\overline{k})}{\A}_{\crr}[-i]\\
  & \simeq 
   \bigoplus_{i\geq 0}H^0_{\eet}(\overline{Y},W\Omega^i_{\log})   \wh{\otimes}_{\Z_p}{\A}_{\crr}[-i] .
   \end{align*}

  We will show now that we have natural strict quasi-isomorphisms
   \begin{align}
   \label{swieczki}
   [H^0_{\eet}(\overline{Y},W\Omega^r_{\log})   \wh{\otimes}_{\Z_p}{\A}_{\crr}]^{\phi=p^r}_{\Q_p} & \simeq H^0_{\eet}(\overline{Y},W\Omega^r_{\log})_{\Q_p}\\
   [H^0_{\eet}(\overline{Y},W\Omega^{r-1}_{\log})   \wh{\otimes}_{\Z_p}{\A}_{\crr}]^{\phi=p^r}_{\Q_p} & \simeq (H^0_{\eet}(\overline{Y},W\Omega^{r-1}_{\log})
  \wh{\otimes}_{\Z_p}{\A}^{\phi=p}_{\crr})_{\Q_p}.\notag
\end{align}
   For $i\geq 0$, 
   set $C_i:=
   H^0_{\eet}(\overline{Y},W\Omega^i_{\log})   \wh{\otimes}_{\Z_p}{\A}_{\crr}.$ We claim that, for $j\geq i$,  the classical eigenspace
\begin{equation}
\label{swieczka1}
C_i^{\phi=p^j}=
H^0_{\eet}(\overline{Y},W\Omega^i_{\log})   \wh{\otimes}_{\Z_p}{\A}^{\phi=p^{j-i}}_{\crr}. 
\end{equation}
To see that, write, using the notation from the proof of Proposition \ref{lyon11}, $H^0_{\eet}(\overline{Y},W\Omega^i_{{\rm log}})\simeq \varprojlim_sH^0_{\eet}(\overline{Y}_s,W\Omega^i_{{\rm log}})$ or, to simplify the notation, $A^i:=A=\varprojlim_sA_s$. Note that $A_s$ is a finite type free $\Z_p$-module. Replace $A_s$ with $B_s:=\cap_{s^{\prime}>s}\im(A_{s^{\prime}}\to A_s)$. Then the maps $B_{s+1}\to B_s$ are surjective and $A=\varprojlim_s B_s$. Choose basis of $B_s$, $s\geq 1$, compatible with the projections, i.e., the chosen basis of $B_{s+1}$ includes a lift of the chosen basis of $B_s$. Using this basis 
 we can write 
  \begin{align}
  \label{bases}
  A\simeq  & \Z_p^{I_1}\times\Z_p^{I_2}\times\Z_p^{I_3}\times\cdots,\quad 
    A\wh{\otimes}_{\Z_p}\A_{\crr}\simeq \A_{\crr}^{I_1}\times\A_{\crr}^{I_2}\times\A_{\crr}^{I_3}\times\cdots.
\end{align}
Since the Frobenius on $ A^i=H^0_{\eet}(\overline{Y},W\Omega^i_{\log})  $ is equal to the multiplication by  $p^i$ we obtain the equality we wanted.

   Consider now the following  exact sequences
\begin{align*}
0  \to C^{\phi=p^r}_{r}\to C_r \lomapr{\phi-p^r} C_r,\quad \quad
0  \to  C^{\phi=p^r}_{r-1}\to C_{r-1}\lomapr{\phi-p^r}C_{r-1}.
\end{align*}
Since the map $\A_{\crr}\lomapr{p^i-\phi} \A_{\crr}$, $i\geq 0$,  is $p^i$-surjective,  the maps $\phi-p^r$ above are rationally surjective
(use the basis (\ref{bases}) and the fact that the Frobenius on $A^i$ is equal to the multiplication by  $p^i$). 
Hence, rationally, the derived eigenspaces $[C_i]^{\phi=p^r}$, $i=r,r-1$, are equal to the classical ones $C_i^{\phi=p^r}$, $i=r,r-1.$ Since, by (\ref{swieczka1}), we have 
$C_r^{\phi=p^r}=H^0_{\eet}(\overline{Y},W\Omega^r_{\log})  $ and $C_{r-1}^{\phi=p^r}=H^0_{\eet}(\overline{Y},W\Omega^{r-1}_{\log}) \wh{\otimes}_{\Z_p}\A_{\crr}^{\phi=p}$,
the strict quasi-isomorphisms in (\ref{swieczki}) follow. 

It remains to show that the natural maps
\begin{align*}
H^0_{\eet}(\overline{Y},W\Omega^r_{\log})_{\Q_p} & \to H^0_{\eet}(\overline{Y},W\Omega^r_{\log}){\otimes}{\Q_p},\\
(H^0_{\eet}(\overline{Y},W\Omega^{r-1}_{\log}) \wh{\otimes}_{\Z_p}{\A}^{\phi=p}_{\crr})_{\Q_p} & \to 
(H^0_{\eet}(\overline{Y},W\Omega^{r-1}_{\log}) \wh{\otimes}_{\Z_p}{\A}^{\phi=p}_{\crr}){\otimes}{\Q_p}.
\end{align*}
are strict quasi-isomorphisms but this follows from Proposition \ref{acyclic-integral}.

\end{proof}
\subsubsection{Computation of syntomic cohomology}
\begin{corollary} \label{synt11}
Let $r\geq 0$.  The $r$-th cohomology of $\rg_{\synt}(\overline{X},\Z_p(r))_{\Q_p}$ is classical and 
there exists a natural topological  isomorphism
$$ H^r(\rg_{\synt}(\overline{X},\Z_p(r))_{\Q_p})\simeq      H^0_{\eet}(\overline{Y},W\Omega^r_{\log}){\otimes}{\Q_p}.
$$
\end{corollary}
\begin{proof}
 We note that, by Proposition \ref{lyon11},   there exist natural strict quasi-isomorphisms 
\begin{align*}
\oplus_{i\geq 0}H^0(X,\Omega^i)\wh{\otimes}_{\so_{K}}F^{r-i}\A_{\crr,K}[-i] & \stackrel{\sim}{\to} F^r(\R\Gamma_{\dr}(X)\wh{\otimes}_{\so_K}\A_{\crr,K}),\\
\oplus_{r-1\geq i\geq 0}H^0(X,\Omega^i)\wh{\otimes}_{\so_{K}}(\A_{\crr,K}/F^{r-i})[-i] & \stackrel{\sim}{\to} (\R\Gamma_{\dr}(X)\wh{\otimes}_{\so_K}\A_{\crr,K})/F^r.
\end{align*}
This, combined with
 the strict  quasi-isomorphisms    (\ref{formulas}),  changes the map $\gamma_{\hk}$ from (\ref{reduction}) into 
 $$\gamma_{\hk}^{\prime}: \quad (H^0_{\eet}({Y},W\Omega^{r-1}_{\log})
  \wh{\otimes}_{\Z_p}{\A}^{\phi=p}_{\crr}){\otimes}{\Q_p}\to 
(H^0(X,\Omega^{r-1})\wh{\otimes}_{\so_{K}}\so_C){\otimes}{\Q_p}.
$$
Hence we obtain the long exact sequence
\begin{align*}
(H^0_{\eet}({Y},W\Omega^{r-1}_{\log})
  \wh{\otimes}_{\Z_p}{\A}^{\phi=p}_{\crr}){\otimes}{\Q_p} & \lomapr{\gamma^{\prime}_{\hk}}
(H^0(X,\Omega^{r-1})\wh{\otimes}_{\so_{K}}\so_C){\otimes}{\Q_p}\to\\
 &  \wt{H}^r(\rg_{\synt}(\overline{X},\Z_p(r))_{\Q_p})\to H^0_{\eet}(\overline{Y},W\Omega^r_{\log}){\otimes}{\Q_p}  \to 0
\end{align*}

 It suffices to show  that  $\gamma^{\prime}_{\hk}$ is strictly surjective.   For that we will need to trace carefully its definition. Consider thus the following  commutative diagram
  \begin{equation}
  \label{quiet}
\xymatrix{(H^0_{\eet}({Y},W\Omega^{r-1}_{\log})
  \wh{\otimes}_{\Z_p}{\A}^{\phi=p}_{\crr})_{\Q_p}[-r+1]\ar[d]\ar[r]^-{\gamma^{\prime}_{\hk}} &
(H^0(X,\Omega^{r-1})\wh{\otimes}_{\so_{K}}\so_C)_{\Q_p}[-r+1]\ar[d]\\
 [(\R\Gamma_{\crr}(Y/\so_{F}^0)  \wh{\otimes}_{\so_{F}}\widehat{\A}_{\st})_{\Q_p}]^{N=0,\phi=p^r}\ar[d]^{\can}\ar[r]^-{\gamma_{\hk}}
    &  (\R\Gamma_{\dr}({X})\wh{\otimes}_{\so_K}{\A}_{\crr,K})_{\Q_p}/F^r\ar[d]^{\can}\\
     [\R\Gamma_{\crr}(Y/\so_{F}^0,F)  \wh{\otimes}^R_{{F}}\widehat{\B}^+_{\st}]^{N=0,\phi=p^r}\ar[r]^-{\gamma_{\hk}}
    &  (\R\Gamma_{\dr}({X_K})\wh{\otimes}^R_{K}{\B}^+_{\dr})/F^r.\\
  [\R\Gamma_{\hk}(Y)  \wh{\otimes}^R_{{F}}{\B}^+_{\st}]^{N=0,\phi=p^r}\ar[u]_{\wr}\ar[r]^-{\iota_{\hk}\otimes\iota} &
     (\R\Gamma_{\dr}({\wt{X}_K})\wh{\otimes}^R_{K}{\B}^+_{\dr})/F^r\ar[u]^{\iota^2_{\rig}}_{\wr}
     }
     \end{equation}
    Here, the fact that the bottom square commutes follows from the  proofs of Proposition \ref{fini2} and Theorem \ref{fini3}.
     Taking $\wt{H}^{r-1}$  of the above diagram we obtain the outer diagram in the commutative diagram
     $$
     \xymatrix{
     (H^0_{\eet}({Y},W\Omega^{r-1}_{\log})
  \wh{\otimes}_{\Z_p}{\A}^{\phi=p}_{\crr}){\otimes}{\Q_p}\ar[d]\ar[r]^-{\gamma^{\prime}_{\hk}} &
(H^0(X,\Omega^{r-1})\wh{\otimes}_{\so_{K}}\so_C){\otimes}{\Q_p}\ar@{-->}[d]\ar[dr]\\
     H^{r-1}_{\hk}(Y)\wh{\otimes}_F\B^{+,\phi=p}_{\crr}\ar[r]^-{\iota_{\hk}\otimes\iota} & H^{r-1}_{\dr}(\wt{X}_K)\wh{\otimes}_KC\ar@{^(->}[r] & H^0(\wt{X}_K,\Omega^{r-1})\wh{\otimes}_KC.
     }
$$
Since $d=0$ on $(H^0(X,\Omega^{r-1})\wh{\otimes}_{\so_{K}}\so_C){\otimes}{\Q_p}$, we get the shown factorization of the slanted map and the commutative square. 
This square is seen (by a careful chase of the diagram (\ref{quiet})) to 
be compatible with the projections $\theta: \A^{\phi=p}_{\crr}\to\so_C,$ $\theta: \B^{+,\phi=p}_{\crr}\to C$.  Using them, we obtain the commutative diagram
 $$\xymatrix{
  (H^0_{\eet}({Y},W\Omega^{r-1}_{\log})\wh{\otimes}_{\Z_p}\so_C){\otimes}{\Q_p}\ar[d]^{\can}\ar[r]^-{\gamma_{C}} &
(H^0(X,\Omega^{r-1})\wh{\otimes}_{\so_{K}}\so_C){\otimes}{\Q_p}\ar[d]^{\can}\\
H^{r-1}_{\hk}(Y)\wh{\otimes}_FC\ar[r]^{\iota_{\hk}\otimes\iota} &  H^{r-1}_{\dr}(\wt{X}_K)\wh{\otimes}_KC
}
$$
and reduce our problem to showing that the induced map $\gamma_C$ is surjective. 

 We will, in fact, show that $\gamma_C$ is an isomorphism. Note that the right vertical map in the above diagram is injective: use that the natural map $H^{r-1}_{\dr}({X})\to H^{r-1}_{\dr}({X}_K)$ is an injection (in particular the domain is torsion-fee). 
 By Remark \ref{noc1}, the above diagram yields that $\gamma_C=(\iota_{\hk}\otimes 1){\otimes}{\Q_p}$. Since the integral Hyodo-Kato map is a topological isomorphism
 $\iota_{\hk}: H^0_{\eet}({Y},W\Omega^{r-1}_{\log})\otimes_{\Z_p}\so_K\stackrel{\sim}{\to} H^0(X,\Omega^{r-1})$, so is $\gamma_C$, as wanted.
 
 To check that the map $\gamma^{\prime}_{\hk}$ is strict, consider the  factorization
$$
  \xymatrix{(H^0_{\eet}({Y},W\Omega^{r-1}_{\log})
  \wh{\otimes}_{\Z_p}{\A}^{\phi=p}_{\crr}){\otimes}{\Q_p}\ar@{->>}[d]^-{1\otimes\theta}\ar[rd]^-{\gamma^{\prime}_{\hk}} \\
     (H^0_{\eet}({Y},W\Omega^{r-1}_{\log})
  \wh{\otimes}_{\Z_p}\so_C){\otimes}{\Q_p}\ar[r]^-{\gamma_{C}}_{\sim} &
(H^0(X,\Omega^{r-1})\wh{\otimes}_{\so_{K}}\so_C){\otimes}{\Q_p}
}
$$
As mentioned above, the map $\gamma_{C}$ is a topological isomorphism. Hence  it suffices to show that the vertical map is strict. But this is clear because the surjection $\theta: {\A}^{\phi=p}_{\crr}\to \so_C$ has a $\Z_p$-linear continuous section. 

 \end{proof}
\subsection{Main theorem}
Before proving the main theorem of this section let us state the following  corollary. Recall that  ${X}$, resp. $\wt{X}$,  is  the standard formal, resp. weak formal, model  of the Drinfeld half-space ${\mathbb H}^d_K$, $Y=X_0$, $\overline{Y}=Y_{\overline{k}}$.  
\begin{corollary}
\label{zima5}Let  $r\geq 0$. The cohomology $ \wt{H}^r_{\eet}({X}_C,\Q_p(r))$ is classical and there is a natural topological isomorphism
$$
 H^r_{\eet}({X}_C,\Q_p(r))\simeq H^0_{\eet}(\overline{Y}, W\Omega^r_{\log}){\otimes}\Q_p.
$$
Moreover, there exists an \'etale regulator
$$
r_{\eet}: M(\sh^{r+1},\Z_p) \to H^r_{\eet}(X_C,\Q_p(r))
$$
compatible with the \'etale Chern classes and, via  the above isomorphism, with  the log de Rham-Witt Chern classes.
\end{corollary}
\begin{proof}The first claim  follows from the comparison between \'etale and syntomic cohomologies in Proposition \ref{fini11} and the computation of syntomic cohomology in Corollary \ref{synt11}.

  For the second claim,  consider  the diagram
  \begin{equation*}
  \xymatrix@C=60pt{
  M(\sh^{r+1},\Z_p)\ar@{-->}[r]^{r_{\eet}}  \ar[rd]^-{r_{\log}} & H^r_{\eet}(X_C,\Z_p(r))_{\Q_p}\ar[d]^{\wr}\\
  & H^r_{\eet}(\overline{Y},W\Omega\kr_{\log})_{\Q_p}.
  }
  \end{equation*}Define $r_{\eet}$ to make this diagram commute. It is compatible with the \'etale Chern classes because
 we have a compatibility of the \'etale and the de logarithmic Rham-Witt Chern classes.  This compatibility  follows from 
  the compatibility of the following Chern classes:  \'etale,   syntomic, crystalline,   classes  $c_1^{\st}$,  crystalline Hyodo-Kato, and logarithmic de Rham-Witt. These compatibilities are discussed in the Appendix. 
\end{proof}

 We are now ready  to prove the following result. 
\begin{theorem}\label{comp19}
Let  $r\geq 0$. 
\begin{enumerate}
\item 
 There is a natural topological isomorphism of $G\times \sg_K$-modules
$$
 H^r_{\eet}({X}_C,\Q_p(r))\simeq {\rm Sp}^{\cont}_r(\Q_p)^*\simeq  {\rm Sp}_r(\Z_p)^*{\otimes}\Q_p.
 $$
\item There are  natural topological isomorphisms of $G$-modules
\begin{align*}
 & H^r_{\dr}({X})\otimes_{\so_K}K\simeq {\rm Sp}^{\cont}_r(K)^*,\quad H^r_{\crr}(Y/\so^0_F)_{\Q_p}\simeq {\rm Sp}^{\cont}_r(F)^*,\\
& H^i_{\eet}(\overline{Y},W\Omega^r_{\overline{Y},\log})_{\Q_p}\simeq
\begin{cases}
{\rm Sp}^{\cont}_r(\Q_p)^* & \mbox{if } i=0,\\
0  & \mbox{if }i>0.
\end{cases}
\end{align*}
\item The regulator maps
\begin{align*}
 & r_{\eet}: M(\sh^{d+1},\Q_p)\to H^r_{\eet}({X}_C,\Q_p(r)),\quad 
  r_{\dr}: M(\sh^{d+1},K)\to H^r_{\dr}({X}){\otimes}_{\so_K}K,\\
 & r_{\hk}: M(\sh^{d+1},F)\to H^r_{\crr}(Y/\so^0_F)_{\Q_p},\quad 
  r_{\log}: M(\sh^{d+1},\Q_p)\to H^0_{\eet}(\overline{Y},W\Omega^r_{\log}){\otimes}{\Q_p}
\end{align*}
are strict surjective maps (of weak duals of Banach spaces), $G\times\sg_K$-equivariant, compatible with the isomorphisms in (1) and (2), and their kernels are equal to the space of degenerate measures (defined as the intersection of the space of measures with the set of degenerate distributions).
\item The natural map $$ H^r_{\eet}({X}_C,\Q_p(r)) \to H^r_{\proeet}({X}_C,\Q_p(r))$$ 
is an injection and identifies $ H^r_{\eet}({X}_C,\Q_p(r))$ with the $G$-bounded vectors of $ H^r_{\proeet}({X}_C,\Q_p(r))$. 
\end{enumerate}
\end{theorem}
\begin{proof}  
Point (2) follows from Theorem \ref{Steinberg} and Proposition \ref{lyon11}. Point (1) follows from Corollary \ref{zima5}, Lemma \ref{change-of-fields}, and the computation of the logarithmic de Rham-Witt cohomology in Theorem \ref{Steinberg}.

  To prove point (4) consider the commutative diagram, where the bottom sequence is strictly exact:
$$
\xymatrix{
 & & H^r_{\eet}({X}_C,\Q_p(r))\ar@{^(->}[d]^{\epsilon} \ar[r]^{\sim} &   {\rm Sp}^{\cont}_r(\Q_p)^*\ar@{^(->}[d]^{\can}\\
0\ar[r] & d\Omega^{r-1}(X_C)\ar[r] &   H^r_{\proeet}({X}_C,\Q_p(r))\ar[r] & {\rm Sp}_r(\Q_p)^*\ar[r] & 0
}
$$
Commutativity can be checked easily by looking at symbols. 
The change of topology map $\epsilon$ has image in $H^r_{\proeet}({X}_C,\Q_p(r))^{\text{$G$-bd}}$ (since  ${\rm Sp}^{\cont}_r(\Z_p)^*$ is compact).  We need to show that
this image is the whole of $H^r_{\proeet}({X}_C,\Q_p(r))^{\text{$G$-bd}}$. For that, since $ ({\rm Sp}_r(\Q_p)^*)^{\text{$G$-bd}}\simeq  {\rm Sp}^{\cont}_r(\Q_p)^*$, it suffices to show that 
$(d\Omega^{r-1}(X_C))^{\text{$G$-bd}}= 0$ or, equivalently,  that the map $(\Omega^{r-1}(X_C)^{d=0})^{\text{$G$-bd}}\to H^{r-1}_{\dr}(X_C)$ is an injection. 
It is enough to show  that the map $\Omega^{r-1}(X_C)^{\text{$G$-bd}}\to H^{r-1}_{\dr}(X_C)$ is injective
or that, by an analogous argument to the one we used in the proof of Proposition \ref{lyon11},  so is the  map $\Omega^{r-1}(X_K)^{\text{$G$-bd}}\to H^{r-1}_{\dr}(X_K)$.

 Now,  since, the  map 
$\Omega^{r-1}(X)\wh{\otimes}_{\so_K}K\to \Omega^{r-1}(X_K)^{\text{$G$-bd}} $ is an isomorphism (use the fact that $X$ can be covered by  $G$-translates of an open subscheme $U$ such that $U_K$ is an affinoid), it suffices to show  that the  map 
$\Omega^{r-1}(X)\wh{\otimes}_{\so_K}K\to H^{r-1}_{\dr}(X_K)$ is an injection. But this we have done in Proposition \ref{injectdR}.

Point (3) follows  from the construction and the Appendix. More precisely, the fact that the regulator maps $r_{\dr}, r_{\hk}, r_{\log}$ are compatible with the isomorphisms in point (2) follow from diagrams
  \ref{zima1}, \ref{zima2}, \ref{zima3}, respectively. Concerning the \'etale regulator, we claim that we have a commutative diagram
  \begin{equation*}
  \xymatrix@C=60pt{
  M(\sh^{r+1},\Z_p)\ar[r]^{r_{\eet}}  \ar[rd]^-{r_{\log}}\ar@{->>}[d]^-{\can} & H^r_{\eet}(X_C,\Z_p(r))_{\Q_p}\ar[d]^{\wr}\\
  {\rm Sp}_r(\Z_p)^* \ar@{-->}[r]^-{r_{\log}} \ar@{-->}[ru]_-{r_{\eet}}& H^r_{\eet}(\overline{Y},W\Omega\kr_{\log})_{\Q_p},
  }
  \end{equation*}
  where the vertical isomorphism is the one from Corollary \ref{synt11}. For the solid arrows this was shown in Corollary \ref{zima5}.  The bottom dashed regulator $r_{\log}$ was defined in diagram \ref{zima3}; it is a rational isomorphism. We define the dashed regulator $r_{\eet}$ to make this diagram commute; it is a rational isomorphism.

  Finally, the fact that the regulators in point (3) are strict follows
 from the fact that so are the corresponding maps $M(\sh^{r+1},\Z_p) \to {\rm Sp}_r(\Z_p)^*$, etc,  as surjective continuous maps of profinite modules, and tensoring with $\Q_p$ is right exact.
\end{proof}
   \appendix
   \section{Symbols}
   \label{symbols}
   We gather in this appendix a few basic facts concerning symbol maps and their compatibilities that we need in this paper. 
   We use the notation from Chapter \ref{etale}.
    \subsection{Hyodo-Kato isomorphisms}\label{HK-Lyon}
    Let $X$ be a semistable Stein weak formal scheme over $\so_K$.  In the first part of this paper we have used  the Hyodo-Kato isomorphism
  as   defined by Grosse-Kl\"onne in \cite{GK2}, $ \iota_{\hk}: H^r_{\hk}(X_0)\otimes_FK\stackrel{\sim}{\to} H^r_{\dr}(X_K)$. 
  But one can  use the  original Hyodo-Kato isomorphism defined for quasi-compact formal schemes in \cite{HK}. Doing that we obtain two 
    Hyodo-Kato isomorphisms. One that, modulo canonical identifications, turns out to be  the same as the one of Grosse-Kl\"onne, the other identifies bounded Hyodo-Kato and de Rham cohomologies.
    \begin{proposition}
    \label{Fontaine-Messing}
We have  compatible Hyodo-Kato (topological) isomorphisms 
    \begin{align}
    \label{rain12}
   \quad \iota_{\hk}: H^r_{\crr}(X_0/\so^0_F,F){\otimes}_FK\stackrel{\sim}{\to} H^r_{\dr}(\wh{X}_K),\quad
    \iota_{\hk}: (H^r_{\crr}(X_0/\so^0_F)\otimes_{\so_F}\so_K){\otimes}{K}\stackrel{\sim}{\to} H^r_{\dr}(\wh{X}){\otimes}{K}.           
 \end{align}
\end{proposition}
\begin{proof}
As mentioned above they are induced by  the original Hyodo-Kato isomorphism \cite{HK}. We will  describe them in more detail.

   To start, assume that we have a quasi-compact semistable formal scheme $Y$ over $\so_K$. We will work in the classical derived category. 
   Recall  that  the Frobenius
   $$
   r^{\rm PD}_{\varpi,n,\phi}\otimes^L_{r^{\rm PD}_{\varpi,n}}\R\Gamma_{\crr}(Y_1/r^{\rm PD}_{\varpi,n})\to \rg_{\crr}(Y_1/r^{\rm PD}_{\varpi,n}),\quad 
   \so_{F,n,\phi}\otimes^L_{\so_{F,n}}\rg_{\crr}(Y_0/\so_{F,n})\to \rg_{\crr}(Y_0/\so_{F,n}^0)
   $$
   has a $p^N$-inverse, for  $N=N(d)$, $d=\dim Y_0$. This is proved in \cite[2.24]{HK}.
   Recall also that   the projection $p_0:  \rg_{\crr}(Y/r^{\rm PD}_{\varpi,n})\to \rg_{\crr}(Y_0/\so_{F,n}^0)$
has a functorial (for maps between formal schemes and a change of $n$) and Frobenius-equivariant $p^{N_{\iota}}$-section
   $$\iota_{\varpi}: 
   \rg_{\crr}(Y_0/\so_{F,n}^0)\to  \rg_{\crr}(Y/r^{\rm PD}_{\varpi,n}),
    $$
 i.e., $p_0\iota_{\varpi}=p^{N_{\iota}}$, $N_{\iota}=N(d)$. This  follows easily from the proof of Proposition 4.13 in \cite{HK}; 
 the key point being that the Frobenius on $\rg_{\crr}(Y_0/\so_{F,n}^0)$ is close to a quasi-isomorphism and the Frobenius on the ${\rm PD}$-ideal of $r^{\rm PD}_{\varpi}$ is close to zero. Moreover, the resulting map 
 $$\iota_{\varpi}: 
   \rg_{\crr}(Y_0/\so_{F,n}^0)\otimes^L_{\so_{F,n}}r^{\rm PD}_{\varpi,n} \to  \rg_{\crr}(Y/r^{\rm PD}_{\varpi,n})
    $$
 is a $p^N$-quasi-isomorphism, $N=N(d)$,  \cite[Lemma 5.2]{HK} and so is the composite
 $$
 p_{\varpi} \iota_{\varpi}: \rg_{\crr}(Y_0/\so_{F,n}^0)\otimes^L_{\so_{F,n}}\so_{K,n} \to  \rg_{\crr}(Y/\so^{\times}_{K,n}).
 $$
 Taking $\holim_n$ of the last map we obtain a map 
 $$
  p_{\varpi} \iota_{\varpi}: \rg_{\crr}(Y_0/\so_{F}^0)\otimes^L_{\so_{F}}\so_{K} \to  \rg_{\crr}(Y/\so^{\times}_{K})
 $$
 that  is a $p^N$-quasi-isomorphism, $N=N(d)$.
 The twisted Hyodo-Kato map is defined as 
 $\wt{\iota}_{\hk}=\rho^{-1} p_{\varpi}\iota_{\varpi}$.  
 We have the  commutative diagram
 $$
 \xymatrix{\rg_{\crr}(Y/r^{\rm PD}_{\varpi})\ar[d]^{p_0}\ar[r]^-{p_{\varpi}} & \rg_{\crr}(Y/\so_K^{\times}) & \rg_{\dr}(Y)\ar[l]_-{\rho}^-{\sim}\\
 \rg_{\crr}(Y_0/\so_{F}^0)\ar@/^40pt/[u]^-{\iota_{\varpi}}\ar@/_10pt/[rru]^{\wt{\iota}_{\hk}}
 }
 $$
 We note that the map 
 $\wt{\iota}_{\hk}:\rg_{\crr}(Y_0/\so_{F}^0)\otimes_{\so_F}\so_K\to \rg_{\dr}(Y)$ 
 is a  $p^N$-quasi-isomorphism, $N=N(d)$. Modulo $p^N, N=N(d)$, it is independent of the choices made. The Hyodo-Kato map 
 $$
 {\iota}_{\hk}:  H^i_{\crr}(Y_0/\so_{F}^0){\otimes}F\to H^i_{\dr}(Y){\otimes}K
 $$
 is defined as ${\iota}_{\hk}:=p^{-N_{\iota}}\wt{\iota}_{\hk}$.  We have obtained the Hyodo-Kato (topological) isomorphism
 $$
  {\iota}_{\hk}:  (H^i_{\crr}(Y_0/\so_{F}^0)\otimes_{\so_F}\so_K){\otimes}K\stackrel{\sim}{\to} H^i_{\dr}(Y){\otimes}K.
 $$

   For a Stein semistable weak formal scheme ${X}$, we choose a Stein covering $\{U_s\}$, $s\in\N$, and define the compatible Hyodo-Kato maps
   \begin{align}
   \label{rain13}
  &  \iota_{\hk}:=\varprojlim_s\iota_{\hk,U_s}: H^i_{\crr}(X_0/\so_{F}^0,F)\to H^i_{\dr}(\wh{X}_K),\\
    & \iota_{\hk}:=p^{-N_{\iota}}(\varprojlim_s\wt{\iota}_{\hk,U_s}{\otimes}{\Q_p}): H^i_{\crr}(X_0/\so_{F}^0){\otimes}{\Q_p}\to H^i_{\dr}(\wh{X}){\otimes}{\Q_p}.\notag
   \end{align}
   We used here the fact that (topologically)
   \begin{align}
   \label{facts1}
    & H^i_{\crr}(X_0/\so_{F}^0,F)\simeq \varprojlim_sH^{i}_{\crr}(U_{s}/\so_{F}^0,F),\quad  H^i_{\crr}(X_0/\so_{F}^0)\simeq 
    \varprojlim_sH^{i}_{\crr}(U_{s}/\so_{F}^0), \\
    &  H^i_{\dr}(\wh{X}_K)\simeq \varprojlim_sH^{i}_{\dr}(]U_s[_{\wh{X}_K}), \quad
  H^i_{\dr}(\wh{X})_{\Q_p}\simeq  (\varprojlim_sH^{i}_{\dr}(\wh{X}|_{U_s}))_{\Q_p}.\notag
   \end{align}
   The third isomorphism, since $X_K$ is Stein,  is standard. The first two follow from the vanishing of the derived functors
   $$
   H^1\holim_sH^{i-1}_{\crr}(U_{s}/\so_{F}^0,F),\quad  H^1\holim_sH^{i-1}_{\crr}(U_{s}/\so_{F}^0),
   $$ which, in turn, follow
   from the fact that the pro-systems ($s\in\N$)
   \begin{align*}
     \{H^{i-1}_{\crr}(U_{s}/\so_{F}^0,F)\}=\{H^{i-1}_{\crr}(Y_{s}/\so_{F}^0,F)\},\quad  \{H^{i-1}_{\crr}(U_{s}/\so_{F}^0)\}=\{H^{i-1}_{\crr}(Y_{s}/\so_{F}^0)\}
     \end{align*}
         are Mittag-Leffler. 
       To show the last isomorphism in (\ref{facts1}), it suffices to show the
       vanishining of $ (H^1\holim_sH^{i-1}_{\dr}(\wh{X}|_{U_s}))_{\Q_p}$. For that, we will  use the fact that the twisted Hyodo-Kato map
       $$
     \wt{\iota}_{\hk}:   H^{i-1}_{\crr}(U_{s}/\so_{F}^{0})\otimes_{\so_F}\so_K\to H^{i-1}_{\dr}(\wh{X}|_{U_s})
       $$
       is a $p^N$-isomorphism. Which implies that so is the induced  map 
      $$
    H^1\holim_s(H^{i-1}_{\crr}(U_s/\so_F^0)\otimes_{\so_F}\so_K)\to 
   H^1\holim_sH^{i-1}_{\dr}(\wh{X}|_{U_s}).
   $$ But, since  the pro-system 
   $  \{H^{i-1}_{\crr}(U_{s}/\so_{F}^0)\}=\{H^{i-1}_{\crr}(Y_{s}/\so_{F}^0)\} $
   is Mittag-Leffler, the first derived limit made rational is trivial, as wanted.  
   
  Now, by definition, the Hyodo-Kato maps from (\ref{rain13}) are topological isomorphisms.
 \end{proof}
 \begin{corollary}
  \label{almost-there}
 The Hyodo-Kato isomorphisms from Proposition \ref{Fontaine-Messing} are compatible with the Grosse-Kl\"onne Hyodo-Kato isomorphism. That is, we have a commutative diagram
 $$
 \xymatrix{
 H^i_{\rig}(X_0/\so_F^0)  \ar[d]^{\wr}\ar[r]^-{\iota_{\hk}}  &  H^i_{\dr}(X_K)\ar[d]^{\wr}\\
 H^i_{\crr}(X_0/\so_F^0,F)\ar[r]^{\iota_{\hk}} & H^i_{\dr}(\wh{X}_K)\\
 H^i_{\crr}(X_0/\so_F^0)_{\Q_p}\ar[u]\ar[r]^{\iota_{\hk}}  & H^i_{\dr}(\wh{X})_{\Q_p}\ar[u]
 }
 $$
 \end{corollary}
 \begin{proof}We may argue locally and assume that $X$ is quasi-compact. Then, 
 this can be checked  by  the commutative diagram (we use   the notation from Section \ref{compmorph})
     \begin{equation}
     \label{symbols-diag}
\xymatrix@C=30pt{ H^i_{\rig}(X_0/\so_F^0)    \ar[ddd]\ar[rr]^-{\iota_{\hk}}   &  &  H^i_{\dr}(X_K)\ar[ddd]\\
 &     
 H^i_{\rig}(\overline{X}_0/r^{\dagger})\ar[lu]_{p_0}^{\sim}\ar[d]^{f_1}     \ar[ru]^-{p_{\varpi}}     \\
& 
 H^i_{\rig}({X}_0/r^{\dagger}) \ar[luu]^{p_0}\ar[d]^{f_2}  \ar[ruu]_-{p_{\varpi}}  \\
H^i_{\crr}(X_0/\so_F^0,F)\ar@/^20pt/[r]^{p^{-N_{\iota}}\iota_{\varpi}}\ar@/_20pt/[rr]^{\iota_{\hk}}  &  H^i_{\crr}(X/r^{\rm PD}_{\varpi},\Q_p)\ar[l]^{p_0} \ar[r]^-{p_{\varpi}}& H^i_{\dr}(\wh{X}_K)\\
H^i_{\crr}(X_0/\so_F^0)_{\Q_p}\ar[u]\ar@/^20pt/[r]_{p^{-N_{\iota}}\iota_{\varpi}}\ar@/_20pt/[rr]^{\iota_{\hk}}  &  H^i_{\crr}(X/r^{\rm PD}_{\varpi})_{\Q_p}\ar[u]\ar[l]^{p_0} \ar[r]^-{p_{\varpi}}& H^i_{\dr}(\wh{X})_{\Q_p}\ar[u]
}
\end{equation}
\end{proof}

   \subsection{Definition of symbols}
   We define now various symbol maps and show that they are compatible.
   
   Let $X$ be a semistable formal scheme over $\so_K$. Let $M$ be the sheaf of monoids on $X$ defining the log-structure, $M^{\gp}$ its group completion. This log-structure is canonical, in the terminology of Berkovich \cite[2.3]{BerL}, i.e., $M(U)=\{x\in\so_X(U)| x_K\in\so^{*}_{X_K}(U_K)\}$. This is shown in \cite[Theorem 2.3.1]{BerL}, \cite[Theorem 5.3]{BerN} and applies also to semistable formal schemes with self-intersections. It follows that  $M^{\gp}(U)=\so(U_K)^*$. 
      \subsubsection{Differential symbols}
      \label{diffsymbols}

     We have the  crystalline first Chern class maps of complexes of sheaves on $X_{\eet}$ 
     \cite[2.2.3]{Ts}
\begin{align*}
 c_1^{\st}:   M^{\gp}\to M^{\gp}_n\to R\varepsilon_*\sj^{[1]}_{X_n/r^{\rm PD}_{\varpi,n}}[1],\quad
  c_1^{\hk}:  M^{\gp}\to M^{\gp}_0\to R\varepsilon_*\sj^{[1]}_{X_0/\so_{F,n}^0}[1].
  \end{align*}
Here, the map $\varepsilon$ is the projection from the corresponding crystalline-\'etale  site to the \'etale site. These maps  are clearly compatible. 
We get the induced  functorial maps
\begin{align*}
   \quad c_1^{\st}:H^0(X_K,\so_{X_K}^*)  \to \R\Gamma_{\crr}(X/r^{\rm PD}_{\varpi},
 \sj^{[1]})[1],\quad
  c_1^{\hk}:H^0(X_K,\so_{X_K}^*) \to \R\Gamma_{\crr}(X_0/\so_{F}^0,\sj^{[1]})[1].
\end{align*}
The Hyodo-Kato classes above can be also defined using the de Rham-Witt complexes. That is, one can define (compatible) Hyodo-Kato Chern class maps \cite[2.1]{des}
   $$
     c_1^{\log}:H^0(X_K,\so_{X_K}^*) \to \R\Gamma_{\eet}(X_{0},W\Omega\kr_{X_0,\log})[1],\quad 
 c_1^{\hk}:H^0(X_K,\so_{X_K}^*) \to \R\Gamma_{\eet}(X_{0},W\Omega\kr_{X_0})[1].
$$
They are compatible with the classical crystalline Hyodo-Kato Chern class maps above (use \cite[I.5]{Gro} and replace $\so^*$ by $M^{\gp}$).

 We also have the  de Rham first Chern class  map $$c_1^{\dr}: M^{\gp}\to M^{\gp}_n\lomapr{\dlog} \Omega^{\scriptscriptstyle\bullet}_{X_n/\so^{\times}_{K,n}} [1].$$
 It induces  the functorial map
 $$c_1^{\dr}:H^0(X_K,\so^*)  \to \R\Gamma_{\dr}(X )[1].
 $$
 It is evident that, by the canonical map $\R\Gamma_{\dr}(X)\to\R\Gamma_{\dr}(X_K)$,
  this map is compatible with  the rigid analytic class (defined using $\dlog$)
 $c_1^{\dr}:H^0(X_K,\so^*)  \to \R\Gamma_{\dr}(X_K )[1]$. 

  Let now $X$ be a semistable weak formal scheme over $\so_K$. The overconvergent classes 
  $$
   c_1^{\hk}:H^0(X_K,\so^*) \to \R\Gamma_{\rig}(X_0/\so_{F}^0)[1],\quad c_1^{\dr}:H^0(X_K,\so^*)  \to \R\Gamma_{\dr}(X _K)[1]
  $$
  are defined in an analogous way to the crystalline Hyodo-Kato classes and the rigid analytic de Rham classes (of $\wh{X}_K$), respectively. Clearly they are compatible with those. 
  \begin{lemma}
 \label{compatibility}Let  $X$ be a semistable Stein weak formal scheme over $\so_K$.
 The Hyodo-Kato maps 
 \begin{align*}
  & \iota_{\hk}: H^1_{\crr}(X_0/\so_F^0)_{\Q_p}{\to} H^1_{\dr}(\wh{X})_{\Q_p},\quad  \iota_{\hk}: H^1_{\crr}(X_0/\so_F^0,F){\to} H^1_{\dr}(\wh{X}_K), \\
  & \iota_{\hk}: H^1_{\rig}(X_0/\so_F^0){\to} H^1_{\dr}({X}_K)
  \end{align*}
 are  compatible with   the  Chern class maps from $H^0(\wh{X}_K,\so^*)$ and $H^0(X_K,\so^*)$. 
   \end{lemma} 
  \begin{proof}For the first two maps, note that we can assume that $X$ is quasi-compact. This is because in the second  map the cohomologies are projective limits and in  the first map this is true up to a universal constant (see the proof of Proposition \ref{Fontaine-Messing}). 
  Now, the proof of an analogous lemma in the algebraic setting goes through with only small changes \cite[Lemma 5.1]{NN}. We will present it for the second map (the proof for the first map is similar with a careful bookkeeping of the appearing constants). 
  Since $\iota_{\hk}=\rho^{-1}p_{\varpi}\iota^{\prime}_{\varpi}$, $\iota^{\prime}_{\varpi}=p^{-N_{\iota}}\iota_{\varpi}$,  and the map $p_{\varpi}$ is compatible with first Chern classes, it suffices to show the compatibility for the section  $\iota^{\prime}_{\varpi}$. Let 
  $x\in
  H^0(\wh{X}_K,\so^*_{\wh{X}_K})$. Since the map $\iota^{\prime}_{\varpi}$ is a section of the map $p_0$ and the  map $p_0$ is compatible with first Chern classes, we have that the element  $\zeta\in H^1_{\crr}(X/r^{\rm PD}_{\varpi})$ defined as  $\zeta=\iota^{\prime}_{\varpi}(c_1^{\hk}(x))-c_1^{\st}(x)\in \ker p_0$. Since the map 
  \begin{equation}
  \label{map-app}
  \beta=\iota^{\prime}_{\varpi}\otimes\id:H^1_{\crr}(X_0/\so_F^0)_{\Q_p}\wh{\otimes}_{F}r^{\rm PD}_{\varpi,\Q_p}\to H^1_{\crr}(X/r^{\rm PD}_{\varpi})_{\Q_p}
  \end{equation}
   is surjective (see Section \ref{HK-Lyon}), we can write $\zeta=\beta(\zeta^{\prime})$ for $\zeta^{\prime}\in H^1_{\crr}(X_0/\so_F^0)_{\Q_p}\wh{\otimes}_{F}r^{\rm PD}_{\varpi,\Q_p}$. Since $p_0\beta(\zeta^{\prime})=0$, we have $\zeta^{\prime}\in \ker (\id\otimes p_0)$. Hence $\zeta^{\prime}=T\gamma$, $\gamma\in H^1_{\crr}(X_0/\so_F^0)_{{\mathbf Q}_p}\wh{\otimes}_{F}r^{\rm PD}_{\varpi,\Q_p}. $
  
  Since the map $\iota^{\prime}_{\varpi}$ is compatible with Frobenius,  $\phi(c_1^{\hk}(x))=pc_1^{\hk}(x)$, $\phi(c_1^{\st}(x))=pc_1^{\st}(x)$, and the map $\beta$ from (\ref{map-app}) is injective, we have $\phi(\zeta^{\prime})=p\zeta^{\prime}$. Since $\phi(T\gamma)=T^p\phi(\gamma)$ this implies that $\gamma\in\bigcap_{n=1}^{\infty}H^1_{\crr}(X_0/\so_F^0)_{{\mathbf Q}_p}\otimes_{F}T^nr^{\rm PD}_{\varpi,\Q_p}$, which  is not possible unless $\gamma$ (and hence $\zeta^{\prime}$) are zero. This implies that $\zeta=0$ and this is what we wanted to show.
  
  For the last map in the lemma, we use the commutative diagram from Corollary \ref{almost-there}
  $$
  \xymatrix{
  H^i_{\rig}(X_0/\so_F^0)\ar[r]^-{\iota_{\hk}} \ar[d]^{\wr}& H^i_{\dr}(X_K)\ar[d]^{\id}_{\wr}\\
 H^i_{\crr}(X_0/\so_F^0,F) \ar[r]^-{\iota_{\hk}}& H^i_{\dr}(\wh{X}_K),
  }
  $$
  the compatibilities from Section \ref{diffsymbols}, and the compatibility of the second map in this lemma with Chern classes that we have shown above. 
  \end{proof}
 \subsubsection{\'Etale symbols and the period map} \label{etale-symbols}
 We have the \'etale first Chern class maps (obtained from Kummer theory)
 $$
 c_1^{\eet}: \so^*_{X_K}\to \Z/p^n(1)[1],\quad c_1^{\eet}: H^0({X}_K,\so^*)\to \R\Gamma_{\eet}(X_K,\Z/p^n(1))[1].
 $$
 We also have the syntomic first Chern class maps \cite[2.2.3]{Ts}
 $$
 c_1^{\synt}: H^0({X}_K,\so^*)\to \R\Gamma_{\synt}(\overline{X},\Z_p(1))[1]
 $$
that are compatible with the crystalline first Chern class maps
 $
  c_1^{\crr}: H^0({X}_K,\so^*)\to \R\Gamma_{\crr}(\overline{X})[1].
 $
By \cite[3.2]{Ts}, they are also compatible with the \'etale Chern classes via  the Fontaine-Messing period maps $\alpha_{\rm FM}$.
 \section{Alternative proof of Corollary \ref{amtrak}}
 We present in this appendix an alternative proof of Corollary \ref{amtrak} (hence also of Proposition \ref{lyon11} which easily follows  from it) that does not use the ordinarity of the truncated log-schemes $Y_s$. 
 
  Let $X/k^0$ be  a fine log-scheme  of Cartier type. Recall that we have the subsheaves 
                                 $$Z_{\infty}^j=\cap_{n\geq 0} Z_n^j, \quad B_{\infty}^j=\cup_{n\geq 0} B_n^j$$
                                 of
      $\Omega^j=\Omega^j_{X/k^0}$ (in what follows we will omit the subscripts in differentials if understood). Via the maps $C: Z_{n+1}^j\to Z_n^j$ (with kernels $B_n^j$), 
     $Z_{\infty}^j$ is the sheaf of forms $\omega$ such that 
      $dC^n \omega=0$ for all $n$. This sheaf is acted upon by the Cartier operator $C$,
      and we recover 
      $$B_{\infty}^j=\cup_{n\geq 0} (Z_{\infty}^j)^{C^n=0}, \quad \Omega^j_{\log}=(Z_{\infty}^j)^{C=1}.$$
      The following result is proved in \cite[2.5.3]{Ill} for classically smooth schemes. It holds most likely in much greater generality than the one stated below, but this will be sufficient for our purposes.
      
      \begin{lemma} \label{Raynaud}
        Assume that $X/k^0$ is  semi-stable (with the induced log structure) and that $k$ is algebraically closed. Then the natural map of \'etale sheaves 
         $$ B_{\infty}^j\oplus (\Omega^j_{\log}\otimes_{\mathbf{F}_p} k)\to Z_{\infty}^j$$
         is an isomorphism.
      \end{lemma}
      
      \begin{proof} It suffices to show that, for $X$ as above and affine,
         the map $B_{\infty}^j(X)\oplus (\Omega^j_{\log}(X)\otimes_{\mathbf{F}_p} k)\to 
      Z_{\infty}^j(X)$ is an isomorphism. Take an open dense subset $j: U\hookrightarrow  X$ which is smooth over $k$.        
         Then $\Omega^i_{X}$ is a subsheaf of $j_*\Omega^i_U$ and so 
      $Z^i_{\infty,X}$ is a subsheaf of $j_*Z^i_{\infty, U}$, giving an inclusion 
      $Z^i_{\infty, X}(X)\subset Z^i_{\infty, U}(U)$. 
     By a result of Raynaud \cite[Prop. 2.5.2]{Ill}, $Z^i_{\infty, U}(U)$ is a union of finite dimensional $k$-vector spaces stable under $C$. We deduce that 
      $Z^i_{\infty, X}(X)$ is also such a union. 
        
        The result follows now from the following basic result of semi-linear algebra (this is where
      the hypothesis that $k$ is algebraically closed is crucial): if $E$ is a finite dimensional $k$-vector space stable under 
      $C$, then $E=E_{\rm nilp}\oplus E_{\rm inv}$, where $E_{\rm nilp}=\cup_{n} E^{C^n=0}$, 
      $E_{\rm inv}=\cap_{n} C^n(E)$, and the natural map $E^{C=1}\otimes_{\mathbf{F}_p} k\to E_{\rm inv}$ is an isomorphism.
                  \end{proof}

   \begin{proof}(of Corollary \ref{amtrak}) 
   (1) We prove this in several steps. 
  We start with the case $i=0$ (the most delicate). 
   By Lemma \ref{Raynaud}, 
        $$H^0_{\eet}(\overline{Y}, Z_{\infty}^j)\simeq H^0_{\eet}(\overline{Y}, B_{\infty}^j)\oplus H^0_{\eet}(\overline{Y}, \Omega^j_{\log}\otimes_{\mathbf{F}_p} k).$$ We need the following intermediate result:
        
        \begin{lemma} \label{tricky vanishing}
         We have $H^0_{\eet}(\overline{Y}, B_{\infty}^j)=0$.
        \end{lemma}
        \begin{proof}We note  that 
         $H^0_{\eet}(\overline{Y}, B_n^j)=0$ for all $n$: because we have $B^j_n\simeq B^j_{n+1}/B^j_1$ this follows from Lemma \ref{final1}.
        This however does not allow us to deduce  formally our lemma because  $B_{\infty}^j=\cup_{n} B_n^j$ and $\overline{Y}$ is not quasi-compact. Instead, we argue as follows: the formation of the sheaves $B_{\infty}^j$ being functorial, we have a natural map
     $\alpha: H^0_{\eet}(\overline{Y}, B_{\infty}^j)\to \prod_{C\in F^0} H^0_{\eet}(\overline{C}, B_{\infty}^j)$. It suffices to prove that 
     $\alpha$ is injective and that $H^0_{\eet}(\overline{T}, B_{\infty}^j)=0$. To prove the injectivity of 
     $\alpha$, it suffices to embed both the domain and target of $\alpha$ in 
     $H^0_{\eet}(\overline{Y}, \Omega^j)$ and $\prod_{C} H^0_{\eet}(\overline{C}, \Omega^j)$, and to use the injectivity of the natural map
     $H^0_{\eet}(\overline{Y}, \Omega^j)\to \prod_{C} H^0_{\eet}(\overline{C}, \Omega^j)$.     
     Next, since $\overline{T}$ is quasi-compact, 
     $$H^0_{\eet}(\overline{T}, B_{\infty}^j)=\varinjlim_{n} H^0_{\eet}(\overline{T}, B_n^j)=0,$$
     the second equality being a consequence of Proposition \ref{GK5} and Lemma \ref{final1} 
(as above, in the case of~$\overline{Y}$). 
             \end{proof}
        
        Consider now the sequence of natural maps
      $$  H^0_{\eet}(\overline{Y},
  \Omega^j_{\log})\widehat{\otimes}_{{\mathbf F}_p}\overline{k}\stackrel{\sim}{\to} H^0_{\eet}(\overline{Y}, \Omega^j_{\log}\otimes_{\mathbf{F}_p} \overline{k})\stackrel{\sim}{\to} H^0_{\eet}(\overline{Y}, Z_{\infty}^j) \to H^0_{\eet}(\overline{Y}, \Omega^j).$$
  The first map is clearly a topological  isomorphism, the second one is a topological  isomorphism by Lemma \ref{tricky vanishing}. Hence it remains to show that the last map is a topological  isomorphism as well. Or that all the natural maps $ H^0_{\eet}(\overline{Y},Z^j_{n})\to  H^0_{\eet}(\overline{Y},\Omega^j)
$, $n\geq 1$, are topological  isomorphisms. But this was done in the proof of Lemma \ref{degenerate}.
        This gives the desired result for $i=0$.
  
     We prove next the result for $i>0$, i.e., that  $H^i_{\eet}(\overline{Y}, \Omega^j_{\log})\wh{\otimes}_{\mathbf{F}_p}\overline{k}=0$ for $i\geq 1$. We start with showing that $H^i_{\eet}(\overline{Y}, \Omega^j_{\log})=0$.
The  exact sequence 
    $$0\to \Omega^j_{\log}\to \Omega^j/B_1^j\verylomapr{1-C^{-1}} \Omega^j/B_2^j\to 0$$
   yield the exact sequence 
   $$0\to H^0_{\eet}(\overline{Y}, \Omega^j_{\log})\to H^0_{\eet}(\overline{Y}, \Omega^j)\verylomapr{1-C^{-1}} H^0_{\eet}(\overline{Y}, \Omega^j)
   \to H^1_{\eet}(\overline{Y}, \Omega^j_{\log})\to 0$$
   and $H^i_{\eet}(\overline{Y}, \Omega^j_{\log})=0$ for $i>1$. It suffices therefore to prove that $1-C^{-1}$ is surjective on 
   $H^0_{\eet}(\overline{Y}, \Omega^j)$. For this, write 
   $A_s=H^0_{\eet}(\overline{Y}_s^{\circ}, \Omega^j_{\log})$. As we have already seen, we have an 
   isomorphism $$
   H^0_{\eet}(\overline{Y}, \Omega^j)\simeq \varprojlim_s (A_s\otimes_{\mathbf{F}_p} \overline{k})=H^0_{\eet}(\overline{Y},
  \Omega^j_{\log})\widehat{\otimes}_{{\mathbf F}_p}\overline{k}.$$
    We have 
  $C^{-1}=\varprojlim_{s}(1\otimes \varphi)$ ($\varphi$ being the
  absolute Frobenius on $\overline{k}$; note that $C-1=0$ on $A_s$). To conclude that $1-C^{-1}$ is surjective on 
   $H^0_{\eet}(\overline{Y}, \Omega^j)$, it suffices to pass to the limit in the exact sequences 
   $$0\to A_s\to A_s\otimes_{\mathbf{F}_p} \overline{k}\lomapr{1-\phi} A_s\otimes_{\mathbf{F}_p} \overline{k}\to 0,$$
   whose exactness is ensured by the Artin-Schreier sequence for $\overline{k}$ and the fact that
   $(A_s)_{s}$ is Mittag-Leffler.
   
    This shows that $H^i_{\eet}(\overline{Y},
  \Omega^j_{\log})=0$ for $i>0$. Choosing a basis $(e_{\lambda})_{\lambda\in I}$ of $\overline{k}$ over $\mathbf{F}_p$ we obtain an embedding
  $$ H^i_{\eet}(\overline{Y},
  \Omega^j_{\log})\widehat{\otimes}_{{\mathbf F}_p}\overline{k}\subset \prod_{\lambda\in I} H^i_{\eet}(\overline{Y},
  \Omega^j_{\log})=0,$$
  which finishes the proof of (1).

        (2) We prove the claim for $W_n$  by induction on $n$, 
        the case 
        $n=1$ being part (1). We pass from
         $n$ to $n+1$ using the strictly exact sequences 
                 \begin{align*}
                 & 0\to H^0(\overline{Y}, \Omega^j)\stackrel{V^n}{\to} H^0(\overline{Y}, W_{n+1}\Omega^j)\to H^0(\overline{Y}, W_n\Omega^j)\to 0,\\
       & 0\to H^0_{\eet}(\overline{Y},
  \Omega^j_{\log})\widehat{\otimes}_{{\mathbf F}_p}\overline{k}\stackrel{V^n}{\to }H^0_{\eet}(\overline{Y},W_{n+1}\Omega^j_{\log})\widehat{\otimes}_{\Z/p^{n+1}}W_{n+1}(\overline{k})\to  H^0_{\eet}(\overline{Y},W_{n}\Omega^j_{\log})\widehat{\otimes}_{\Z/p^n}W_n(\overline{k})\to 0,
  \end{align*}
  as well as the natural map between them. The first sequence is exact by Lemma \ref{degenerate}.                      
 To show that 
 the second sequence is exact, consider, 
  as above, the exact sequences 
  $$0\to H^0_{\eet}(\overline{Y}_s^{\circ}, \Omega^j_{\log})\stackrel{V^n}{\to} H^0_{\eet}(\overline{Y}_s^{\circ}, W_{n+1}\Omega^j_{\log})\to 
  H^0_{\eet}(\overline{Y}_s^{\circ}, W_n\Omega^j_{\log})\to H^1_{\eet}(\overline{Y}_s^{\circ}, \Omega^j_{\log}).$$
  Tensoring over $\mathbf{Z}$ by $W_{n+1}(\overline{k})$, we can rewrite them as
  \begin{align*}
  0\to  H^0_{\eet}(\overline{Y}_s^{\circ}, \Omega^j_{\log})\otimes_{\mathbf{F}_p} \overline{k} & \to 
  H^0_{\eet}(\overline{Y}_s^{\circ}, W_{n+1}\Omega^j_{\log})\otimes_{\mathbf{Z}/p^{n+1}} W_{n+1}(\overline{k})\to 
H^0_{\eet}(\overline{Y}_s^{\circ}, W_n\Omega^j_{\log})\otimes_{\mathbf{Z}/p^n} W_n(\overline{k})\\
 & \to H^1_{\eet}(\overline{Y}_s^{\circ}, \Omega^j_{\log})\otimes_{\mathbf{F}_p} \overline{k}.
 \end{align*}
  To lighten the notation, write them simply as
  $0\to A_s\to B_s\to C_s\to D_s$. Using that $(A_s)_s, (C_s)_s$ are finite $W_n(\overline{k})$-modules and that 
  $\varprojlim_s D_s= H^1_{\eet}(\overline{Y},
  \Omega^j_{\log})\widehat{\otimes}_{{\mathbf F}_p}\overline{k}=0$ (as follows from (1)), we obtain the exact sequence
  $$0\to \varprojlim_s A_s\to \varprojlim_s B_s\to \varprojlim_s C_s\to 0,$$
  which finishes the proof of (2) for $W_n$, $n\geq 1$. Passing to the limit over $n$ gives us the proof for $W$. 
     \end{proof}

\end{document}